\documentclass[a4paper]{amsart}
\usepackage{hyperref}
\usepackage{amsmath,amssymb,amsfonts,amsthm}
\usepackage[francais,english]{babel}
\usepackage{dsfont}
\allowdisplaybreaks
\renewcommand{\leq}{\leqslant}
\renewcommand{\geq}{\geqslant}

\theoremstyle{plain}

\newtheorem{theorem}{Theorem}[section]
\newtheorem{lemma}[theorem]{Lemma}
\newtheorem{proposition}[theorem]{Proposition}
\newtheorem{corollary}[theorem]{Corollary}

\newcommand{\ip}[1]{\left\langle#1\right\rangle}

\newcommand{\set}[2]{\left\{#1\,:\,#2\right\}}

\theoremstyle{definition}

\newtheorem*{definition}{Definition}

\theoremstyle{remark}

\newtheorem*{remark}{Remark}

\numberwithin{equation}{section}
\makeatletter

\@addtoreset{equation}{section}
\makeatother

\newcommand{\Z}{{\mathbb Z}}

\newcommand{\R}{{\mathbb R}}

\newcommand{\eps}{{\varepsilon}}

\def\curl{\operatorname{curl}}
\def\div{\operatorname{div}}

\title[The vortex method]{The vortex method for 2D ideal flows \\ in exterior domains}
\author[D.\ Ars\'enio, E.\ Dormy \& C.\ Lacave]{Diogo Ars\'enio, Emmanuel Dormy \& Christophe Lacave}

\address[D.\ Ars\'enio]{New York University Abu Dhabi, Abu Dhabi, United Arab Emirates.} 
\email{diogo.arsenio@nyu.edu}

\address[E.\ Dormy]{Department of Mathematics and their Applications, CNRS UMR 8553, \'Ecole Normale Sup\'erieure, Paris, France.}
\email{emmanuel.dormy@ens.fr}

\address[C.\ Lacave]{Univ.\ Grenoble Alpes, CNRS, Institut Fourier, F-38000 Grenoble, France.}
\email{christophe.lacave@univ-grenoble-alpes.fr}

\date{\today}

\begin{document}
\maketitle
\begin{abstract}
The vortex method is a common numerical and theoretical approach used to implement the motion of an ideal flow, in which the vorticity is approximated by a sum of point vortices, so that the Euler equations read as a system of ordinary differential equations. Such a method is well justified in the full plane, thanks to the explicit representation formulas of Biot and Savart. In an exterior domain, we also replace the impermeable boundary by a collection of point vortices generating the circulation around the obstacle. The density of these point vortices is chosen in order that the flow remains tangent at midpoints between adjacent vortices and that the total vorticity around the obstacle is conserved.

In this work, we provide a rigorous justification of this method for any smooth exterior domain, one of the main mathematical difficulties being that the Biot--Savart kernel defines a singular integral operator when restricted to a curve (here, the boundary of the domain). We also introduce an alternative method---the fluid charge method---which, as we argue, is better conditioned and therefore leads to significant numerical improvements.

\bigskip

\end{abstract}

\tableofcontents

Numerical methods describing the evolution of a fluid flow have an important practical interest in engineering and applications. Such approximation methods often also provide deeper theoretical insight and physical intuition into the properties of fluids. It is therefore important to justify that given methods provide good approximations of analytic solutions. The goal of this article is to validate mathematically the vortex method in exterior smooth domains for the two-dimensional Euler equations and to further develop other similar refined methods.

\section{The Euler equations in exterior domains}

The motion of an incompressible ideal fluid filling a domain $\Omega \subset \R^2$ is governed by the Euler equations:
\begin{equation}\label{Euler}
	\left\{
	\begin{array}{lcl}
		\partial_{t} u + u\cdot \nabla u +\nabla p=0 & \text{in} &(0,\infty)\times \Omega, \\
		\div u  =0 & \text{in} &[0,\infty)\times \Omega, \\
		u \cdot n  = 0 & \text{on} &[0,\infty)\times \partial \Omega, \\
		u(0,\cdot) = u_0 & \text{in} & \Omega, \\
	\end{array}
	\right.
\end{equation}
where $u=(u_{1}(t,x_{1},x_{2}),u_{2}(t,x_{1},x_{2}))$ is the velocity, $p=p(t,x_{1},x_{2})$ the pressure and $n$ the unit inward normal vector.

There is an extensive literature about the study of this difficult system, first on physical motivations and second because it provides elegant mathematical problems at the frontier of elliptic theory, dynamical systems, convex geometry and evolution partial differential equations. One may argue that the richness of these equations is due to the role of the vorticity:
\[
\omega(t,x):= \curl u(t,x) = \partial_{1} u_{2} - \partial_{2} u_{1}.
\]
Indeed, taking the curl of the momentum equation in \eqref{Euler}, we note that this quantity satisfies a transport equation:
\begin{equation}\label{Euler vort}
		\partial_{t} \omega + u\cdot \nabla \omega =0  \quad \text{in}\quad (0,\infty)\times \Omega.
\end{equation}
From this form, we may deduce several conservation properties (e.g.\ the conservation of all $L^p\left(\Omega\right)$-norms of $\omega$, for all $1\leq p \leq\infty$) which allow to establish the wellposedness of the Euler equations in various settings (standard references can be found in  \cite{GVL2,MajdaBertozzi}). One of the key steps in the analysis of \eqref{Euler} consists in reconstructing the velocity $u$ from the vorticity $\omega$ by solving the following elliptic problem:
\begin{equation}\label{elliptic}
	\left\{
	\begin{array}{lcl}
		\div u  =0 & \text{in}& \Omega, \\
		\curl u  =\omega & \text{in}& \Omega, \\
		u \cdot n  = 0 & \text{on}&\partial\Omega, \\
		u \rightarrow 0 & \text{as}& x\rightarrow\infty, \\
	\end{array}
	\right.
\end{equation}
where $\omega\in C^{0,\alpha}_c\left(\Omega\right)$, for some $0<\alpha\leq 1$.

In the case of the full plane $\Omega=\R^2$, any solution of
\begin{equation}\label{system BS}
		\div u  =0  \text{ in } \mathbb{R}^2, \quad
		\curl u  =\omega \text{ in } \mathbb{R}^2, \quad
		u  \rightarrow 0  \text{ as } x\rightarrow\infty,
\end{equation}
satisfies
\begin{equation*}
	\Delta u = \nabla^\perp \omega \quad\text{in }\mathbb{R}^2,
\end{equation*}
in the sense of distributions, which easily yields
\begin{equation*}
	u=K_{\mathbb{R}^2}[\omega]=\mathcal{F}^{-1}\frac{-i\xi^\perp}{|\xi|^2}\mathcal{F}\omega.
\end{equation*}
Here, the superscript $\perp$ denotes the rotation by $\pi/2$, that is $(x_{1},x_{2})^\perp =(-x_{2},x_{1})$. It follows, employing standard results on Fourier multipliers, that $K_{\mathbb{R}^2}$ has bounded extensions from $L^p$ to $\dot W^{1,p}$, for any $1<p<\infty$. Furthermore, writing $\Phi(x)=-\frac 1{2\pi}\log|x|$ the fundamental solution of the Laplacian in $\mathbb{R}^2$, it holds that (see e.g.\ \cite{gilbarg})
\begin{equation}\label{BS R2}
	\begin{aligned}
		u= & K_{\mathbb{R}^2}[\omega]=- \Phi*\left(\nabla^\perp\omega\right)=-\nabla^\perp \left(\Phi*\omega\right)
		\\
		= & \frac 1{2\pi}\int_{\mathbb{R}^2} \frac{(x-y)^\perp}{|x-y|^2}\omega(y)dy
		\in C^1\left(\mathbb{R}^2\right).
	\end{aligned}
\end{equation}
We refer to \cite[p.\ 249]{courant} for a justification of the $C^1$-regularity of $K_{\mathbb{R}^2}[\omega]$, for any $\omega\in C_c^{0,\alpha}\left(\Omega\right)$ (which may also be deduced from the representation formula \eqref{representation velocity}, below).

When $\Omega=\mathbb{R}^2\setminus \mathcal{C}$ is an exterior domain, with $\mathcal{C}$ a compact, smooth (i.e.\ its boundary is $C^\infty$) and simply connected set, there are an infinite number of solutions of \eqref{elliptic}, because there exists a unique harmonic vector field $H\in C^1\left(\Omega\right)\cap C\left(\overline\Omega\right)$ (see \cite[Proposition 2.1]{ILL}, for instance) verifying
\begin{equation}\label{harmonic}
	\begin{gathered}
		\div H  =0  \text{ in } \Omega, \quad
		\curl H  =0 \text{ in } \Omega, \quad
		\oint_{\partial\Omega} H\cdot\tau ds=1, \quad
		\\
		H\cdot n  =0 \text{ on } \partial\Omega, \quad
		H(x) \rightarrow 0  \text{ as } x\rightarrow\infty,
	\end{gathered}
\end{equation}
where $\tau:= n^\perp$ is the tangent vector to $\partial\Omega$ (note that $n$ points out of the obstacle $\mathcal{C}$ so that $\tau$ orients $\partial\Omega$ counterclockwise). In fact, it can be shown that $H$ belongs to $C^\infty\left(\Omega\right)$ and that all its derivatives are continuous up to the boundary $\partial\Omega$ (use the representation formula \eqref{harmonic field representation}, below).

Thus, in order to reconstruct uniquely the velocity in terms of the vorticity, the standard idea consists in prescribing the circulation:
\[
\oint_{\partial\Omega} u \cdot \tau \, ds = \gamma,
\]
where $\gamma\in \R$. This constraint is natural because Kelvin's theorem implies then that the circulation of $u$ around an obstacle is a conserved quantity for the Euler equations. With this additional condition, it holds now true that there exists a unique classical solution $u\in C^1\left(\Omega\right)\cap C\left(\overline\Omega\right)$ of
\begin{equation}\label{elliptic2}
	\left\{
	\begin{array}{lcl}
		\div u  =0 & \text{in}& \Omega, \\
		\curl u  =\omega & \text{in}& \Omega, \\
		u \cdot n = 0 & \text{on}& \partial\Omega, \\
		u  \rightarrow 0 &\text{as}& x\rightarrow\infty, \\
		\oint_{\partial\Omega} u \cdot \tau ds  = \gamma, &&
	\end{array}
	\right.
\end{equation}
where $\omega\in C^{0,\alpha}_c\left(\Omega\right)$, for some $0<\alpha\leq 1$, and $\gamma\in\mathbb{R}$.

To solve this elliptic problem, we may introduce (as in \cite[Lemma 2.2 and Proposition 2.1]{ILL}; see also \cite[Section~1.2]{marchioro}) the Green function with Dirichlet boundary condition $G_{\Omega}:\ \Omega\times\Omega\to \R$ as the function verifying:
	\begin{align*}
		G_{\Omega}(x,y) & =G_{\Omega}(y,x) &&\text{for all } (x,y)\in \Omega^2, \\
		\Delta_{x} G_{\Omega}(x,y) & = \delta(x-y)\ &&\text{for all } (x,y)\in \Omega^2, \\
		G_{\Omega}(x,y) & =0 &&\text{for all } (x,y)\in \partial\Omega \times \Omega,
	\end{align*}
where $\delta$ denotes the Dirac function centered at the origin. Using a conformal $C^\infty$-diffeomorphism $T:\Omega\to \left\{|x|>1\right\}$ such that (this transformation can be constructed through the Riemann mapping theorem; see \cite[Lemma 2.1]{ILL}):
\begin{itemize}
	\item it is bijective from $\overline\Omega$ onto $\left\{|x|\geq 1\right\}$,
	\item all its derivatives are continuous up to $\partial\Omega$ and are uniformly bounded over $\Omega$,
	\item all the derivatives of its inverse are continuous up to $\left\{|x|= 1\right\}$ and are uniformly bounded over $\left\{|x|> 1\right\}$,
	\item there is $\beta\in\mathbb{R}\setminus\left\{0\right\}$ such that $T(x)-\beta x$ and $T^{-1}(x)-\beta^{-1} x$ are uniformly bounded over $\Omega$ and $\left\{|x|> 1\right\}$, respectively,
\end{itemize}
one has the formula
\begin{equation*}
	G_{\Omega}(x,y) = \frac1{2\pi} \log \frac{\left|T(x)-T(y)\right|}{\left|T(x)-T(y)^*\right| \left|T(y)\right|},
\end{equation*}
with the notation $y^*=\frac{y}{|y|^2}$, for any $y\in\mathbb{R}^2\setminus\left\{0\right\}$. This expression allows us to write explicitly the solution of \eqref{elliptic2} (for all details, we refer e.g.\ to \cite{ILL}):
\begin{equation}\label{BS exterior}
	\begin{aligned}
		u(x) = & K_{\Omega}[\omega](x) + \alpha H(x) := \int_{\Omega} \nabla^\perp_{x} G_{\Omega}(x,y) \omega(y)dy +  \alpha H(x)
		\\	
		= & \frac1{2\pi} \int_{\Omega} \left( \frac{DT^t(x)\left(T(x)-T(y)\right)}{\left|T(x)-T(y)\right|^2} - \frac{DT^t(x)\left(T(x)-T(y)^*\right)}{\left|T(x)-T(y)^*\right|^2} \right)^{\perp}\omega(y) dy
		\\
		& + \frac{\alpha}{2\pi}
		\frac{\left(DT^t(x)T(x)\right)^\perp}{\left|T(x)\right|^2},
	\end{aligned}
\end{equation}
where we have set
\begin{equation*}
\alpha = \gamma + \int_{\Omega}\omega(y) dy.
\end{equation*}
Note that the total mass of the vorticity is also a conserved quantity of incompressible ideal two-dimensional flows. Note also that \eqref{BS exterior} uses the representation
\begin{equation}\label{harmonic field representation}
	H(x)=\frac{\left(DT^t(x)T(x)\right)^\perp}{2\pi\left|T(x)\right|^2},
\end{equation}
for the unique solution $H$ of \eqref{harmonic}. Further employing that (see \cite[p.\ 249]{courant}, for instance, or use the representation formula \eqref{representation velocity}, below)
\begin{equation}\label{convolution T}
	\int_{\left\{|y|>1\right\}}\frac{x-y}{\left|x-y\right|^2} \omega\left(T^{-1}(y)\right)\left|\det DT^{-1}(y)\right|dy\in C^1\left(\mathbb{R}^2\right),
\end{equation}
it is readily seen from \eqref{BS exterior} that $u\in C^1\left(\overline\Omega\right)$, thus yielding the unique classical solution to \eqref{elliptic2}.

\bigskip

All in all, the Euler equations around the obstacle $\mathcal{C}$ can be seen as the transport of the vorticity \eqref{Euler vort} by the velocity field $u$ defined by \eqref{BS exterior}. This property conveniently allows for the use of various mathematical theories and it is therefore crucial to develop efficient and robust methods to rebuild the velocity field $u$ from the vorticity $\omega$ or an approximation of it. In particular, for the sake of applications, we are going to focus on the theoretical and numerical approximation of \eqref{BS exterior}.

\section{The vortex method}

In the full plane $\mathbb{R}^2$, when the initial vorticity is close to being concentrated at $N$ given points $\left\{x_i^0\right\}_{i=1}^N\subset\mathbb{R}^2$, i.e.\ $\omega(t=0)\sim\sum_{i=1}^N \gamma_{i} \delta_{x_{i}^0}$ in some suitable sense, Marchioro and Pulvirenti \cite{MP91} have shown that the corresponding solution of the Euler equations in the full plane has a vorticity which remains close to a combination of Dirac masses $\omega(t)\sim\sum_{i=1}^N \gamma_{i} \delta_{x_{i}(t)}$ (in some suitable sense) where the centers $\{ x_{i} \}_{i=1}^N$ verify a system of ODE's, called the point vortex system:
\begin{equation}\label{point vortex 0}
	\left\{
	\begin{split}
		\dot{x}_{i}(t) &= \frac1{2\pi} \sum_{j\neq i} \gamma_{j} \frac{(x_{i}(t)-x_{j}(t))^\perp}{|x_{i}(t)-x_{j}(t)|^2},
		\\
		x_{i}(0)&=x_{i}^0.
	\end{split}
	\right.
\end{equation}
Here, the point vortex $\gamma_{i} \delta_{x_{i}(t)}$ moves under the velocity field produced by the other point vortices. This approximation by point vortices remains valid as long as the solution to \eqref{point vortex 0} exists, i.e.\ as long as no collisions occur.

It turns out that this Lagrangian formulation is much easier to handle numerically than the Eulerian formulation \eqref{Euler vort}. Indeed, standard numerical methods on \eqref{Euler vort} generate an ``inherent numerical viscosity'' and some quantities which should be conserved instead decrease (see e.g.\ \cite{Hirsch,Toro}). Actually, smoothing the Biot--Savart kernel by mollifying $\frac{x^\perp}{|x|^2}$ in \eqref{point vortex 0} gives a more stable system, called the vortex blob method (i.e.\ an approximation of the vorticity by Dirac masses and a regularization of the kernel). The stability and the convergence as $N\rightarrow\infty$ of the vortex blob and point vortex methods have been extensively studied: in \cite{CCM88} for the vortex blob method when the initial vorticity is bounded, in \cite{GHL90} for the point vortex method for smooth initial data and in \cite{liu, Schochet} for both methods and for weak solutions as e.g.\ a vortex sheet (see also the textbook \cite{CK00}).

However, all these works use the explicit formula of the Biot--Savart law in the full plane \eqref{BS R2} where the flow $\frac{(x-x_{i})^\perp}{2\pi |x-x_{i}|^2}$ is identified with $K_{\R^2}[\delta_{x_{i}}]$. In an exterior domain, the Biot--Savart law is much more complicated. A possible approach could be to use the  explicit formula \eqref{BS exterior} in order to adapt the previous vortex methods. But such an approach would yield serious practical difficulties. Indeed, explicit Riemann mappings are only available for few domains with specific symmetry properties. In general, if we consider that $\Omega$ is the exterior of a smooth, compact, simply connected subset of $\R^2$, formula \eqref{BS exterior} only gives an implicit representation, which has some theoretical interest, but remains impractical.

Our alternative strategy consists in modeling the impermeable boundary of the exterior domain by a collection of point vortices $\sum_{i=1}^N \frac{\gamma_{i}^N(t)}N \delta_{x_{i}^N}$, where the vortex positions $\{x_{i}^N\}_{i=1}^N$ are fixed on $\partial\Omega$ but the density of points $\{\gamma_{i}^N\}_{i=1}^N$ now evolves with time and is chosen in order that the resulting velocity field remains tangent at midpoints on the boundary between the $x_i^N$'s. Note that this approach appears sometimes in physics and engineering books (see e.g.\ \cite{BL93,GZ09}).

\subsection{Static convergence of the vortex approximation}

We need to explain now how the vortex method is used to replace the obstacle in \eqref{elliptic2} by vortices. To this end, we introduce $u_{P}$ the solution of \eqref{system BS} in the full plane, which is explicitly given by \eqref{BS R2}:
\begin{equation}\label{uP}
	u_{P}:=K_{\R^2}[\omega]\in C^1\left(\mathbb{R}^2\right)\subset C^1\left(\overline\Omega\right),
\end{equation}
and the remainder velocity field $u_{R}$ defined by:
\begin{equation}\label{uR}
	u_R:=u-u_P\in C^0\left(\overline\Omega\right)\cap C^1\left(\Omega\right),
\end{equation}
where $u$ is the unique solution to \eqref{elliptic2}. As $\omega$ is compactly supported in $\Omega$ we get by the Stokes formula that $\oint_{\partial\Omega} u_P \cdot \tau ds =\int_{\overline \Omega^c}\curl u_P = \int_{\overline \Omega^c} \omega=0$. Hence, it is readily seen that $u_R$ solves
\begin{equation}\label{eq uR}
	\left\{
	\begin{array}{lcl}
		\div u_R  =0 & \text{in}& \Omega, \\
		\curl u_R  =0 & \text{in}& \Omega, \\
		u_R \cdot n  = -u_P\cdot n & \text{on}&\partial\Omega, \\
		u_R  \rightarrow 0 & \text{as}& x\rightarrow\infty, \\
		\oint_{\partial\Omega} u_R \cdot \tau ds  = \gamma. &&
	\end{array}
	\right.
\end{equation}
In particular, $u_R$ is harmonic in $\Omega$ and therefore it is smooth in $\Omega$, i.e.\ $u_R\in C^\infty\left(\Omega\right)$ (see \cite[Corollary 8.11]{gilbarg} or \cite{gray}).

The vortex method for the exterior domain $\Omega$ is essentially an approximation procedure of $u_R$ by point vortices on $\partial \Omega$. More precisely, let now $\left(x_{1}^N, x_{2}^N,\dots , x_{N}^N\right)$ be the positions of $N$ distinct point vortices on the boundary $\partial\Omega$. Given an arc-length parametrization $l:\left[0,\left|\partial\Omega\right|\right]\to\mathbb{R}^2$ of $\partial\Omega$, oriented counterclockwise so that $l'=\tau$, we consider
\begin{equation}\label{xi}
	0= s_{1}^N< s_{2}^N<\dots < s_{N}^N<\left|\partial\Omega\right| \text{ such that }x_{i}^N=l\left(s_i^N\right).
\end{equation}
We further introduce some intermediate points on the boundary, for each $i=1,\dots,N$ (setting $s_{N+1}^N=\left|\partial\Omega\right|$):
\begin{equation}\label{tildexi}
	\tilde s_{i}^N \in (s_{i}^N,s_{i+1}^N),\quad \tilde x_{i}^N:=l\left(\tilde s_i^N\right).
\end{equation}
The method consists in approximating the solution $u_R$ to \eqref{eq uR} by a suitable flow
\begin{equation}\label{approx}
u_{\rm app}^N(x):=\frac1{2\pi} \sum_{j=1}^N \frac{\gamma_{j}^N}N \frac{( x - x_{j}^N)^\perp}{|x - x_{j}^N|^2} = K_{\R^2}\Big[ \sum_{j=1}^N \frac{\gamma_{j}^N}N \delta_{x_{j}^N} \Big],
\end{equation}
whose vorticity is precisely made of $N$ point vortices with densities $\left\{\frac{\gamma_i^N}{N}\right\}_{i=1}^N$ on the boundary $\partial\Omega$. Observe that the approximate flow $u_{\rm app}^N$ is not defined at each $x_i^N$ on the boundary, which is why we introduce the intermediate points $\tilde x_i^N$. These points will be used to compute some values of $u_{\rm app}^N$ on the boundary.

It is to be emphasized that this approximation is consistent with and motivated by the physical idea that the circulation around an obstacle is created by a collection of vortices on the boundary of the obstacle, i.e.\ a vortex sheet on the boundary.

However, it is {\it a priori} not obvious that such a flow $u_{\rm app}^N$ can be made a good approximation of $u_R$. Nevertheless, note that $u_{\rm app}^N$ already naturally satisfies
\begin{equation*}
	\left\{
	\begin{array}{lcl}
		\div u_{\rm app}^N  =0 & \text{in}& \Omega, \\
		\curl u_{\rm app}^N  =0 & \text{in}& \Omega, \\
		u_{\rm app}^N  \rightarrow 0 & \text{as}& x\rightarrow\infty.
	\end{array}
	\right.
\end{equation*}
Therefore, the key idea lies in enforcing that the boundary and circulation conditions be satisfied as $N\rightarrow\infty$ by setting $\gamma^N=(\gamma_{1}^N,\dots, \gamma_{N}^N)\in\mathbb{R}^N$ to be the solution of the following system of $N$ linear equations:
\begin{equation}\label{point vortex}
	\begin{aligned}
		& \frac1{2\pi}\sum_{j=1}^N \frac{\gamma_{j}^N}N \frac{(\tilde x_{i}^N - x_{j}^N)^\perp}{|\tilde x_{i}^N - x_{j}^N|^2}\cdot n(\tilde x_{i}^N)
		= -[u_{P}\cdot n](\tilde x_{i}^N), \quad \text{for all }i=1,\dots, N-1,\\
		& \sum_{i=1}^N \frac{\gamma_{i}^N}N = \gamma.
	\end{aligned}
\end{equation}
In order to emphasize the dependence of $u_{\rm app}^N$ on $\omega$ (through $u_P$) and $\gamma$, we will sometimes use the notation $u_{\rm app}^N=u_{\rm app}^N[\omega,\gamma]$. Note that $u_{\rm app}^N$ is linear in $(\omega,\gamma)$.

It is shown below, under suitable hypotheses on the placement of point vortices and provided $N$ is sufficiently large, that the above system always has a unique solution $\gamma^N$. The fact that $u_{\rm app}^N$ is a good approximation of $u_R$ is precisely the content of our first main theorem below (see Theorem \ref{main theo}). Clearly, it then follows that $u$ is well approximated by $u_{\rm app}^N+K_{\mathbb{R}^2}[\omega]$, which concludes the rigorous justification of the vortex method for the boundary of an obstacle applied to the elliptic system \eqref{elliptic2}.

We give now a precise definition of a well distributed mesh $\left\{x_i^N\right\}_{1\leq i\leq N}$ and $\left\{\tilde x_i^N\right\}_{1\leq i\leq N}$.
\begin{definition}
	We say that the points $\left\{x_i^N\right\}_{1\leq i\leq N}$ and $\left\{\tilde x_i^N\right\}_{1\leq i\leq N}$ given by \eqref{xi}-\eqref{tildexi} are \emph{well distributed} on $\partial\Omega$ if there exists an integer $\kappa \geq 2$ such that, as $N\to\infty$,
	\begin{equation}\label{mesh2}
		\max_{i=1,\ldots,N}\left|s_i^N-\theta_i^N\right|=\mathcal{O}\left(N^{-(\kappa +1)}\right)
		\ \text{and}\ 
		\max_{i=1,\ldots,N}\left|\tilde s_i^N-\tilde \theta_i^N\right|=\mathcal{O}\left(N^{-(\kappa +1)}\right),
	\end{equation}
	where
	\begin{equation}\label{mesh}
	\theta_{i}^N = \frac{(i-1) \left|\partial\Omega\right|}{N} \quad \text{and} \quad \tilde\theta_{i}^N = \frac{(i-\frac 12) \left|\partial\Omega\right|}{N} \quad \text{for all } i=1,\dots, N.
	\end{equation}
	The points on $\partial\Omega$ corresponding to $\left\{\theta_i^N\right\}_{1\leq i\leq N}$ and $\left\{\tilde\theta_i^N\right\}_{1\leq i\leq N}$ are said to be \emph{uniformly distributed}.
\end{definition}

Our first main result states that the approximate flow $u_{\rm app}^N$, constructed through the procedure \eqref{point vortex}, is a good approximation of $u_{R}$ provided the vortices are well distributed on $\partial\Omega$.

\begin{theorem}\label{main theo}
	Let $\omega\in L^1_c\left(\Omega\right)$ and $\gamma\in\mathbb{R}$ be given. For any $N\geq 2$, we consider a well distributed mesh satisfying \eqref{mesh2} and $u_{P}$ defined in \eqref{uP}.
	
	Then, there exists $N_0$ (independent of $\omega$ and $\gamma$) such that, for all $N\geq N_0$, the system \eqref{point vortex} admits a unique solution $\gamma^N\in \R^N$. Moreover, for any closed set $K\subset \Omega$  there exists a constant $C>0$ independent of $N$, $K$, $\omega$ and $\gamma$ such that
		\begin{align*}
			\| u_{R} - u_{\rm app}^N \|_{L^\infty(K)}
			\hspace{-5mm}&
			\\
			\leq & \frac{C}{N^\kappa }
			\left(
			\frac 1{\operatorname{dist}\left(K,\partial\Omega\right)} + \frac 1{\operatorname{dist}\left(K,\partial\Omega\right)^{\kappa +2}}
			\right)
			\\
			& \times \left(\left(
			\frac 1{\operatorname{dist}\left(\operatorname{supp}\omega,\partial\Omega\right)} + \frac 1{\operatorname{dist}\left(\operatorname{supp}\omega,\partial\Omega\right)^{\kappa +1}}
			\right)\left\|\omega\right\|_{L^1\left(\mathbb{R}^2\right)} + |\gamma|\right),
		\end{align*}
	where $u_{\rm app}^N$ is given by \eqref{approx} in terms of $\gamma^N$ and $u_{R}$ is the continuous flow \eqref{uR}.
\end{theorem}

We refer to \cite{ADLproc} for an equivalent theorem in the much simpler case of the unit disk $\mathcal{C}=\overline{B(0,1)}$. This restricted geometry allows for an easier proof based on the circular Hilbert transform.

If we consider the uniformly distributed mesh, then the estimate in Theorem
\ref{main theo} holds true for any $\kappa \geq 2$ but $C$ depends on
$\kappa $.
This implies that the numerical method converges faster than any fixed
order method. It is likely that assuming the boundary is analytic enforces
exponential convergence (see Appendix \ref{riemann appendix} for the
convergence of Riemann sums).

The proof of Theorem \ref{main theo} is developed over the course of the next few sections. In Section \ref{vortex sheet}, we establish important representation formulas for the solution of \eqref{eq uR} and we show the link between our problem, Cauchy integrals and the Poincar\'e--Bertrand formula. Then, in Section \ref{section discrete}, we prove that the linear system \eqref{point vortex} is invertible for sufficiently many point vortices. In Section \ref{sect:conv}, we establish that $(u_{R}-u_{\rm app}^N)\cdot n\vert_{\partial \Omega}$ converges to zero in a weak sense together with other related convergence properties on $\partial \Omega$. Finally, in Section \ref{proof of main}, we deduce that such a weak convergence implies a stronger form of convergence away from the boundary and thus reach the conclusion of Theorem \ref{main theo}.

\begin{remark}
	Let us already advertise here that, in Section \ref{charges}, we introduce and discuss a novel discretization method of the flow $u_R$---the fluid charge method---and establish convergence results similar to Theorem \ref{main theo} (see Theorems \ref{main theo alt} and \ref{main theo alt alt} therein), which may potentially improve the efficiency of corresponding numerical methods.
\end{remark}

\begin{remark}
	Removing the harmonic part $H(x)$ from \eqref{BS exterior} and the circulation condition in \eqref{elliptic2} and \eqref{eq uR}, the above main result can be readily adapted to describe an ideal fluid inside a bounded domain. It is also possible to consider obstacles with several connected components by prescribing a circulation condition for each obstacle. More precisely, assuming that the whole obstacle is such that its boundary can be decomposed into a disjoint union
	\begin{equation*}
		\partial \Omega=\partial\Omega_1\cup\dots\cup \partial\Omega_n,
	\end{equation*}
	where each $\partial\Omega_i$ is a simple closed curve enclosing a connected component of the obstacle, the elliptic problem \eqref{elliptic2} has then to be generalized to
	\begin{equation*}
		\left\{
		\begin{array}{lcl}
			\div u  =0 & \text{in}& \Omega, \\
			\curl u  =\omega & \text{in}& \Omega, \\
			u \cdot n = 0 & \text{on}& \partial\Omega, \\
			u  \rightarrow 0 &\text{as}& x\rightarrow\infty, \\
			\oint_{\partial\Omega_i} u \cdot \tau ds  = \gamma_i & \text{for each}& i=1,\ldots,n,
		\end{array}
		\right.
	\end{equation*}
	where $\gamma_i\in\mathbb{R}$ is the prescribed circulation around each connected component. The corresponding systems \eqref{eq uR} and \eqref{point vortex} undergo similar modifications. Theorem \ref{main theo} can then be adapted to account for multiple circulations $\gamma_i$.
	
	Finally, one can also adapt the methods in this work to handle flows with non-zero velocities at infinity. Note, however, that the consideration of non-smooth domains (with corners and cusps, for instance) would require subtle and nontrivial adaptations (particularly altering the analysis conducted in Section~\ref{vortex sheet}, below) which we leave for other works.
\end{remark}

Numerically, we indeed verify that the system \eqref{point vortex} is always invertible for large $N$, and that, on any compact set $K$, the flow $u_{\rm app}^N$ (given in \eqref{approx}) converges in the $L^\infty$-norm faster than any $N^{-\kappa}$ with $\kappa>0$ for a regular distribution of points. Deviations from a regular grid yields a slower convergence. The rate obtained in Theorem \ref{main theo} may however not be optimal in the case of random perturbations (of size $N^{-\kappa-1}$) of the uniform grid, because numerical simulations indicate faster convergence than $N^{-\kappa}$.

\subsection{Dynamic convergence of the vortex approximation}\label{dynamic thm}

We have previously explained how the influence of an obstacle on a flow solving \eqref{elliptic2} can be modeled by a collection of vortices on its boundary. We explain now how this approximation procedure is used to replace the obstacle in the Euler equations \eqref{Euler} by an evolving collection of vortices, thereby providing a dynamic picture of the vortex method.

To this end, let $\omega_0\in C_c^1\left(\Omega\right)$ and consider the unique classical solution $\omega\in C^1_c\left([0,t_1]\times \overline\Omega\right)$ (note that all natural definitions of $C^1$ on closed sets are equivalent here, for $\partial\Omega$ is smooth; see e.g.\ \cite{whitney}) constructed in \cite{Kikuchi}, for some fixed but arbitrary time $t_1>0$, of
\begin{equation}\label{dynamic1}
	\left\{
	\begin{aligned}
		& \partial_{t} \omega + u\cdot \nabla \omega =0,
		\\
		& \omega(t=0)=\omega_0,
	\end{aligned}
	\right.
\end{equation}
with a velocity flow
\begin{equation*}
	u=K_{\Omega}[\omega](x) + \alpha H(x)\in C^1\left([0,t_1]\times \overline\Omega\right),
\end{equation*}
where $K_{\Omega}$, $\alpha$ and $H$ are given explicitely in \eqref{BS exterior}, for some prescribed circulation $\gamma\in\mathbb{R}$. It is to be emphasized that the main theorem from \cite{Kikuchi} concerns the Eulerian formulation \eqref{Euler}. The corresponding proofs, however, are based on the vorticity formulation and indeed establish the above-mentioned wellposedness of \eqref{dynamic1}.

We recall that a classical estimate, which we reproduce, later on, in Section \ref{dynamic proofs 1}, shows that the support of $\omega$ in $x$ remains uniformly bounded away from the boundary $\partial\Omega$ (see \eqref{support away boundary}).

\bigskip

Let us also focus on the following vortex approximation of \eqref{dynamic1}, for sufficiently large integers $N$ (at least as large as $N_0$ determined by Theorem \ref{main theo} so that \eqref{point vortex} is invertible):
\begin{equation}\label{dynamic5}
	\left\{
	\begin{aligned}
		& \partial_{t} \omega^N + u^N\cdot \nabla \omega^N =0,
		\\
		& \omega^N(t=0) = \omega_0,
	\end{aligned}
	\right.
\end{equation}
for the same initial data $\omega_0\in C_c^1\left(\Omega\right)$ extended by zero outside $\Omega$ and with a velocity flow
\begin{equation}\label{uN}
	u^N=K_{\mathbb{R}^2}[\omega^N]+u_{\rm app}^N[\omega^N,\gamma],
\end{equation}
where $u_{\rm app}^N[\omega^N,\gamma]$ is given by \eqref{approx}-\eqref{point vortex}, for some prescribed $\gamma\in\mathbb{R}$ and where $u_P$ in the right-hand side of \eqref{point vortex} is now $K_{\mathbb{R}^2}[\omega^N]$.

The difficulty in solving system \eqref{dynamic5} resides in that the velocity flow $u^N$ is singular at the mesh points $x_i^N$ on the boundary $\partial\Omega$. However, we are able to claim the wellposedness of \eqref{dynamic5} in $C^1$ at least on some finite time interval, which can be arbitrarily large provided $N$ is sufficiently large. More precisely we show the following theorem, whose proof relies crucially on Theorem \ref{main theo} and is deferred to Section \ref{dynamic proofs}, for clarity.

\begin{theorem}\label{wellposedness}
	Let $\omega_0\in C_c^1\left(\Omega\right)$, $\gamma\in\mathbb{R}$ and consider any fixed time $t_1>0$. Then, for a well distributed mesh on $\partial \Omega$, there exists $N_1\geq N_0$ ($N_0$ is determined in Theorem \ref{main theo}) such that, for any $N\geq N_1$, there is a unique classical solution $\omega^N\in C^1_c\left([0,t_1]\times \Omega\right)$ to \eqref{dynamic5}. Moreover, the sequence of solutions $\left\{\omega^N\right\}_{N\geq N_1}$ is uniformly bounded in $C^1_c\left([0,t_1]\times \Omega\right)$.
\end{theorem}

\begin{remark}
	It is to be emphasized that, for a given fixed $N$, it may not be possible to prolong the classical solution from the preceding theorem indefinitely due to the interaction of the vortices defining the singular flow $u_{\rm app}^N$ with the vorticity $\omega^N$. This justifies the introduction of $N_1$ possibly depending on $t_1$ and $\omega_0$.
\end{remark}

\begin{remark}
	In Section \ref{charges}, we give in Theorem \ref{wellposedness alt} a similar wellposedness result for the fluid charge method.
\end{remark}

The following main theorem establishes the convergence of system \eqref{dynamic5} towards system \eqref{dynamic1} as $N\to\infty$, thereby completing the mathematical validation of the vortex method for the boundary of an obstacle in the Euler equations \eqref{Euler}. We emphasize, again, that the practical usefulness of this method lies in the computation of $\gamma^{N}$ through \eqref{point vortex} allowing the construction of an approximate flow $K_{\R^2}\left[\omega^N\right]+ u_{\rm app}^N\left[\omega^N,\gamma\right]$ which only requires the use of the Biot--Savart kernel in the whole plane and does not resort to \eqref{BS exterior}.

\begin{theorem}\label{main convergence}
	Let $\omega_0\in C_c^1\left(\Omega\right)$, $\gamma\in\mathbb{R}$ and consider any fixed time $t_1>0$. Then, for a well distributed mesh on $\partial\Omega$, as $N\to\infty$, the unique classical solution $\omega^N\in C^1_c\left([0,t_1]\times \Omega\right)$ to \eqref{dynamic5} converges uniformly towards the unique classical solution $\omega\in C^1_c\left([0,t_1]\times \Omega\right)$ to \eqref{dynamic1}. More precisely, it holds that
	\begin{equation*}
		\left\|\omega-\omega^N\right\|_{L^\infty([0,t_1]\times\Omega)}=\mathcal{O}\left(N^{-\kappa}\right).
	\end{equation*}
\end{theorem}

The proof of the above theorem is given in Section \ref{dynamic proofs}. It relies on both Theorems \ref{main theo} and \ref{wellposedness}.

\begin{remark}
	In Section \ref{charges}, we provide in Theorem \ref{main convergence alt} a similar convergence result for the fluid charge method.
\end{remark}

\begin{remark}
	In the above theorem, our original motivation was to reconstruct classical solutions from the vortex approximation. Observe, however, that the convergence estimate given in Theorem \ref{main theo} only requires that the vorticity $\omega$ be in $L_c^1$, whereas Theorems~\ref{wellposedness} and \ref{main convergence} operate at a much higher $C^1$-regularity for the vorticity.
	
	In fact, in view of the weak assumptions of Theorem \ref{main theo}, we believe that, using the vortex approximation studied in this work, it is also possible to recover solutions of the incompressible Euler equations in the much weaker setting of Yudovich's solutions, where the vorticity $\omega$ merely belongs to $L^\infty_c$ uniformly in time (see \cite[Section 7.2]{bahouri} for Yudovich's theorem, as well as other global existence results in the whole plane).
	
	Indeed, such a convergence result would hinge on the fact that all vorticities produced by the vortex approximation have compact supports that remain uniformly bounded away from the boundary of the domain, in order to ensure that all solutions exist on a uniform interval of time. This fact is a consequence of our analysis from Section \ref{dynamic proofs 1} and, in particular, from the proof of Proposition \ref{linear wellposedness} where the distance from the support of the vorticity to the boundary is controlled by the $L^\infty$-norm and the support of the initial vorticity (and not of its gradient). It seems that the remainder of a convergence proof would follow from standard weak compactness arguments.
	
	For the sake of brevity, however, we will not be going into further detail on this interesting topic.
\end{remark}

\begin{remark}
	Finally, notice that it is also possible to combine the vortex method for the boundary of an exterior domain with the aforementioned vortex method in the whole plane (consisting in an approximation of the vorticity $\omega$ by a collection of point vortices) in order to obtain a full and dynamic vortex method for an exterior domain. To this end, we consider an approximation of the initial vorticity $\omega_{0}$ by a combination of point vortices $\sum_{k=1}^{M} \alpha_{k} \delta_{y_{k}(0)}$. Then, the position $y_{k}(t)$ of each point vortex evolves under the influence of the vector field created by the remaining vortices $\sum_{p\neq k} \alpha_{p} \delta_{y_{p}(t)}$ (with possible regularization of the kernel for the vortex blob method) and the fixed vortices on the boundary $\sum_{i=1}^{N} \frac{\gamma_{i}^N(t)}N \delta_{x_{i}^N}$, where the variable vortex density $\gamma_{i}^N(t)$ is computed through \eqref{point vortex} with $u_{P}$ replaced by $K_{\R^2} [ \sum_{k=1}^{M} \alpha_{k} \delta_{y_{k}(t)} ]$. A rigorous proof of convergence for this full vortex method remains challenging, though.
\end{remark}

\section{Representation formulas}\label{vortex sheet}

In this section, we present some representation formulas for the solution $u_R$ of \eqref{eq uR}, which are crucial for the justification of Theorem \ref{main theo} and whose understanding sheds light on the approximation of $u_R$ by point vortices on the boundary $\partial \Omega$.

Here, we are considering some given vorticity $\omega\in C^{0,\alpha}_c\left(\Omega\right)$, with $0<\alpha\leq 1$, and $\gamma\in\mathbb{R}$, and wish to construct a velocity field $u_R\in C^0\left(\overline\Omega\right)\cap C^1\left(\Omega\right)$ solving \eqref{eq uR}. Essentially, we show below that it is possible to express the solution to \eqref{eq uR} as a vortex sheet on the boundary $\partial\Omega$, which, again, is consistent with the physical idea that the flow around an obstacle is produced by a boundary layer of vortices.

\bigskip

To this end, we need to introduce the integral operators
\begin{equation}\label{AB}
	\begin{aligned}
		A\varphi(x) & = \int_{\partial\Omega}\frac{x-y}{|x-y|^2}\cdot n(x)\varphi(y)dy, & x& \in\partial\Omega,
		\\
		B\varphi(x) & = \int_{\partial\Omega}\frac{x-y}{|x-y|^2}\cdot \tau(x)\varphi(y)dy, & x& \in\partial\Omega,
	\end{aligned}
\end{equation}
and their adjoints
	\begin{align*}
		A^*\varphi(x) & = -\int_{\partial\Omega}\frac{x-y}{|x-y|^2}\cdot n(y)\varphi(y)dy, & x & \in\mathbb{R}^2,
		\\
		B^*\varphi(x) & = -\int_{\partial\Omega}\frac{x-y}{|x-y|^2}\cdot \tau(y)\varphi(y)dy, & x & \in\mathbb{R}^2.
	\end{align*}
These operators are closely related to Cauchy integrals, i.e.\ complex valued integrals of the type $\int_{\Gamma}\frac{\varphi(z_1)}{z-z_1}dz_1$ where $\Gamma\subset\mathbb{C}$ is a contour. Indeed, identifying the contour $\Gamma\subset\mathbb{C}$ with the boundary of our domain $\partial\Omega\subset\mathbb{R}^2$, note that
\begin{equation}\label{complex representation}
	\int_{\Gamma}\frac{\varphi(y_1,y_2)}{(x_1+ix_2)-(y_1+iy_2)}d(y_1+iy_2)
	=
	-\left(B^*+iA^*\right)\varphi(x_1,x_2).
\end{equation}

Up to a multiplicative constant, the function $A^*\varphi$ is the so-called \emph{double layer potential} associated with the density $\varphi$. Such operators have been extensively studied in the context of Dirichlet and Neumann problems for Laplace's equation (see the classical references \cite[Chapter IV]{courant} and \cite{kellogg}, for instance). For the sake of clarity and completeness, we will nevertheless provide below complete justifications of our methods. We also refer to \cite{fabes} for clear proofs of some functional properties, which are explored in this work, of such operators on smooth domains in any dimension. Finally, we emphasize that the smoothness of the domain, which we assume to hold throughout this paper, is not a mere technical simplification and that the loss of regularity of $\partial\Omega$ brings on subtle and difficult questions about the above operators. We avoid this interesting and important discussion altogether, though, and leave it for subsequent works. We refer to \cite{chang, coifman, fabes2, fabes3, verchota} concerning the theory of double layer potentials on Lipschitz domains.

We begin this section by gathering and discussing several properties of the above operators, which will then allow us to establish important representation formulas for $u_R$.

\subsection{Boundedness and adjointness}

Expressing $y-x=\tau(x)\left((y-x)\cdot\tau(x)\right)+\mathcal{O}\left(|y-x|^2\right)$, the integrals defining $A\varphi(x)$ and $A^*\varphi(x)$ are always well defined for any $x\in\partial\Omega$ and $\varphi\in C\left(\partial\Omega\right)$, and the operators $A$ and $A^*$ are bounded over $L^p\left(\partial\Omega\right)$, for all $1\leq p\leq\infty$. In particular, for all $\varphi,\psi\in C\left(\partial\Omega\right)$, it holds that
\begin{equation}\label{adjoint A}
	\int_{\partial\Omega}\psi(x)A\varphi(x) dx
	= \int_{\partial\Omega}A^*\psi(y)\varphi(y) dy.
\end{equation}
In fact, since $\Omega$ is smooth, a slightly more refined analysis (employing an explicit Taylor expansion of a smooth parametrization of $\partial\Omega$, for instance) shows that the integral kernels of $A$ and $A^*$ are smooth so that these operators are regularizing. More precisely, one can show that $A\varphi$ and $A^*\varphi$ belong to $C^\infty\left(\partial\Omega\right)$, for any $\varphi\in L^1\left(\partial\Omega\right)$.

On the other hand, for any $x\in\partial\Omega$, $B\varphi(x)$ and $B^*\varphi(x)$ only make sense in the sense of Cauchy's principal value:
\begin{equation}\label{limit B}
	\begin{aligned}
		B\varphi(x) & = \lim_{ \eps\to 0}\int_{\partial\Omega\setminus B(x, \eps)}\frac{x-y}{|x-y|^2}\cdot \tau(x)\varphi(y)dy,
		\\
		B^*\varphi(x) & = - \lim_{ \eps\to 0}\int_{\partial\Omega\setminus B(x, \eps)}\frac{x-y}{|x-y|^2}\cdot \tau(y)\varphi(y)dy.
	\end{aligned}
\end{equation}
Indeed, notice that $\frac{x-y}{|x-y|^2}\cdot \tau(y)=0$ whenever $y\in\partial B(x, \eps)$ and $x\in\partial\Omega$. It is therefore always possible to replace the integration domain in the above limit defining $B^*$ by $\left(\partial\Omega\setminus B(x,\eps)\right)\cup \left(\partial B(x,\eps)\cap \Omega^c\right)$, thereby avoiding the singularity at $x$ of the kernel, whence
\begin{equation}\label{B1}
	\begin{aligned}
		B^*1(x) = & - \lim_{ \eps\to 0}\int_{\partial\Omega\setminus B(x, \eps)}\frac{x-y}{|x-y|^2}\cdot \tau(y)dy
		\\
		= & - \lim_{ \eps\to 0}\int_{\Omega^c\setminus B(x,\eps)}\curl\frac{x-y}{|x-y|^2} dy=0,
	\end{aligned}
\end{equation}
by the divergence theorem (which holds for piecewise smooth domains). It follows that $B$ and $B^*$ are given by the formulas, for all $x\in\partial\Omega$,
	\begin{align*}
		B\varphi(x) & = \int_{\partial\Omega}\frac{x-y}{|x-y|^2}\cdot \left(\tau(x)-\tau(y)\right)\varphi(y)dy - B^*\varphi(x),
		\\
		B^*\varphi(x) & = \int_{\partial\Omega}\frac{x-y}{|x-y|^2}\cdot \tau(y)\left(\varphi(x)-\varphi(y)\right)dy,
	\end{align*}
which are clearly well defined for every $\varphi\in C^{0,\alpha}\left(\partial\Omega\right)$, with $0<\alpha\leq 1$, and that the limits in \eqref{limit B} are uniform over $\partial\Omega$. In particular, it is now readily verified that, for all $\varphi,\psi\in C^{0,\alpha}\left(\partial\Omega\right)$,
\begin{equation}\label{adjoint B}
	\begin{aligned}
		\int_{\partial\Omega}\psi(x)B\varphi(x) dx
		& =
		\lim_{\eps\to 0}\int_{\partial\Omega\times\partial\Omega \setminus\left\{|x-y|\geq \eps\right\}}
		\frac{x-y}{|x-y|^2}\cdot\tau(x)\psi(x)\varphi(y)dxdy
		\\
		& =
		\int_{\partial\Omega}B^*\psi(y)\varphi(y) dy.
	\end{aligned}
\end{equation}

Employing $\frac{x-y}{|x-y|^2}\cdot n(y)=\frac 1\eps$ whenever $y\in\partial B(x, \eps)\cap\Omega^c$, observe that a similar calculation yields
\begin{equation}\label{A1}
	\begin{aligned}
		A^*1(x) & = - \lim_{ \eps\to 0}\int_{\partial\Omega\setminus B(x, \eps)}\frac{x-y}{|x-y|^2}\cdot n(y)dy
		\\
		& = \lim_{ \eps\to 0}\left( \frac{1}{\eps}\left|\partial B(x, \eps)\cap\Omega^c\right| - \int_{\Omega^c\setminus B(x,\eps)}\div\frac{x-y}{|x-y|^2} dy\right)=\pi.
	\end{aligned}
\end{equation}
By duality, notice that the identities \eqref{B1} and \eqref{A1} establish
\begin{equation}\label{mean}
		\int_{\partial\Omega}B\varphi dx = 0
		\quad\text{and}\quad
		\int_{\partial\Omega}A\varphi dx = \pi \int_{\partial\Omega}\varphi dx,
\end{equation}
for every $\varphi\in C^{0,\alpha}\left(\partial\Omega\right)$.

More generally, the operators $B$ and $B^*$ are bounded over $L^p\left(\partial\Omega\right)$, for any $1<p<\infty$. Perhaps the easiest way to justify such boundedness properties is by considering the arc-length parametrization $l:\left[0,\left|\partial\Omega\right|\right]\to\mathbb{R}^2$ of $\partial\Omega$ (interpreted as a smooth periodic function over $\left[0,\left|\partial\Omega\right|\right]$) and expressing the kernel, whenever $|s-t|<\frac{2\left|\partial\Omega\right|}3$, as
\begin{equation}\label{kernel decomposition}
	\frac{l(s)-l(t)}{\left|l(s)-l(t)\right|^2}
	=
	\frac {\tau\left(l(t)\right)}{s-t} + r(s,t),
\end{equation}
where, using that $\tau\left(l(t)\right)=l'(t)$ and $|l'(t)|=1$, the function $r(s,t)$ can be written as
\begin{equation*}
	\begin{aligned}
		r(s,t) & =
		\frac{l(s)-l(t)}{\left|l(s)-l(t)\right|^2}
		-\frac {l'(t)}{s-t}
		\\
		& =
		\left(l'(t)+\frac{l(s)-l(t)}{s-t}\right)\cdot\left(\frac{l(t)+l'(t)(s-t)-l(s)}{(s-t)^2}\right)
		\frac{\frac{l(s)-l(t)}{s-t}}{\left|\frac{l(s)-l(t)}{s-t}\right|^2}
		\\
		& \quad +\frac {l(s)-l(t)-l'(t)(s-t)}{(s-t)^2}.
	\end{aligned}
\end{equation*}
By smoothness of the parametrization $l$, it is then clear that $r(s,t)$ is smooth over $\{|s-t|<\frac{2\left|\partial\Omega\right|}3\}$ and that (recall that $l'(t)\cdot l''(t)=0$ because $l$ is an arc-length parametrization)
\begin{equation*}
	\lim_{s\to t}r(s,t)=\frac{l''(t)}2.
\end{equation*}
The periodicity of $l$ and the boundedness of the Hilbert transform over $L^p\left(\mathbb{R}\right)$, for any $1<p<\infty$ (see \cite[Section 4.1]{grafakos}), then yields corresponding bounds on $B$ and $B^*$. Similarly, the behavior of the Hilbert transform over H\"older spaces allows us to deduce that $B\varphi,B^*\varphi\in C^{0,\alpha}\left(\partial\Omega\right)$ as soon as $\varphi\in C^{0,\alpha}$, with $0<\alpha<1$, and that $B\varphi,B^*\varphi\in C^{0,1-\eps}\left(\partial\Omega\right)$, for all $0<\eps<1$, as soon as $\varphi\in C^{0,1}$. This result is known as the Plemelj--Privalov theorem (see \cite[Chap.~2, \S~19 and 20]{musk}).

\subsection{The Plemelj formulas and the Poincar\'e--Bertrand formula}

The theory of double layer potentials (or of Cauchy integrals; see \cite{fabes, musk}) instructs us that, for a smooth boundary $\partial\Omega$ and for any $\varphi\in C^{0,\alpha}\left(\partial\Omega\right)$, with $0<\alpha\leq 1$, the functions $A^*\varphi$ and $B^*\varphi$ are continuous up to the boundary $\partial\Omega$ (see \cite[Chap. 2, \S\ 16]{musk}) and that the limiting values of $A^*\varphi$ and $B^*\varphi$ on $\partial\Omega$ are given by the Plemelj formulas (see \cite[Chap.\ 2, \S\ 17]{musk}):
\begin{equation}\label{plemelj}
	\begin{aligned}
		\lim_{\substack{x\rightarrow x_0\in\partial\Omega \\ x\in\Omega\cup\overline{\Omega}^c}}
		B^*\varphi(x) & = B^*\varphi(x_0),
		\\
		\lim_{\substack{x\rightarrow x_0\in\partial\Omega \\ x\in\Omega}}
		A^*\varphi(x) & = A^*\varphi(x_0) - \pi \varphi(x_0),
		\\
		\lim_{\substack{x\rightarrow x_0\in\partial\Omega \\ x\in\overline{\Omega}^c}}
		A^*\varphi(x) & = A^*\varphi(x_0) + \pi \varphi(x_0).
	\end{aligned}
\end{equation}

These limiting formulas can be used to show the celebrated Poincar\'e--Bertrand formula (see \cite[Chap. 3, \S\ 23]{musk}), which we now recall in its simpler version concerning the inversion of Cauchy integrals (see \cite[Chap. 3, \S\ 27]{musk}): for any smooth contour $\Gamma\subset \mathbb{C}$ and any $\varphi\in C^{0,\alpha}\left(\Gamma\right)$, with $0<\alpha\leq 1$, one has that
\begin{equation*}
	\int_{\Gamma}\frac{1}{z-z_1}\int_{\Gamma}\frac{\varphi(z_2)}{z_1-z_2}dz_2 dz_1
	=-\pi^2\varphi(z),\quad \text{for all }z\in \Gamma.
\end{equation*}
Translating this identity into real variables, utilizing \eqref{complex representation}, yields
\begin{equation*}
	\left(B^*+iA^*\right)^2\varphi(x) = -\pi^2\varphi(x).
\end{equation*}
Therefore, we deduce that
\begin{equation}\label{PBadjoint}
	\begin{aligned}
		\left(A^{*2}-B^{*2}\right)\varphi & = \pi^2\varphi,
		\\
		\left(A^*B^*+B^*A^*\right)\varphi & = 0.
	\end{aligned}
\end{equation}
Equivalently, in view of \eqref{adjoint A} and \eqref{adjoint B}, we note that the adjoint operators satisfy
\begin{equation}\label{PB}
	\begin{aligned}
		\left(A^2-B^2\right)\varphi & = \pi^2\varphi,
		\\
		\left(AB+BA\right)\varphi & = 0.
	\end{aligned}
\end{equation}
In the case of the disk $\mathcal{C}=\overline{B(0,1)}$, these identities correspond exactly to the inversion of the circular Hilbert transform.

\subsection{Boundary vortex sheets}

We focus now on velocity flows given as boundary vortex sheets:
\begin{equation}\label{boundary sheet}
	\begin{aligned}
		v(x) = & K_{\mathbb{R}^2}\left[g\delta_{\partial\Omega}\right]
		=\frac 1{2\pi}\int_{\partial\Omega} \frac{(x-y)^\perp}{|x-y|^2}g(y)dy
		\\
		= & \frac 1{2\pi}\left(B^*[ng]-A^*[\tau g]\right)(x)
		\in C^\infty\left(\mathbb{R}^2\setminus\partial\Omega\right),
	\end{aligned}
\end{equation}
for some suitable $g\in C^{0,\alpha}\left(\partial\Omega\right)$, with $0<\alpha\leq 1$. We show, later on, that any flow $u_R$ solving \eqref{eq uR} can be written as a boundary vortex sheet \eqref{boundary sheet}.

Note that the formula \eqref{boundary sheet} merely defines the flow $v$ away from the boundary $\partial\Omega$. However, using the Plemelj formulas \eqref{plemelj}, one can extend this flow by continuity to $\partial\Omega$ in two different ways, yielding either $v\in C\left(\overline \Omega\right)$ or $v\in C\left(\Omega^c\right)$. More precisely, it is understood that the flow $v\in C\left(\overline \Omega\right)$ has the limit boundary values
\begin{equation}\label{plemelj exterior}
	\lim_{\substack{x\rightarrow x_0\in\partial\Omega \\ x\in\Omega}} v(x)=
	\frac 1{2\pi}\int_{\partial\Omega} \frac{(x_0-y)^\perp}{|x_0-y|^2}g(y)dy +
	\frac 12 \tau(x_0)g(x_0),
\end{equation}
whereas the flow $v\in C\left(\Omega^c\right)$ has the limit boundary values
\begin{equation}\label{plemelj interior}
	\lim_{\substack{x\rightarrow x_0\in\partial\Omega \\ x\in\overline{\Omega}^c}} v(x)=
	\frac 1{2\pi}\int_{\partial\Omega} \frac{(x_0-y)^\perp}{|x_0-y|^2}g(y)dy -
	\frac 12 \tau(x_0)g(x_0),
\end{equation}
where, again, the integrals in the right-hand sides above are defined in the sense of Cauchy's principal value. In other words, the normal component of $v(x)$ is continuous across the boundary $\partial\Omega$, where it takes the value
\begin{equation*}
	v\cdot n(x)=-\frac {1}{2\pi}Bg(x),\quad\text{for all }x\in\partial\Omega,
\end{equation*}
whereas its tangential component has a jump of size $g(x_0)$ at $x_0\in\partial\Omega$ and takes the values on the boundary
\begin{equation}\label{tangent plemelj}
	\begin{aligned}
		v\cdot \tau(x) & =\frac {1}{2\pi}Ag(x)+\frac 12 g(x),\quad\text{for all $x\in\partial\Omega$ from within $\Omega$},
		\\
		v\cdot \tau(x) & =\frac {1}{2\pi}Ag(x)-\frac 12 g(x),\quad\text{for all $x\in\partial\Omega$ from within $\overline{\Omega}^c$}.
	\end{aligned}
\end{equation}

Further employing \eqref{mean}, it follows that the flow $v(x)$ solves uniquely the systems
\begin{equation}\label{boundary sheet system}
	\left\{
	\begin{array}{lcl}
		\div v  =0 & \text{in}& \Omega, \\
		\curl v  =0 & \text{in}& \Omega, \\
		v \cdot n  = -\frac {1}{2\pi}Bg(x) & \text{on }&\partial\Omega, \\
		v  \rightarrow 0 & \text{as}& x\rightarrow\infty, \\
		\oint_{\partial\Omega} v \cdot \tau ds  = \oint_{\partial\Omega}gds, &&
	\end{array}
	\right.
\end{equation}
and
\begin{equation}\label{system interior}
	\left\{
	\begin{array}{lcl}
		\div v  =0 & \text{in}& \overline{\Omega}^c, \\
		\curl v  =0 & \text{in}& \overline{\Omega}^c, \\
		v \cdot n  = -\frac {1}{2\pi}Bg(x) & \text{on }&\partial\Omega.
	\end{array}
	\right.
\end{equation}

\subsection{Invertibility of $B$ and $A^2-\pi^2$}\label{construction of inverse}

The operators $B$ and $A^2-\pi^2$ over $L^2\left(\partial\Omega\right)$ are not invertible, for their image is a proper subset of $L^2\left(\partial\Omega\right)$, by \eqref{mean}. However, as we are about to show, they are invertible when their action is restricted to functions of zero mean value.

More precisely, introducing the space
\begin{equation*}
	L^2_0\left(\partial\Omega\right) = \set{h\in L^2\left(\partial\Omega\right)}{\oint_{\partial\Omega}h\, ds = 0},
\end{equation*}
we are now looking for an inverse of the bounded operator $B:L^2_0\left(\partial\Omega\right)\to L^2_0\left(\partial\Omega\right)$. We also consider the bounded operator $A:L^2_0\left(\partial\Omega\right)\to L^2_0\left(\partial\Omega\right)$, which is well defined by \eqref{mean}.

Notice that $A$ is a Hilbert-Schmidt operator, for the kernel of $A$ is smooth. It therefore follows that $A^2$ is compact and that the Fredholm alternative (see \cite[Theorem VI.14]{reed}) applies to $A^2-\pi^2$. More precisely, either $(A^2-\pi^2)^{-1}:L^2_0\left(\partial\Omega\right)\to L^2_0\left(\partial\Omega\right)$ exists or $A^2\varphi=\pi^2\varphi$ (i.e.\ $B^2\varphi=0$ by \eqref{PB}) has a nontrivial solution in $L^2_0\left(\partial\Omega\right)$. It turns out that the latter alternative never holds.

Indeed, suppose that there is some $g\in L^2_0\left(\partial\Omega\right)$ such that $Bg=0$. By \eqref{PB}, it holds that $\pi^2g=A^2g$ so that $g$ is smooth, for $A$ is a regularizing operator. Then, plugging $g$ into \eqref{boundary sheet} yields a velocity field $v(x)$ solving the system
\begin{equation*}
	\left\{
	\begin{array}{lcl}
		\div v  =0 & \text{in}& \Omega\cup\overline{\Omega}^c, \\
		\curl v  =0 & \text{in}& \Omega\cup\overline{\Omega}^c, \\
		v \cdot n  = 0 & \text{on }&\partial\Omega, \\
		v  \rightarrow 0 & \text{as}& x\rightarrow\infty, \\
		\oint_{\partial\Omega} v \cdot \tau ds  = 0. &&
	\end{array}
	\right.
\end{equation*}
By uniqueness, we find that $v\equiv 0$ on $\Omega\cup\overline{\Omega}^c$, whence $g=0$ by \eqref{plemelj exterior} and \eqref{plemelj interior}.

In virtue of the Fredholm alternative, this establishes that $A^2-\pi^2:L^2_0\left(\partial\Omega\right)\to L^2_0\left(\partial\Omega\right)$ always has an inverse. Using \eqref{PB}, it is now possible to produce an inverse for $B:L^2_0\left(\partial\Omega\right)\to L^2_0\left(\partial\Omega\right)$, too. Indeed, noticing that $\left(A^2-\pi^2\right)^{-1}$ commutes with $B$ because $A^2$ commutes with $B$ by virtue of \eqref{PB}, one verifies that
\begin{equation}\label{inverse B}
	B^{-1}=\left(A^2-\pi^2\right)^{-1}B:L^2_0\left(\partial\Omega\right)\to L^2_0\left(\partial\Omega\right).
\end{equation}
Observe, finally, that an inverse for $A-\pi:L^2_0\left(\partial\Omega\right)\to L^2_0\left(\partial\Omega\right)$ is readily given by
\begin{equation*}
	\left(A-\pi\right)^{-1}=\left(A^2-\pi^2\right)^{-1}\left(A+\pi\right):L^2_0\left(\partial\Omega\right)\to L^2_0\left(\partial\Omega\right).
\end{equation*}

\subsection{Representation of $u_R$ as a boundary vortex sheet}

We show now that $u_R$ can be expressed as a boundary vortex sheet \eqref{boundary sheet}.

Since the flow $v(x)$ defined by \eqref{boundary sheet} is the unique solution to \eqref{boundary sheet system}, we conclude that $v(x)$ coincides with the unique solution $u_R(x)\in C^0\left(\overline\Omega\right)\cap C^1\left(\Omega\right)$ of \eqref{eq uR} if and only if $g\in C^{0,\alpha}\left(\partial\Omega\right)$ satisfies
\begin{equation}\label{density 1}
	-\frac 1{2\pi}Bg(x) = u_R\cdot n(x)=-u_P\cdot n(x),
	\quad\text{for every }x\in\partial\Omega,
\end{equation}
and
\begin{equation}\label{density 2}
	\begin{aligned}
		\int_{\partial\Omega} g(x) dx
		= \gamma.
	\end{aligned}
\end{equation}

By \eqref{mean}, all functions in the image of $B$ have zero mean over $\partial\Omega$. Note, in particular, that the right-hand side of \eqref{density 1} has indeed zero mean because $u_P$ is solenoidal in $\mathbb{R}^2$. Moreover, by linearity, considering $g-\frac \gamma{\left|\partial\Omega\right|}$ instead of $g$, inverting \eqref{density 1}-\eqref{density 2} easily reduces to finding an inverse for $B$ over functions with zero mean value, which we have already shown to exist in \eqref{inverse B}.

 The system \eqref{density 1}-\eqref{density 2} is then solved by
\begin{equation}\label{inverse g}
	g=B^{-1}\left[2\pi u_P\cdot n - B\frac\gamma{|\partial\Omega|}\right]+\frac\gamma{|\partial\Omega|}\in L^2\left(\partial\Omega\right).
\end{equation}
It is to be emphasized that $B^{-1}B1\neq 1$, for $B^{-1}B1$ has mean zero.

\begin{remark}
	When the obstacle is the unit disk $\mathcal{C}=\overline{B(0,1)}$, the system \eqref{density 1}-\eqref{density 2} is related to the inversion of the circular Hilbert transform. This restricted geometry leads to more explicit representation formulas for $g$, because $B$ becomes a Hilbert transform and $A$ is an averaging operator (so that $Bg$ and $Ag$ are fully determined by \eqref{density 1} and \eqref{density 2}, respectively). We refer to \cite{ADLproc} for full details on this setting. In general, condition \eqref{density 2} is not easily expressed in terms of $A$ and $B$ unless $\partial\Omega$ is a circle.
\end{remark}

Observe that the above formula {\it a priori} only places the density $g$ in the space $L^2\left(\partial\Omega\right)$. Nonetheless, by \eqref{density 1} and \eqref{PB}, it is readily seen that
\begin{equation*}
	\pi^2 g=A^2g-2\pi B[u_P\cdot n],
\end{equation*}
which implies, by the aforementioned regularity properties of the operators $A$ and $B$, since $g\in L^1\left(\partial\Omega\right)$ and $u_P\cdot n\in C^\infty\left(\partial\Omega\right)$, that $g\in C^{0,\alpha}\left(\partial\Omega\right)$, for all $0<\alpha\leq 1$ (in fact, $g$ is even smoother than this).

On the whole, we have shown, for any given $0<\alpha\leq 1$, that there exists a unique $g\in C^{0,\alpha}(\partial\Omega)$ (given by \eqref{inverse g}) such that $u_R$ is expressed as a boundary vortex sheet \eqref{boundary sheet}. Thus, combining \eqref{boundary sheet} with \eqref{inverse g}, we find that
	\begin{align*}
		u_R(x)
		= & \int_{\partial\Omega} \frac{(x-y)^\perp}{|x-y|^2}
		B^{-1}\left[u_P\cdot n\right](y)dy
		\\
		& +
		\frac \gamma{2\pi|\partial\Omega|}\int_{\partial\Omega} \frac{(x-y)^\perp}{|x-y|^2}\left(1-B^{-1}B1\right)(y) dy.
	\end{align*}
Considering this representation formula for the unique harmonic vector field $H(x)$ in $\Omega$ defined by \eqref{harmonic}, i.e.\ setting $u_P\cdot n=0$ and $\gamma=1$ above, we further obtain that
\begin{equation}\label{harmonic2}
	H(x)
	=
	\frac 1{2\pi|\partial\Omega|}\int_{\partial\Omega} \frac{(x-y)^\perp}{|x-y|^2}\left(1-B^{-1}B1\right)(y) dy \quad\text{in }\Omega.
\end{equation}
It follows that
\begin{equation}\label{boundary sheet omega}
	\begin{aligned}
		u_R(x)
		& =\int_{\partial\Omega} \frac{(x-y)^\perp}{|x-y|^2}
		B^{-1}\left[u_P\cdot n\right](y)dy
		+
		\gamma H(x)
		\\
		& =\int_{\partial\Omega} \frac{(x-y)^\perp}{|x-y|^2}
		\left(A^2-\pi^2\right)^{-1}B\left[u_P\cdot n\right](y)dy
		+
		\gamma H(x).
	\end{aligned}
\end{equation}
(We strongly advise the reader to compare the above representation formula for $u_R$ on the exterior of any smooth obstacle with the corresponding much simpler representation formula (2.10) in \cite{ADLproc} for the exterior of the unit disk. The operator $H$ therein represents the circular Hilbert transform whereas $\frac{x^\perp}{2\pi|x|^2}$ is precisely the harmonic vector field for the unit disk.)

The existence of the density $g\in C^{0,\alpha}\left(\partial\Omega\right)$ satisfying conditions \eqref{density 1} and \eqref{density 2} for any suitable given data is nontrivial and at the heart of the present work, for \eqref{approx} is essentially a discretization of \eqref{boundary sheet}. The abstract construction of the inverse of $B$ in Section \ref{construction of inverse} through the Fredlhom alternative is not suitable for a discretization procedure, though. In order to use the invertibility \eqref{density 1}-\eqref{density 2} to solve system \eqref{point vortex}, we need now to refine our understanding of the operators $A$ and $B$ and their respective spectra.

\subsection{Kernels of $B$ and $A-\pi$}\label{kernels}

Further observe, by \eqref{system interior}, that the right-hand side of \eqref{harmonic2} also defines the unique solution (which is trivially zero) to
\begin{equation*}
	\left\{
	\begin{array}{lcl}
		\div v  =0 & \text{in}& \overline{\Omega}^c, \\
		\curl v  =0 & \text{in}& \overline{\Omega}^c, \\
		v \cdot n  = 0 & \text{on }&\partial\Omega,
	\end{array}
	\right.
\end{equation*}
whereby, by \eqref{tangent plemelj}, we find the relations
	\begin{align*}
		\left|\partial\Omega\right|H\cdot\tau(x) = & \frac 1{2\pi}A\left[1-B^{-1}B1\right](x)+\frac 12\left(1-B^{-1}B1\right)(x),
		\\
		0 = & \frac 1{2\pi}A\left[1-B^{-1}B1\right](x)-\frac 12\left(1-B^{-1}B1\right)(x),
	\end{align*}
for all $x\in\partial\Omega$ (we emphasize here that the values of $H$ on $\partial\Omega$ are given by its limiting values from $\Omega$), which are equivalent to
\begin{equation}\label{harmonic4}
	\begin{aligned}
		(1-B^{-1}B1) & = \left|\partial\Omega\right|H\cdot\tau,
		\\
		A[H\cdot\tau] & = \pi H\cdot\tau,
	\end{aligned}
\end{equation}
on $\partial\Omega$.

We conclude that $H\cdot\tau$ lies in the kernels of $B$ and $A-\pi$, and that one has the representations (using \eqref{harmonic2} and then \eqref{plemelj exterior}, again)
\begin{equation}\label{harmonic3}
	\begin{aligned}
		H(x)
		& =
		\frac 1{2\pi}\int_{\partial\Omega} \frac{(x-y)^\perp}{|x-y|^2}H\cdot\tau(y) dy && \text{in }\Omega,
		\\
		H(x)
		& =
		\frac 1{2\pi}\int_{\partial\Omega} \frac{(x-y)^\perp}{|x-y|^2}H\cdot\tau(y) dy + \frac 12 \tau(x)H\cdot\tau(x) && \text{on }\partial\Omega.
	\end{aligned}
\end{equation}

Finally, since any $g\in L^2\left(\partial\Omega\right)$ satisfies $g-\gamma H\cdot\tau\in L^2_0\left(\partial\Omega\right)$ for some appropriate $\gamma\in\mathbb{R}$ and both operators $A-\pi$ and $B$ are invertible over $L^2_0\left(\partial\Omega\right)$, we deduce that the kernels of $A-\pi:L^2\left(\partial\Omega\right)\to L^2_0\left(\partial\Omega\right)$ and $B:L^2\left(\partial\Omega\right)\to L^2_0\left(\partial\Omega\right)$ coincide exactly with the span of $H\cdot\tau$.

\subsection{Spectrum of $A$}

It is now possible to deduce some simple spectral properties for $A:L^2\left(\partial\Omega\right)\to L^2\left(\partial\Omega\right)$ from the preceding developments. First, since $A$ is compact, by the classical Riesz-Schauder theorem (see \cite[Theorem VI.15]{reed} or \cite[Chapter~X, \S~5]{yosida}), we know that its spectrum $\sigma(A)$ is at most countable with no limit points except, possibly, at zero. Moreover, $\sigma(A)\setminus\{0\}$ is composed solely of eigenvalues with finite multiplicity (i.e.\ corresponding eigenspaces are finite dimensional).

We also consider the spectrum of $A:L^2_0\left(\partial\Omega\right)\to L^2_0\left(\partial\Omega\right)$ which we distinguish from $\sigma(A)$ by denoting it by $\sigma_0(A)$. Clearly, if $\lambda\in \mathbb{C}\setminus\sigma(A)$, then $\lambda\neq\pi$ (for $\pi-A$ has a nontrivial kernel; see \eqref{harmonic4}), the operator $\lambda - A$ is invertible and, by \eqref{mean}, its inverse leaves $L^2_0\left(\partial\Omega\right)$ invariant, whereby $\lambda\in \mathbb{C}\setminus\left(\sigma_0(A)\cup\left\{\pi\right\}\right)$. It follows that $\sigma_0(A)\cup\left\{\pi\right\}\subset\sigma(A)$. On the other hand, if $\lambda\in \mathbb{C}\setminus\left(\sigma_0(A)\cup\left\{\pi\right\}\right)$, then $\lambda-A$ has a bounded inverse over functions with mean zero. It is then possible to extend this inverse to functions with non-zero average with the definition
\begin{equation*}
	\left(\lambda-A\right)^{-1}\varphi
	= \left(\lambda-A\right)^{-1}\left(\varphi-\gamma H\cdot\tau\right) + \frac{\gamma}{\lambda-\pi}H\cdot\tau,
	\quad\text{where }\gamma=\oint_{\partial\Omega}\varphi dx,
\end{equation*}
and one verifies that this produces a well defined inverse which is bounded over $L^2\left(\partial\Omega\right)$, whereby $\lambda\in \mathbb{C}\setminus\sigma(A)$ and therefore $\sigma(A)\subset \sigma_0(A)\cup\{\pi\}$. On the whole, we conclude that $\sigma(A)=\sigma_0(A)\cup\{\pi\}$.

We have already identified, in Section \ref{kernels}, the span of $H\cdot\tau$ as the eigenspace corresponding to the eigenvalue $\pi\in\sigma(A)$. In particular, since $H\cdot\tau$ does not have mean zero over $\partial\Omega$, we see that $\pi\notin\sigma_0(A)$.

Suppose now that $\varphi\in L^2\left(\partial\Omega\right)$ satisfies $A\varphi=-\pi\varphi$. Then, by \eqref{PB}, it holds that $B^2\varphi=0$, whence $B\varphi$ both belongs to the kernel of $B$ and has mean zero by \eqref{mean}, which implies that $B\varphi=0$ (recall that the mean of $H\cdot\tau$ is non-zero). We conclude that $\varphi$ also belongs to the kernel of $B$ and that it is therefore a constant multiple of $H\cdot\tau$. Since we have already shown that $A[H\cdot\tau]=\pi H\cdot\tau$, this establishes that $-\pi$ is not an eigenvalue of $A$.

Finally, by \eqref{PB}, if $\varphi\in L^2_0\left(\partial\Omega\right)$ is an eigenvector of $A$ for some eigenvalue $\lambda\in\mathbb{C}\setminus\left\{\pm\pi\right\}$, we find that $B\varphi\in L^2_0\left(\partial\Omega\right)$ is an eigenvector of $A$ for the eigenvalue $-\lambda$, whence $\sigma_0(A)=-\sigma_0(A)$.

In fact, it is well-known that the spectral radius of $A$ is no larger than $\pi$ (see \cite[Chapter~XI, \S~11]{kellogg}). For convenience of the reader, though, we provide here a short argument showing that any eigenvalue $\lambda\in\mathbb{C}$ is actually real and has modulus bounded by $\pi$. To this end, consider an eigenvector $g\in L^2\left(\partial\Omega\right)$ of $A$ corresponding to some eigenvalue $\lambda\in\mathbb{C}\setminus\left\{0,\pi\right\}$ (note that, since $A$ is regularizing, $g$ is actually smooth and that, by \eqref{mean}, $g$ has mean zero). Then, considering the velocity field given by \eqref{boundary sheet} and defining $h(x)=\frac 1{2\pi}\int_{\partial\Omega}\log\left(|x-y|\right)g(y)dy$, we compute that, employing \eqref{plemelj exterior}, \eqref{plemelj interior} and that $v(x)=v(x)-\frac 1{2\pi}\frac{x^\perp}{|x|^2}\int_{\partial\Omega}gdx=\mathcal{O}\left(|x|^{-2}\right)$ for densities with mean zero,
	\begin{align*}
		\left(\lambda - \pi\right) & \int_{\Omega}|v(x)|^2dx
		+
		\left(\lambda + \pi\right)\int_{\overline\Omega^c}|v(x)|^2dx
		\\
		= &
		\left(\lambda - \pi\right)\int_{\Omega}
		v(x) \cdot \nabla^\perp\overline{h(x)}dx
		+
		\left(\lambda + \pi\right)\int_{\overline\Omega^c}
		v(x) \cdot \nabla^\perp\overline{h(x)}dx
		\\
		= &
		\left(\lambda - \pi\right)\int_{\Omega}
		\curl\left(v(x) \overline{h(x)}\right)dx
		+
		\left(\lambda + \pi\right)\int_{\overline\Omega^c}
		\curl\left(v(x) \overline{h(x)}\right)dx
		\\
		= &
		\frac{\pi-\lambda}{2\pi} \int_{\partial\Omega}
		\left(Ag(x)+\pi g(x)\right) \overline{h(x)} dx
		+
		\frac{\pi+\lambda}{2\pi} \int_{\partial\Omega}
		\left(Ag(x)-\pi g(x)\right) \overline{h(x)} dx
		\\
		& + \lim_{R\to\infty}\left(\lambda-\pi\right)
		\int_{\partial B(0,R)} v(x)\cdot\tau(x) \overline{h(x)} dx = 0,
	\end{align*}
where the tangent vector $\tau(x)$ on $\partial B(0,R)$ points in the counterclockwise direction. Thus, since $v\neq 0$ (otherwise $g=0$ by \eqref{plemelj exterior}-\eqref{plemelj interior}), we conclude that the origin $0\in\mathbb{C}$ can be expressed as a convex combination of $\lambda-\pi$ and $\lambda+\pi$. Some elementary geometry implies then that $\lambda\in \left[-\pi,\pi\right]\subset\mathbb{C}$.

On the whole, we conclude that
\begin{equation*}
	\sigma_0(A)=-\sigma_0(A)\subset \left(-\pi,\pi\right)
	\quad\text{and}\quad
	\sigma(A)=\sigma_0(A)\cup\left\{\pi\right\}\subset \left(-\pi,\pi\right].
\end{equation*}
In particular, by Gelfand's formula for the spectral radius (see \cite[Chapter~VIII, \S~2]{yosida}, for instance), we obtain that
\begin{equation}\label{gelfand}
	\begin{aligned}
		\lim_{k\to\infty}\left\|A^k\right\|_{\mathcal{L}\left(L^2\right)}^\frac 1k
		& =\inf_{k\geq 1}\left\|A^k\right\|_{\mathcal{L}\left(L^2\right)}^\frac 1k
		=\pi,
		\\
		\lim_{k\to\infty}\left\|A^k\right\|_{\mathcal{L}\left(L^2_0\right)}^\frac 1k
		& =\inf_{k\geq 1}\left\|A^k\right\|_{\mathcal{L}\left(L^2_0\right)}^\frac 1k
		<\pi,
	\end{aligned}
\end{equation}
and the inverse of $A^2-\pi^2:L^2_0\left(\partial\Omega\right)\to L^2_0\left(\partial\Omega\right)$ is therefore given by the Neumann series
\begin{equation*}
	\left(A^2-\pi^2\right)^{-1}
	=
	-\pi^{-2}\sum_{n=0}^\infty \left(\frac A\pi\right)^{2n},
\end{equation*}
which is absolutely convergent in $\mathcal{L}\left(L^2_0\right)$ for, by \eqref{gelfand}, there is some $\eps>0$ such that $\left\|\left(\frac A\pi\right)^{k}\right\|_{\mathcal{L}\left(L^2_0\right)}\leq \left(1-\eps\right)^{k}$ for large $k$.

Contrary to the abstract method of contruction of inverses based on the Fredholm alternative from Section \ref{construction of inverse}, the present spectral approach allows us to deduce precise bounds on the inverses by quantifying the spectral gap of $A$ at $\pm\pi$. The ensuing estimates are robust and well adapted for discretization procedures, which will be crucial in the remainder of this work.

\subsection{Other representations of $u_R$}

It turns out that there is yet another convenient representation formula for the flow $u_R$, which is a variant of the boundary vortex sheet \eqref{boundary sheet}.

More precisely, we claim now that in the exterior of a given obstacle, $u_R$ can also be expressed as:
\begin{equation}\label{boundary sheet 2}
	\begin{aligned}
		w(x) = & \frac 1{2\pi}\int_{\partial\Omega} \frac{x-y}{|x-y|^2}h(y)dy
		+ \gamma H(x)
		\\
		= & -\frac 1{2\pi}\left(A^*[nh]+B^*[\tau h]\right)(x) + \gamma H(x)\in C^\infty\left(\mathbb{R}^2\setminus\partial\Omega\right),
	\end{aligned}
\end{equation}
for some suitable $h\in C^{0,\alpha}\left(\partial\Omega\right)$, with $0<\alpha\leq 1$. Recall that $H(x)$ is the harmonic vector field uniquely defined in $\Omega$ by \eqref{harmonic}, which we extend into $\overline\Omega^c$ by zero so that $H(x)$ is represented by \eqref{harmonic2} in $\Omega\cup\overline\Omega^c$.

As before, the theory of Cauchy integrals instructs us that, for a smooth boundary $\partial\Omega$ and for any $h\in C^{0,\alpha}\left(\partial\Omega\right)$, the flow $w$ is continuous up to the boundary $\partial\Omega$, that is $w\in C\left(\overline \Omega\right)\cup C\left(\Omega^c\right)$, and that the limiting values of $w$ on $\partial\Omega$ are given by the Plemelj formulas \eqref{plemelj}. Hence, we deduce that
\begin{equation*}
	\lim_{\substack{x\rightarrow x_0\in\partial\Omega \\ x\in\Omega}} w(x)=
	\frac 1{2\pi}\int_{\partial\Omega} \frac{x_0-y}{|x_0-y|^2}h(y)dy +
	\frac 12 n(x_0)h(x_0)+ \gamma H(x_0),
\end{equation*}
and
\begin{equation*}
	\lim_{\substack{x\rightarrow x_0\in\partial\Omega \\ x\in\overline{\Omega}^c}} w(x)=
	\frac 1{2\pi}\int_{\partial\Omega} \frac{x_0-y}{|x_0-y|^2}h(y)dy -
	\frac 12 n(x_0)h(x_0).
\end{equation*}
Again, we emphasize that the values of $H$ on $\partial\Omega$ are given here by its limiting values from $\Omega$ so that the representation formulas \eqref{harmonic3} are valid.

Therefore, we conclude that the flow $w(x)$ given by \eqref{boundary sheet 2} defines the unique solution $u_R(x)\in C^0\left(\overline\Omega\right)\cap C^1\left(\Omega\right)$ of \eqref{eq uR} if and only if $h\in C^{0,\alpha}\left(\partial\Omega\right)$ satisfies
\begin{equation}\label{density 3}
	\begin{aligned}
		\frac 1{2\pi}
		\left(A+\pi\right)h(x)
		& =
		\frac 1{2\pi}\int_{\partial\Omega} \frac{x-y}{|x-y|^2}\cdot n(x)h(y)dy + \frac 12 h(x)
		\\
		& =u_R\cdot n(x)=-u_P\cdot n(x),
		\quad\text{for every }x\in\partial\Omega.
	\end{aligned}
\end{equation}
Provided \eqref{density 3} is verified and using \eqref{mean}, note that it necessarily holds
	\begin{align*}
		\int_{\partial\Omega} h(x) dx
		=\frac 1{2\pi}
		\int_{\partial\Omega} \left(A+\pi\right)h(x) dx
		= -\int_{\partial\Omega} u_P\cdot n(x)dx=0,
	\end{align*}
and that the circulation condition
	\begin{align*}
		\int_{\partial\Omega} u_R\cdot\tau(x)dx
		= &
		\int_{\partial\Omega}\left(\frac 1{2\pi}\int_{\partial\Omega} \frac{x-y}{|x-y|^2}h(y)dy +
		\frac 12 n(x)h(x) + \gamma H(x) \right)\cdot\tau(x)dx
		\\
		= & \frac 1{2\pi}\int_{\partial\Omega}Bh(x)dx + \gamma=\gamma,
	\end{align*}
is automatically satisfied.

The existence of such a density $h\in C^{0,\alpha}\left(\partial\Omega\right)$ satisfying \eqref{density 3} for any suitable given data is nontrivial (again, we refer to \cite{ADLproc} for a treatment of the simpler case of the unit disk). However, in view of the above spectral analysis of the operator $A$, it is readily seen that $A+\pi:L^2_0\left(\partial\Omega\right)\to L^2_0\left(\partial\Omega\right)$ has an inverse given by the Neumann series
\begin{equation*}
	\left(A+\pi\right)^{-1}
	=
	\pi^{-1}\sum_{n=0}^\infty (-1)^n\left(\frac A\pi\right)^{n},
\end{equation*}
which is absolutely convergent in $\mathcal{L}\left(L^2_0\right)$. In fact, observing that the spectrum of $A-\pi$ is contained in $\left(-2\pi,0\right]$, i.e.\ $\sigma\left(A-\pi\right)\subset \left(-2\pi,0\right]$, yields, by Gelfand's formula again, the precise estimate
\begin{equation}\label{gelfand 2}
	\lim_{k\to\infty}\left\|\left(A-\pi\right)^k\right\|_{\mathcal{L}\left(L^2\right)}^\frac 1k
	=\inf_{k\geq 1}\left\|\left(A-\pi\right)^k\right\|_{\mathcal{L}\left(L^2\right)}^\frac 1k
	< 2\pi,
\end{equation}
which implies that $A+\pi:L^2\left(\partial\Omega\right)\to L^2\left(\partial\Omega\right)$ also has a bounded inverse given by the Neumann series
\begin{equation*}
	\left(A+\pi\right)^{-1}
	=
	\frac 1{2\pi}\sum_{n=0}^\infty \left(\frac {\pi -A}{2\pi}\right)^{n},
\end{equation*}
which is absolutely convergent in $\mathcal{L}\left(L^2\right)$ for, by \eqref{gelfand 2}, there is some $\eps>0$ such that $\left\|\left(\frac {\pi -A}{2\pi}\right)^{k}\right\|_{\mathcal{L}\left(L^2\right)}\leq \left(1-\eps\right)^{k}$ for large $k$.

Therefore, it is now readily seen that the equation \eqref{density 3} is uniquely solved by
\begin{equation*}
	h=-2\pi\left(A+\pi\right)^{-1}[u_P\cdot n]\in C^\infty\left(\partial\Omega\right),
\end{equation*}
whereby, in view of \eqref{boundary sheet 2}, we obtain the following representation formula on the exterior of a smooth obstacle:
\begin{equation}\label{boundary sheet omega 2}
	u_R(x) = -\int_{\partial\Omega} \frac{x-y}{|x-y|^2}\left(A+\pi\right)^{-1}[u_P\cdot n](y)dy
	+ \gamma H(x).
\end{equation}

\begin{remark}
For any velocity field $H_*$ satisfying
	\begin{equation}\label{harmonic5}
		\left\{
		\begin{array}{lcl}
			\div H_*  =0 & \text{in}& \Omega, \\
			\curl H_*  =0 & \text{in}& \Omega, \\
			H_*  \rightarrow 0 & \text{as}& x\rightarrow\infty, \\
			\oint_{\partial\Omega} H_* \cdot \tau ds  = 1, &&
		\end{array}
		\right.
	\end{equation}
we note, by the Plemelj formulas \eqref{plemelj} and by \eqref{mean}, that
\[
\tilde H(x) := H_*(x) - \int_{\partial \Omega} \frac{x-y}{|x-y|^2} (A+\pi)^{-1}[H_*\cdot n](y)dy
\]
is divergence and curl free in $\Omega$, goes to $0$ when $x\to \infty$ and verifies
\begin{gather*}
\tilde H \cdot n = H_*\cdot n-(A+\pi)(A+\pi)^{-1}[H_*\cdot n]=0 \text{ on  }\partial \Omega,\\
\oint_{\partial \Omega} \tilde H \cdot \tau ds = \oint_{\partial \Omega}  H_* \cdot \tau ds-\oint_{\partial \Omega} B (A+\pi)^{-1}[H_*\cdot n] =1.
\end{gather*}
By uniqueness in \eqref{harmonic}, we deduce $\tilde H=H$. This can be used to replace the harmonic vector field $H(x)$ in \eqref{boundary sheet omega 2} with more convenient expressions, thereby yielding a variant formula:
	\begin{equation}\label{boundary sheet omega 4}
		u_R(x) = -\int_{\partial\Omega} \frac{x-y}{|x-y|^2}\left(A+\pi\right)^{-1}\left[\left(u_P+\gamma H_*\right)\cdot n\right](y)dy
		+ \gamma H_*(x).
	\end{equation}
	For instance, one may consider the velocity field $H_*(x)=\frac{(x-x_*)^\perp}{2\pi |x-x_*|^2}=K_{\R^2}[\delta_{x_*}]$, for any given $x_*\in {\overline \Omega^c}$.
\end{remark}

It then follows, by comparing \eqref{boundary sheet omega 2} with \eqref{boundary sheet omega} and by uniqueness of solutions to system \eqref{eq uR}, that
	\begin{align*}
		-\int_{\partial\Omega} \frac{x-y}{|x-y|^2} & \left(A+\pi\right)^{-1}[u_P\cdot n](y)dy
		\\
		& =
		\int_{\partial\Omega} \frac{(x-y)^\perp}{|x-y|^2}
		\left(A^2-\pi^2\right)^{-1}B\left[u_P\cdot n\right](y)dy,
		\quad\text{for every }x\in\Omega,
	\end{align*}
whence we infer that, replacing $u_P\cdot n$ by $B\left(A-\pi\right)\varphi$ in view of the arbitrariness of zero-mean boundary data in \eqref{eq uR} and using the Poincar\'e--Bertrand identities \eqref{PB},
\begin{equation}\label{vortex identity 0}
	\begin{aligned}
		\int_{\partial\Omega} \frac{x-y}{|x-y|^2}B\varphi(y)dy
		= &
		-\int_{\partial\Omega} \frac{x-y}{|x-y|^2}\left(A+\pi\right)^{-1}B\left(A-\pi\right)\varphi(y)dy
		\\
		= &
		\int_{\partial\Omega} \frac{(x-y)^\perp}{|x-y|^2}
		\left(A^2-\pi^2\right)^{-1}B^2\left(A-\pi\right)\varphi(y)dy
		\\
		= &
		\int_{\partial\Omega} \frac{(x-y)^\perp}{|x-y|^2}
		\left(A-\pi\right)\varphi(y)dy,
		\quad \text{for every }x\in\Omega.
	\end{aligned}
\end{equation}
By adjointness (see \eqref{adjoint A} and \eqref{adjoint B}), we further obtain that
	\begin{align*}
		\int_{\partial\Omega} \int_{\partial\Omega} & \frac{y-z}{|y-z|^2} \cdot \tau(z) \frac{x-z}{|x-z|^2}dz\varphi(y)dy
		\\
		& =
		\int_{\partial\Omega} \int_{\partial\Omega}\frac{y-z}{|y-z|^2}\cdot n(z)\frac{(x-z)^\perp}{|x-z|^2}
		dz\varphi(y)dy
		+\pi
		\int_{\partial\Omega} \frac{(x-y)^\perp}{|x-y|^2}
		\varphi(y)dy,
	\end{align*}
and, thus, by the arbitrariness of $\varphi$, we deduce the identity
\begin{equation}\label{vortex identity}
	\begin{aligned}
		\pi
		\frac{(x-y)^\perp}{|x-y|^2}
		=&
		\int_{\partial\Omega}\frac{y-z}{|y-z|^2} \cdot \tau(z) \frac{x-z}{|x-z|^2}dz\\
		&-
		\int_{\partial\Omega}\frac{y-z}{|y-z|^2}\cdot n(z)\frac{(x-z)^\perp}{|x-z|^2}
		dz,
		\quad
		\forall (x,y)\in\Omega\times\partial\Omega,
	\end{aligned}
\end{equation}
which will be useful later on.

Finally, note that, combining \eqref{vortex identity} (or \eqref{vortex identity 0}) with \eqref{boundary sheet omega}, we obtain yet another convenient representation formula
\begin{equation}\label{vortex identity 2}
\begin{aligned}
 	u_R(x)
	=&\frac 1\pi\int_{\partial\Omega} \frac{(x-y)^\perp}{|x-y|^2}
	AB^{-1}\left[u_P\cdot n\right](y)dy\\
	&-
	\frac 1\pi\int_{\partial\Omega} \frac{x-y}{|x-y|^2}
	\left[u_P\cdot n\right](y)dy
	+
	\gamma H(x),
\end{aligned}
\end{equation}
whereas combining \eqref{vortex identity} (or \eqref{vortex identity 0}) with \eqref{boundary sheet omega 2} yields
	\begin{align*}
		u_R(x) = & -\frac 1\pi \int_{\partial\Omega}\frac{(x-z)^\perp}{|x-z|^2}
			B\left(A+\pi\right)^{-1}\left[u_P\cdot n\right](z)dz
			\\
			& -\frac 1\pi
			\int_{\partial\Omega} \frac{x-z}{|x-z|^2}
			A\left(A+\pi\right)^{-1}\left[u_P\cdot n\right](z)dz
		+ \gamma H(x).
	\end{align*}

\begin{remark}
	A variant representation formula is obtained by combining \eqref{vortex identity} (or \eqref{vortex identity 0}) with \eqref{boundary sheet omega 4} instead of \eqref{boundary sheet omega 2}:
	\begin{equation}\label{vortex identity 3}
		\begin{aligned}
			u_R(x) = & -\frac 1\pi \int_{\partial\Omega}\frac{(x-z)^\perp}{|x-z|^2}
				B\left(A+\pi\right)^{-1}\left[\left(u_P+\gamma H_*\right)\cdot n\right](z)dz
				\\
				& -\frac 1\pi
				\int_{\partial\Omega} \frac{x-z}{|x-z|^2}
				A\left(A+\pi\right)^{-1}\left[\left(u_P+\gamma H_*\right)\cdot n\right](z)dz
			+ \gamma H_*(x).
		\end{aligned}
	\end{equation}
\end{remark}

\begin{remark}
As previously explained, our goal is to justify that $u_{\rm app}^N$, defined by \eqref{approx}, is a good discretization of the formulation \eqref{boundary sheet omega}. In fact, it is also possible to discretize \eqref{boundary sheet omega 2} (or \eqref{boundary sheet omega 4}) which provides another approximation of $u_{R}$. We explore this alternative approach in Section \ref{charges}.
\end{remark}

\section{Solving system \eqref{point vortex} and the discrete Poincar\'e--Bertrand formula}\label{section discrete}

In this section, we explain how system \eqref{point vortex} can be uniquely solved as soon as $N$ is sufficiently large provided $\{ x_i^N \}$ and $\{ \tilde x_i^N \} $ are well distributed. This will be achieved by employing a strategy inspired by the inversion of system \eqref{density 1}-\eqref{density 2}.

Considering the parameters $\{ s_i^N \}$ and $\{ \tilde s	_i^N \} $ associated to $\{ x_i^N \}$ and $\{ \tilde x_i^N \} $ (see \eqref{xi}-\eqref{tildexi}), the system \eqref{point vortex} of $N$ equations can be recast as
\begin{equation}\label{point toy}
	\begin{aligned}
		& \frac1{N}\sum_{j=1}^N \gamma_{j}^N \frac{l\left(\tilde s_{i}^N\right) - l\left(s_{j}^N\right)}
		{\left|l\left(\tilde s_{i}^N\right) - l\left(s_{j}^N\right)\right|^2}\cdot \tau\left(l\left(\tilde s_{i}^N\right)\right)
		= f\left(\tilde s_i^N\right), \ \text{for all }i=1,\dots, N-1,\\
		& \frac 1N \sum_{i=1}^N \gamma_{i}^N = \gamma,
	\end{aligned}
\end{equation}
where $\gamma^N=(\gamma_{1}^N,\dots, \gamma_{N}^N)\in\mathbb{R}^N$ is the unknown and $f(s)=2\pi[u_P\cdot n]\left(l(s)\right)$, for all $s\in\left[0,\left|\partial\Omega\right|\right]$. Loosely speaking, solving system \eqref{point toy} amounts to inverting a discrete version of the operator $B$ introduced in \eqref{AB}. Indeed, \eqref{point toy} clearly is a discretization of \eqref{density 1}-\eqref{density 2}.

From now on, we will also conveniently denote the matrices:
	\begin{align*}
		A_{N} & := 
		\left(
		\frac{l\left(\tilde s_{i}^N\right) - l\left(s_{j}^N\right)}
		{\left|l\left(\tilde s_{i}^N\right) - l\left(s_{j}^N\right)\right|^2}\cdot n\left(l\left(\tilde s_{i}^N\right)\right)
		\right)_{1\leq i, j\leq N},
		\\
		\tilde A_{N} & := 
		\left(
		\frac{l\left(s_{i}^N\right) - l\left(\tilde s_{j}^N\right)}
		{\left|l\left(s_{i}^N\right) - l\left(\tilde s_{j}^N\right)\right|^2}\cdot n\left(l\left(s_{i}^N\right)\right)
		\right)_{1\leq i, j\leq N},
		\\
		B_{N} & := 
		\left(
		\frac{l\left(\tilde s_{i}^N\right) - l\left(s_{j}^N\right)}
		{\left|l\left(\tilde s_{i}^N\right) - l\left(s_{j}^N\right)\right|^2}\cdot \tau\left(l\left(\tilde s_{i}^N\right)\right)
		\right)_{1\leq i, j\leq N},
		\\
		\tilde B_{N} & := 
		\left(
		\frac{l\left(s_{i}^N\right) - l\left(\tilde s_{j}^N\right)}
		{\left|l\left(s_{i}^N\right) - l\left(\tilde s_{j}^N\right)\right|^2}\cdot \tau\left(l\left(s_{i}^N\right)\right)
		\right)_{1\leq i, j\leq N},
	\end{align*}
and we will make use of the following notations for $z\in \R^N$:
	\begin{align*}
		\| z\|_{\ell^p} & := \Big(\frac1N \sum_{i=1}^N |z_{i}|^p \Big)^{1/p}, \quad \text{for any }p\in [1,\infty),\\
		\|z\|_{\ell^\infty} & := \max_{i=1,\dots,N}  |z_{i}|,\\
		\langle z \rangle & := \frac1N\sum_{i=1}^N z_{i}.
	\end{align*}
Note that, with this normalization of the norms, we have:
\[
\|z\|_{\ell^p}\leq \| z \|_{\ell^q}, \text{ for any } 1\leq p\leq q \leq \infty.
\]

\subsection{Boundedness of discretized operators}

For a uniformly distributed mesh \eqref{mesh}, notice that, by odd symmetry of the cotangent function,
\begin{equation}\label{perfect distri}
	\sum_{1\leq j \leq N} \cot\left(\frac{\left(\tilde\theta_{i}^N-\theta_{j}^N\right)\pi}{\left|\partial\Omega\right|}\right)=0
	\quad\text{and}\quad
	\sum_{1\leq j \leq N} \cot\left(\frac{\left(\tilde\theta_{j}^N-\theta_{i}^N\right)\pi}{\left|\partial\Omega\right|}\right) =0,
\end{equation}
and
\begin{equation}\label{perfect distri 2}
	\sum_{\substack{1\leq j \leq N \\ j\neq i}} \cot\left(\frac{\left(\theta_{i}^N-\theta_{j}^N\right)\pi}{\left|\partial\Omega\right|}\right)=0,
\end{equation}
for each $i=1,\ldots,N$. In fact, it can be shown that the only possible mesh satisfying \eqref{perfect distri} and $\theta_1^N=0$ is necessarily given by \eqref{mesh}. Indeed, suppose that some other given mesh $\left\{\phi_i^N\right\}$ and $\left\{\tilde\phi_i^N\right\}$, with $\phi_1^N=0$, also satisfies \eqref{perfect distri}. Then, suitable linear combinations of \eqref{perfect distri} yield that
	\begin{multline*}
		\sum_{1\leq i,j \leq N} 
		\left(
		\left(\tilde\phi_{i}^N-\phi_{j}^N\right)
		-
		\left(\tilde\theta_{i}^N-\theta_{j}^N\right)
		\right)
		\\
		 \times
		\left(
		\cot\left(\frac{\left(\tilde\theta_{i}^N-\theta_{j}^N\right)\pi}{\left|\partial\Omega\right|}\right)
		-
		\cot\left(\frac{\left(\tilde\phi_{i}^N-\phi_{j}^N\right)\pi}{\left|\partial\Omega\right|}\right)
		\right)
		=0,
	\end{multline*}
whence, using that $(b-a)\left(\cot a - \cot b\right) = (b-a)\int_a^b\frac 1{\sin^2 x}dx\geq (b-a)^2$, for any $0<a,b<\pi$ (or $-\pi<a,b<0$),
\begin{equation*}
	\tilde\phi_{i}^N-\phi_{j}^N
	=
	\tilde\theta_{i}^N-\theta_{j}^N
	\quad\text{for all }1\leq i,j\leq N.
\end{equation*}
Further using that $\theta_1^N=\phi_1^N$, we conclude that $\theta_i^N=\phi_i^N$ and $\tilde \theta_i^N=\tilde \phi_i^N$, for all $1\leq i\leq N$.

The cancellations embodied in identities \eqref{perfect distri} and \eqref{perfect distri 2} are related with their continuous counterpart $\int_{0}^{\pi} \cot\left(\theta-\tilde\theta\right) d\theta =0$, for any $\tilde\theta\in\mathbb{R}$. As the oddness of the cotangent function plays a crucial role to define Cauchy's principal value, the symmetry of the points $(\theta_{i}^N,\tilde \theta_{i}^N)$ is important to make sure that singular integrals defined in the sense of Cauchy's principal value are suitably approximated by their corresponding discretization.

The first result in this section is technical and shows that a well distributed mesh retains sufficient approximate symmetry to satisfy an approximation of \eqref{perfect distri}. This property will be important to ensure that singular integrals are well approximated by discretizations corresponding to well distributed meshes.

\begin{lemma}\label{technical cot}
	For any $N\geq 2$, consider a well distributed mesh $(s_{1}^N,\dots , s_{N}^N)\in \left[0,\left|\partial\Omega\right|\right)^N$, $(\tilde s_{1}^N, \dots , \tilde s_{N}^N)\in \left[0,\left|\partial\Omega\right|\right)^N$. Then, as $N\to\infty$,
		\begin{align*}
			& \max_{1\leq i\leq N}\left|\sum_{1\leq j \leq N} \cot\left(\frac{\left(\tilde s_{i}^N-s_{j}^N\right)\pi}{\left|\partial\Omega\right|}\right)\right|=\mathcal{O}\left(N^{-\kappa+1}\right)
			\\
			\text{and }
			& \max_{1\leq i\leq N}\left|\sum_{1\leq j \leq N} \cot\left(\frac{\left(\tilde s_{j}^N-s_{i}^N\right)\pi}{\left|\partial\Omega\right|}\right)\right|=\mathcal{O}\left(N^{-\kappa+1}\right).
		\end{align*}
\end{lemma}

\begin{proof}
	Note first that, for all $1\leq i,j\leq N$ and large enough $N$,
		\begin{align*}
			\left|\tilde s_{i}^N-s_{j}^N\right|
			& \geq \left|\tilde \theta_{i}^N-\theta_{j}^N\right|
			-\left|\tilde s_{i}^N-\tilde\theta_{i}^N\right|
			-\left|s_{j}^N-\theta_{j}^N\right|
			= \frac{\left|i-j+\frac 12\right| \left|\partial\Omega\right|}{N}-\mathcal{O}\left(N^{-3}\right)
			\\
			& \geq \frac{\left|i-j+\frac 12\right| \left|\partial\Omega\right|}{N}-\frac{\left|\partial\Omega\right|}{4N}
			\geq \frac{\left|i-j+\frac 12\right| \left|\partial\Omega\right|}{2N},
		\end{align*}
	and
		\begin{align*}
			\left|\tilde s_{i}^N-s_{j}^N\right|
			& \leq \left|\tilde \theta_{i}^N-\theta_{j}^N\right|
			+\left|\tilde s_{i}^N-\tilde\theta_{i}^N\right|
			+\left|s_{j}^N-\theta_{j}^N\right|
			= \frac{\left|i-j+\frac 12\right| \left|\partial\Omega\right|}{N}+\mathcal{O}\left(N^{-3}\right)
			\\
			& \leq \frac{\left|i-j+\frac 12\right| \left|\partial\Omega\right|}{N}+\frac{\left|\partial\Omega\right|}{4N}
			\leq
			\frac{\left|\partial\Omega\right|}{2} +
			\frac{\left|i-j+\frac 12\right| \left|\partial\Omega\right|}{2N}.
		\end{align*}
	Therefore, by \eqref{mesh2} and the mean value theorem, defining the open interval
	\begin{equation*}
		I_{ij}=
		\left( \pi\frac{\left|i-j+\frac 12\right|}{2N},
		\pi\left(\frac{1}{2} +
		\frac{\left|i-j+\frac 12\right|}{2N}\right)\right),
	\end{equation*}
	we find that
	\begin{equation}\label{cot mean value}
		\begin{aligned}
			\left|\cot\left(\frac{\left(\tilde s_{i}^N-s_{j}^N\right)\pi}{\left|\partial\Omega\right|}\right)
			-
			\cot\left(\frac{\left(\tilde \theta_{i}^N-\theta_{j}^N\right)\pi}{\left|\partial\Omega\right|}\right)\right|
			\hspace{-50mm}&
			\\
			& \leq
			\frac \pi{\left|\partial\Omega\right|}
			\left|\tilde s_{i}^N-s_{j}^N - \tilde \theta_{i}^N+\theta_{j}^N\right|
			\sup_{x\in I_{ij}}
			\frac 1{\sin^2x}
			\\
			& \leq \mathcal{O}\left(N^{-(\kappa+1)}\right)\times\max\left\{\frac{N^2}{\left|i-j+\frac 12\right|^2},
			\frac{N^2}{\left(N-\left|i-j+\frac 12\right|\right)^2}\right\}
			\\
			& = \mathcal{O}\left(N^{-(\kappa+1)}\right)\times \mathcal{O}(N^2)=\mathcal{O}\left(N^{-\kappa+1}\right).
		\end{aligned}
	\end{equation}
	
	Then, summing over $1\leq i\leq N$ or $1\leq j\leq N$ yields
	\begin{equation}\label{cot mean value 2}
		\begin{aligned}
			\sum_{\substack{i=1\\\text{or}\\j=1}}^N
			\left|\cot\left(\frac{\left(\tilde s_{i}^N-s_{j}^N\right)\pi}{\left|\partial\Omega\right|}\right)
			-
			\cot\left(\frac{\left(\tilde \theta_{i}^N-\theta_{j}^N\right)\pi}{\left|\partial\Omega\right|}\right)\right|
			\hspace{-45mm}&
			\\
			& \leq
			\mathcal{O}\left(N^{-\kappa+1}\right)\sum_{k=1-N}^{N-1}
			\max\left\{\frac{1}{\left|k+\frac 12\right|^2},
			\frac{1}{\left(N-\left|k+\frac 12\right|\right)^2}\right\}
			\\
			& \leq
			\mathcal{O}\left(N^{-\kappa+1}\right)\sum_{k=1-N}^{N-1}
			\frac{1}{\left|k+\frac 12\right|^2}=\mathcal{O}\left(N^{-\kappa+1}\right),
		\end{aligned}
	\end{equation}
	which, when combined with the identities \eqref{perfect distri}, is sufficient to conclude the proof of the lemma.
\end{proof}

\begin{remark}
	The preceding lemma essentially asserts that approximate Riemann sums of $\int_{0}^{\pi} \cot\left(\theta-\tilde\theta\right) d\theta$, for any given $\tilde\theta\in\mathbb{R}$, on a well distributed mesh satisfying \eqref{mesh2} vanish with a convergence rate $\mathcal{O}\left(N^{-\kappa}\right)$. This is crucial if one aims at obtaining a convergence rate $\mathcal{O}\left(N^{-\kappa}\right)$ in Theorem \ref{main theo}. In other words, the consideration of a better mesh produces a faster convergence rate in Lemma \ref{technical cot} which, in turn, results in a faster rate in Theorem \ref{main theo}.
\end{remark}

The following lemma is a precise $\ell^2$-estimate for the uniformly distributed mesh \eqref{mesh}. The first part of this result was already featured in \cite{ADLproc} for the unit disk.

\begin{lemma}\label{est l2}
	Consider the uniformly distributed mesh $(\theta_{1}^N,\dots , \theta_{N}^N)\in \left[0,\left|\partial\Omega\right|\right)^N$, $(\tilde \theta_{1}^N, \dots , \tilde \theta_{N}^N)\in \left[0,\left|\partial\Omega\right|\right)^N$ defined by \eqref{mesh}. Then, for any $z\in\mathbb{R}^N$, we have that
	\begin{equation}\label{cot identity}
		\frac1N \left\| \left\{\sum_{1\leq j \leq N} \cot\left(\frac{\left(\tilde\theta_{k}^N-\theta_{j}^N\right)\pi}{\left|\partial\Omega\right|}\right)z_{j}
		\right\}_{1\leq k\leq N}
		\right\|_{\ell^2}
		=
		\| z - \langle z \rangle \mathbf{1} \|_{\ell^2},
	\end{equation}
	and
	\begin{equation}\label{cot identity 2}
		\frac1N \left\| \left\{\sum_{\substack{1\leq j \leq N \\ j\neq k}} \cot\left(\frac{\left(\theta_{k}^N-\theta_{j}^N\right)\pi}{\left|\partial\Omega\right|}\right)z_{j}
		\right\}_{1\leq k\leq N}
		\right\|_{\ell^2}
		\leq
		\| z - \langle z \rangle \mathbf{1} \|_{\ell^2},
	\end{equation}
	where $\mathbf{1}=(1,\ldots,1)\in\mathbb{R}^N$.
\end{lemma}

\begin{proof}
	First, we compute, utilizing \eqref{perfect distri},
	\begin{align*}
		N \left\| \left\{\sum_{1\leq j \leq N} \cot\left(\frac{\left(\tilde\theta_{k}^N-\theta_{j}^N\right)\pi}{\left|\partial\Omega\right|}\right)z_{j}
		\right\}_{1\leq k\leq N}
		\right\|_{\ell^2}^2
		\hspace{-60mm}&
		\\
		= & \sum_{1\leq k \leq N} \left| \sum_{1\leq j \leq N}
		\cot\left(\frac{\left(\tilde\theta_{k}^N-\theta_{j}^N\right)\pi}{\left|\partial\Omega\right|}\right)z_{j}\right|^2
		\\
		= & \sum_{1\leq k \leq N}  \sum_{1\leq i,j \leq N}
		\cot\left(\frac{\left(\tilde\theta_{k}^N-\theta_{i}^N\right)\pi}{\left|\partial\Omega\right|}\right)
		\cot\left(\frac{\left(\tilde\theta_{k}^N-\theta_{j}^N\right)\pi}{\left|\partial\Omega\right|}\right)
		z_{i} z_{j}
		\\
		=&-\frac12 \sum_{1\leq i,j \leq N} (z_{i} - z_{j})^2 \sum_{1\leq k \leq N}
		\cot\left(\frac{\left(\tilde\theta_{k}^N-\theta_{i}^N\right)\pi}{\left|\partial\Omega\right|}\right)
		\cot\left(\frac{\left(\tilde\theta_{k}^N-\theta_{j}^N\right)\pi}{\left|\partial\Omega\right|}\right)
		\\
		&+ \sum_{1\leq i,k \leq N} |z_{i}|^2
		\cot\left(\frac{\left(\tilde\theta_{k}^N-\theta_{i}^N\right)\pi}{\left|\partial\Omega\right|}\right)
		\sum_{1\leq j \leq N}
		\cot\left(\frac{\left(\tilde\theta_{k}^N-\theta_{j}^N\right)\pi}{\left|\partial\Omega\right|}\right)
		\\
		=&-\frac12 \sum_{1\leq i,j \leq N} (z_{i} - z_{j})^2 \sum_{1\leq k \leq N}
		\cot\left(\frac{\left(\tilde\theta_{k}^N-\theta_{i}^N\right)\pi}{\left|\partial\Omega\right|}\right)
		\cot\left(\frac{\left(\tilde\theta_{k}^N-\theta_{j}^N\right)\pi}{\left|\partial\Omega\right|}\right).
	\end{align*}
	Similarly, employing \eqref{perfect distri 2} instead of \eqref{perfect distri}, we find
		\begin{align*}
			N \left\| \left\{\sum_{\substack{1\leq j \leq N \\ j\neq k}}
			\cot\left(\frac{\left(\theta_{k}^N-\theta_{j}^N\right)\pi}{\left|\partial\Omega\right|}\right)z_{j}
			\right\}_{1\leq k\leq N}
			\right\|_{\ell^2}^2
			\hspace{-60mm}&
			\\
			=&-\frac12 \sum_{1\leq i,j \leq N} (z_{i} - z_{j})^2 \sum_{\substack{1\leq k \leq N \\ k\neq i,j}}
			\cot\left(\frac{\left(\theta_{k}^N-\theta_{i}^N\right)\pi}{\left|\partial\Omega\right|}\right)
			\cot\left(\frac{\left(\theta_{k}^N-\theta_{j}^N\right)\pi}{\left|\partial\Omega\right|}\right).
		\end{align*}
	
	Next, we use the following elementary relation, valid for any $a,b$ such that $a,b,a-b \notin \pi \Z$:
	\[
		\cot a \cot b = \cot (b-a)[\cot a - \cot b]-1,
	\]
	to write, using \eqref{perfect distri},
		\begin{align*}
			N \left\| \left\{\sum_{1\leq j \leq N} \cot\left(\frac{\left(\tilde\theta_{k}^N-\theta_{j}^N\right)\pi}{\left|\partial\Omega\right|}\right)z_{j}
			\right\}_{1\leq k\leq N}
			\right\|_{\ell^2}^2
			\hspace{-40mm}&
			\\
			=& -\frac12 \sum_{\substack{1\leq i, j \leq N \\ i\neq j}} (z_{i} - z_{j})^2  \cot\left(\frac{\left(\theta_{i}^N-\theta_{j}^N\right)\pi}{\left|\partial\Omega\right|}\right)
			\\
			& \times \sum_{1\leq k \leq N}
			\left[
			\cot\left(\frac{\left(\tilde\theta_{k}^N-\theta_{i}^N\right)\pi}{\left|\partial\Omega\right|}\right)
			-
			\cot\left(\frac{\left(\tilde\theta_{k}^N-\theta_{j}^N\right)\pi}{\left|\partial\Omega\right|}\right)
			\right]
			\\
			&+\frac{N}2\sum_{\substack{1\leq i, j \leq N \\ i\neq j}} (z_{i} - z_{j})^2\\
			=&\frac{N}2\sum_{1\leq i,j \leq N} (z_{i} - z_{j})^2,
		\end{align*}
	and, similarly, using \eqref{perfect distri 2},
	\begin{align*}
		N \left\| \left\{\sum_{\substack{1\leq j \leq N \\ j\neq k}} \cot\left(\frac{\left( \theta_{k}^N-\theta_{j}^N\right)\pi}{\left|\partial\Omega\right|}\right)z_{j}
		\right\}_{1\leq k\leq N}
		\right\|_{\ell^2}^2
		\hspace{-50mm}&
		\\
		=& -\frac12 \sum_{\substack{1\leq i,j \leq N \\ i\neq j}} (z_{i} - z_{j})^2
		\cot\left(\frac{\left(\theta_{i}^N-\theta_{j}^N\right)\pi}{\left|\partial\Omega\right|}\right)
		\\
		& \times \sum_{\substack{1\leq k \leq N \\ k\neq i,j }}
		\left[
		\cot\left(\frac{\left(\theta_{k}^N-\theta_{i}^N\right)\pi}{\left|\partial\Omega\right|}\right)
		-
		\cot\left(\frac{\left(\theta_{k}^N-\theta_{j}^N\right)\pi}{\left|\partial\Omega\right|}\right)
		\right]
		\\
		&+\frac{N-2}2\sum_{\substack{1\leq i,j \leq N \\ i\neq j}} (z_{i} - z_{j})^2\\
		=&\frac{N-2}2\sum_{1\leq i,j \leq N} (z_{i} - z_{j})^2
		- \sum_{\substack{1\leq i,j \leq N \\ i\neq j}} (z_{i} - z_{j})^2
		\cot^2\left(\frac{\left(\theta_{i}^N-\theta_{j}^N\right)\pi}{\left|\partial\Omega\right|}\right)
		\\
		\leq & \frac{N-2}2\sum_{1\leq i,j \leq N} (z_{i} - z_{j})^2.
	\end{align*}
	
	Finally, the remaining sums are easily recast as
	\begin{align*}
	\frac{N}2\sum_{1\leq i,j \leq N} (z_{i} - z_{j})^2& = N\sum_{1\leq i,j \leq N} (z_{i} -\langle z\rangle)^2 - N \sum_{1\leq i,j \leq N} (z_{i} -\langle z\rangle)(z_{j} -\langle z\rangle)\\
	&=N^2 \sum_{1\leq i \leq N} (z_{i} -\langle z\rangle)^2 = N^3 \| z - \langle z\rangle\mathbf{1} \|_{\ell^2}^2.
	\end{align*}
	We have therefore obtained that
		\begin{align*}
			N \left\| \left\{\sum_{1\leq j \leq N} \cot\left(\frac{\left(\tilde\theta_{k}^N-\theta_{j}^N\right)\pi}{\left|\partial\Omega\right|}\right)z_{j}
			\right\}_{1\leq k\leq N}
			\right\|_{\ell^2}^2
			& = N^3 \| z - \langle z\rangle\mathbf{1} \|_{\ell^2}^2,
			\\
			N \left\| \left\{\sum_{\substack{1\leq j \leq N \\ j\neq k}} \cot\left(\frac{\left(\theta_{k}^N-\theta_{j}^N\right)\pi}{\left|\partial\Omega\right|}\right)z_{j}
			\right\}_{1\leq k\leq N}
			\right\|_{\ell^2}^2
			& \leq (N-2)N^2 \| z - \langle z\rangle\mathbf{1} \|_{\ell^2}^2,
		\end{align*}
	which ends the proof of the lemma.
\end{proof}

Combining \eqref{cot mean value} (or a slight variant of it without tildes) with \eqref{cot identity} and \eqref{cot identity 2}, it is readily seen that a well distributed mesh enjoys sufficient approximate symmetry to satisfy suitable approximations of \eqref{cot identity} and \eqref{cot identity 2}, which we record in precise terms in the corollary below.

\begin{corollary}\label{est l2 well}
	For any $N\geq 2$, consider a well distributed mesh $(s_{1}^N,\dots , s_{N}^N)\in \left[0,\left|\partial\Omega\right|\right)^N$, $(\tilde s_{1}^N, \dots , \tilde s_{N}^N)\in \left[0,\left|\partial\Omega\right|\right)^N$. Then, for any $z\in\mathbb{R}^N$, we have that
		\begin{align*}
			\left|
			\frac1N \left\| \left\{\sum_{1\leq j \leq N} \cot\left(\frac{\left(\tilde s_{k}^N-s_{j}^N\right)\pi}{\left|\partial\Omega\right|}\right)z_{j}
			\right\}_{1\leq k\leq N}
			\right\|_{\ell^2}
			-
			\| z - \langle z \rangle \mathbf{1} \|_{\ell^2}
			\right|
			\leq & \frac C{N^{\kappa-1}}\left\|z\right\|_{\ell^1},
			\\
			\frac1N \left\| \left\{\sum_{\substack{1\leq j \leq N \\ j\neq k}} \cot\left(\frac{\left(s_{k}^N-s_{j}^N\right)\pi}{\left|\partial\Omega\right|}\right)z_{j}
			\right\}_{1\leq k\leq N}
			\right\|_{\ell^2}
			\leq & \| z - \langle z \rangle \mathbf{1} \|_{\ell^2}
			\\
			& + \frac C{N^{\kappa-1}} \left\|z\right\|_{\ell^1},
		\end{align*}
	where $\mathbf{1}=(1,\ldots,1)\in\mathbb{R}^N$ and the constant $C>0$ is independent of $N$ and $z$.
\end{corollary}

A direct consequence of the preceding estimates on well distributed meshes concerns the uniform boundedness of the operators defined above.

\begin{corollary}\label{boundedness}
	For any $N\geq 2$, consider a well distributed mesh $(s_{1}^N,\dots , s_{N}^N)\in \left[0,\left|\partial\Omega\right|\right)^N$, $(\tilde s_{1}^N, \dots , \tilde s_{N}^N)\in \left[0,\left|\partial\Omega\right|\right)^N$. Then, there exists a constant $C>0$ independent of $N$ such that, for each $1\leq p\leq\infty$,
	\begin{equation*}
		\frac 1N
		\left(
		\left\|A_Nz\right\|_{\ell^p}
		+
		\left\|\tilde A_Nz\right\|_{\ell^p}
		\right)
		\leq C\left\|z\right\|_{\ell^1},
	\end{equation*}
	and
	\begin{equation*}
		\frac 1N
		\left(
		\left\|B_Nz\right\|_{\ell^2}
		+
		\left\|\tilde B_Nz\right\|_{\ell^2}
		\right)
		\leq C\left\|z\right\|_{\ell^2},
	\end{equation*}
	for all $z\in\mathbb{R}^N$.
\end{corollary}

\begin{proof}
	The boudedness of $A_N$ and $\tilde A_N$ easily follows from the uniform boundedness of each component of the corresponding matrices (recall that $\frac{l(s)-l(s_*)}{|l(s)-l(s_*)|^2}\cdot n(l(s))$ is a continuous bounded function, which follows directly from \eqref{kernel decomposition}).
	
	As for the boundedness of $B_N$ and $\tilde B_N$, it follows from Corollary \ref{est l2 well} combined with the fact that $\frac{l(s)-l(s_*)}{|l(s)-l(s_*)|^2}\cdot \tau(l(s))-\frac{\pi}{\left|\partial\Omega\right|}\cot\left(\frac{(s-s_*)\pi}{\left|\partial\Omega\right|}\right)$ is a continuous bounded function (which is a direct consequence of \eqref{kernel decomposition} and the fact that $\frac 1s-\frac{\pi}{\left|\partial\Omega\right|}\cot\left(\frac{s\pi}{\left|\partial\Omega\right|}\right)$ is smooth near $s=0$).
\end{proof}

\subsection{Approximation of the Poincar\'e--Bertrand identities}

The next proposition provides a crucial discretization of \eqref{mean} and the Poincar\'e--Bertrand identities \eqref{PB}.

\begin{proposition}\label{important}
	For any $N\geq 2$, consider a well distributed mesh $(s_{1}^N,\dots , s_{N}^N)\in \left[0,\left|\partial\Omega\right|\right)^N$, $(\tilde s_{1}^N, \dots , \tilde s_{N}^N)\in \left[0,\left|\partial\Omega\right|\right)^N$. Then, for all $z\in \mathbb{R}^N$, as $N\to\infty$,
	\begin{equation}\label{mean approx}
		\begin{aligned}
			\left|\ip{\frac{\left|\partial\Omega\right|}{N}B_Nz}\right|
			+
			\left|\ip{\frac{\left|\partial\Omega\right|}{N}A_Nz-\pi z}\right|
			& \leq
			\frac{C}{N^\kappa}
			\left\|z\right\|_{\ell^1},
			\\
			\left|\ip{\frac{\left|\partial\Omega\right|}{N}\tilde B_Nz}\right|
			+
			\left|\ip{\frac{\left|\partial\Omega\right|}{N}\tilde A_Nz-\pi z}\right|
			& \leq
			\frac{C}{N^\kappa}
			\left\|z\right\|_{\ell^1},
		\end{aligned}
	\end{equation}
	and
	\begin{equation}\label{PB approx}
		\begin{aligned}
			\left\|
			\frac{\left|\partial\Omega\right|^2}{N^2}
			\left( B_N\tilde B_N- A_N\tilde A_N\right) z
			+\pi^2 z
			\right\|_{\ell^2} \hspace{35mm}&
			\\
			+
			\left\|
			\frac{\left|\partial\Omega\right|^2}{N^2}
			\left( A_N\tilde B_N+B_N\tilde A_N\right) z
			\right\|_{\ell^2}
			& \leq \frac{C}{N} \| z \|_{\ell^2},
			\\
			\left\|
			\frac{\left|\partial\Omega\right|^2}{N^2}
			\left( \tilde B_NB_N- \tilde A_N A_N\right) z
			+\pi^2 z
			\right\|_{\ell^2} \hspace{35mm}&
			\\
			+
			\left\|
			\frac{\left|\partial\Omega\right|^2}{N^2}
			\left( \tilde A_N B_N+\tilde B_N A_N\right) z
			\right\|_{\ell^2}
			& \leq \frac{C}{N} \| z \|_{\ell^2},
		\end{aligned}
	\end{equation}
	where $C>0$ is an independent constant.
\end{proposition}

\begin{proof}
	Recall that $l:\left[0,\left|\partial\Omega\right|\right]\to\mathbb{R}^2$ denotes a given arc-length parametrization of $\partial \Omega$. For clarity, we also introduce the smooth path $L:\left[0,\left|\partial\Omega\right|\right]\to\mathbb{C}$ defining the contour $\Gamma\subset\mathbb{C}$ matching $\partial\Omega$ through the usual identification of $\mathbb{R}^2$ with the complex plane $\mathbb{C}$, so that $L(s)$ is identified with $l(s)$, for each $s\in \left[0,\left|\partial\Omega\right|\right]$.
	
	We claim that
	\begin{align}
		\label{claim1}
		\max_{1\leq k\leq N}
		\left|
		\frac{\left|\partial\Omega\right|}{N}\sum_{1\leq j\leq N}
		\frac{L'\left(\tilde s_j^N\right)}{L\left(\tilde s_j^N\right)-L\left(s_k^N\right)}
		-i\pi
		\right| & =\mathcal{O}\left(N^{-\kappa}\right),
		\\
		\label{claim2}
		\max_{1\leq k\leq N}
		\left|
		\frac{\left|\partial\Omega\right|^2}{N^2}\sum_{1\leq j\leq N}
		\frac{L'\left(\tilde s_j^N\right)L'\left(s_k^N\right)}{\left(L\left(\tilde s_j^N\right)-L\left(s_k^N\right)\right)^2}
		-\pi^2
		\right| & =\mathcal{O}\left(N^{-1}\right),
	\end{align}
	and, for any $z\in\mathbb{R}^N$, that
	\begin{equation}\label{claim3}
		\frac1N \left\| \left\{\sum_{\substack{1\leq j \leq N \\ j\neq k}}
		\frac{L'\left(s_j^N\right)}{L\left(s_k^N\right)-L\left(s_j^N\right)}
		z_{j}
		\right\}_{1\leq k\leq N}
		\right\|_{\ell^2}
		\leq C \| z \|_{\ell^2},
	\end{equation}
	where the constant $C>0$ only depends on fixed parameters.

	Indeed, the first claim \eqref{claim1} is obtained by writing
	\begin{equation*}
			\sum_{1\leq j\leq N}
			\frac{L'\left(\tilde s_j^N\right)}{L\left(\tilde s_j^N\right)-L\left(s_k^N\right)}
			=
			\sum_{1\leq j\leq N} f\left(\tilde s_j^N,s_k^N\right)
			+ \frac{\pi}{\left|\partial\Omega\right|} \sum_{1\leq j\leq N} \cot\left(\frac{\left(\tilde s_j^N-s_k^N\right)\pi}{\left|\partial\Omega\right|}\right),
	\end{equation*}
	where
	\begin{equation*}
		f(s,s_*)=
		\frac{L'(s)}{L(s)-L(s_*)}
		-\frac{\pi}{\left|\partial\Omega\right|}
		\cot\left(\frac{\left(s-s_*\right)\pi}{\left|\partial\Omega\right|}\right),
		\quad \text{with }s,s_*\in \left[0,\left|\partial\Omega\right|\right],
	\end{equation*}
	is a smooth function, which is a direct consequence of \eqref{kernel decomposition} and the fact that $\frac 1s-\frac{\pi}{\left|\partial\Omega\right|}\cot\left(\frac{s\pi}{\left|\partial\Omega\right|}\right)$ is smooth near $s=0$. Then, by virtue of Lemma \ref{technical cot} and the uniform convergence of Riemann sums for smooth functions (see Corollary \ref{riemann2} from Appendix \ref{riemann appendix}, if necessary), we deduce that
	\begin{equation*}
		\max_{1\leq k\leq N}
		\left|
		\frac{\left|\partial\Omega\right|}{N}\sum_{1\leq j\leq N}
		\frac{L'\left(\tilde s_j^N\right)}{L\left(\tilde s_j^N\right)-L\left(s_k^N\right)}
		-\int_0^{\left|\partial\Omega\right|} f\left(s,s_k^N\right)ds
		\right|= \mathcal{O}\left(N^{-\kappa}\right).
	\end{equation*}
	The justification of \eqref{claim1} is then completed upon noticing that, for any $s_*\in\left[0,\left|\partial\Omega\right|\right]$,
	\begin{equation*}
		\int_0^{\left|\partial\Omega\right|}f(s,s_*)ds
		=
		\int_\Gamma\frac{1}{z-L(s_*)}
		dz=i\pi,
	\end{equation*}
	where the last integral above is defined in the sense of Cauchy's principal value and is easily evaluated employing Plemelj's formulas \eqref{plemelj} with the representation of Cauchy integrals \eqref{complex representation}.

	We turn now to the justification of \eqref{claim2}. To this end, we use, again, that $f(s,s_*)$ is smooth to deduce that
	\begin{equation*}
		\frac{\partial f}{\partial s_*}(s,s_*) + \frac{\pi^2}{\left|\partial\Omega\right|^2}=
		\frac{L'(s)L'(s_*)}{\left(L(s)-L(s_*)\right)^2}
		-\frac{\pi^2}{\left|\partial\Omega\right|^2}
		\cot^2\left(\frac{\left(s-s_*\right)\pi}{\left|\partial\Omega\right|}\right),
	\end{equation*}
	is smooth, as well. Therefore, by Corollary \ref{riemann2} from Appendix \ref{riemann appendix} and employing \eqref{mesh2}-\eqref{mesh}, we find that
		\begin{align*}
			\left|
			\sum_{1\leq j\leq N}
			\left(\frac{L'\left(\tilde s_j^N\right)L'\left(s_k^N\right)}{\left(L\left(\tilde s_j^N\right)-L\left(s_k^N\right)\right)^2}
			-
			\frac{\pi^2}{\left|\partial\Omega\right|^2} \cot^2\left(\frac{\left(\tilde s_j^N-s_k^N\right)\pi}{\left|\partial\Omega\right|}\right)
			\right)
			\right|
			=\mathcal{O}(N).
		\end{align*}
	By \eqref{cot mean value 2}, we further find that
		\begin{align*}
			\sum_{1\leq j\leq N} \left|\cot^2\left(\frac{\left(\tilde s_j^N-s_k^N\right)\pi}{\left|\partial\Omega\right|}\right)
			-\cot^2\left(\frac{\left(\tilde \theta_j^N-\theta_k^N\right)\pi}{\left|\partial\Omega\right|}\right)\right|
			\hspace{-70mm}&
			\\
			& \leq C N \sum_{1\leq j\leq N}
			\left|\cot\left(\frac{\left(\tilde s_j^N-s_k^N\right)\pi}{\left|\partial\Omega\right|}\right)
			-\cot\left(\frac{\left(\tilde \theta_j^N-\theta_k^N\right)\pi}{\left|\partial\Omega\right|}\right)\right|
			= \mathcal{O}(N^{-\kappa+2}),
		\end{align*}
	for some independent constant $C>0$, which implies that
		\begin{align*}
			\left|
			\sum_{1\leq j\leq N}
			\left(\frac{L'\left(\tilde s_j^N\right)L'\left(s_k^N\right)}{\left(L\left(\tilde s_j^N\right)-L\left(s_k^N\right)\right)^2}
			-
			\frac{\pi^2}{\left|\partial\Omega\right|^2} \cot^2\left(\frac{\left(\tilde \theta_j^N-\theta_k^N\right)\pi}{\left|\partial\Omega\right|}\right)
			\right)
			\right|
			=\mathcal{O}(N).
		\end{align*}
	The justification of \eqref{claim2} is then completed upon noticing that
		\begin{align*}
			\frac {\pi^2}N + & \frac{\pi^2}{N^2} \sum_{1\leq j\leq N} \cot^2\left(\frac{\left(\tilde \theta_j^N-\theta_k^N\right)\pi}{\left|\partial\Omega\right|}\right)
			\\
			= &
			\frac{\pi^2}{N^2} \sum_{1-k\leq j\leq N-k}
			\frac{1}{\sin^2\left(\frac{\left(j+\frac 12\right)\pi}{N}\right)}
			=
			\frac{\pi^2}{N^2} \sum_{1-\left[\frac N2\right]\leq j\leq N-\left[\frac N2\right]}
			\frac{1}{\sin^2\left(\frac{\left(j+\frac 12\right)\pi}{N}\right)}
			\\
			= &
			\frac {\pi^2}{N^2}\sum_{1-\left[\frac N2\right]\leq j\leq N-\left[\frac N2\right]}
			\left(\frac{1}{\sin^2\left(\frac{\left(j+\frac 12\right)\pi}{N}\right)}
			-\frac{1}{\left(\frac{\left(j+\frac 12\right)\pi}{N}\right)^2}
			\right)
			\\
			& + \sum_{1-\left[\frac N2\right]\leq j\leq N-\left[\frac N2\right]}
			\frac{4}{\left(2j+1\right)^2},
		\end{align*}
	whereby, by the convergence of Riemann sums (see Corollary \ref{riemann2} from Appendix \ref{riemann appendix}, if necessary),
		\begin{align*}
			\frac{\pi^2}{N^2} \sum_{1\leq j\leq N} \cot^2\left(\frac{\left(\tilde \theta_j^N-\theta_k^N\right)\pi}{\left|\partial\Omega\right|}\right)
			& = \mathcal{O}\left(N^{-1}\right)
			+
			\sum_{0\leq j\leq \left[\frac N2\right]}
			\frac{8}{\left(2j+1\right)^2}
			\\
			& = \mathcal{O}\left(N^{-1}\right)
			+
			\sum_{0\leq j<\infty}
			\frac{8}{\left(2j+1\right)^2}
			\\
			& = \mathcal{O}\left(N^{-1}\right)
			+8
			\left(
			\sum_{1\leq j<\infty}
			\frac{1}{j^2}
			-\sum_{1\leq j<\infty}
			\frac{1}{(2j)^2}
			\right)
			\\
			& = \mathcal{O}\left(N^{-1}\right)
			+6
			\sum_{1\leq j<\infty}\frac{1}{j^2}
			= \mathcal{O}\left(N^{-1}\right)
			+\pi^2.
		\end{align*}

	As for our third claim \eqref{claim3}, it directly follows from Corollary\ \ref{est l2 well} using that $f(s,s_*)$ is a continuous---and therefore bounded---function.

	Now that \eqref{claim1}, \eqref{claim2} and \eqref{claim3} are established, we move on to the actual justification of \eqref{PB approx}. To this end, we decompose, for any $z\in\mathbb{R}^N$ and each $1\leq k\leq N$,
		\begin{align*}
			\sum_{1\leq i,j\leq N}
			\frac{L'\left(\tilde s_j^N\right)}{L\left(\tilde s_j^N\right)-L\left(s_k^N\right)}
			\frac{L'\left(s_i^N\right)}{L\left(\tilde s_j^N\right)-L\left(s_i^N\right)}
			z_i
			\hspace{-40mm}&
			\\
			= & \sum_{\substack{1\leq i\leq N \\ i\neq k}}
			\frac{L'\left(s_i^N\right)}{L\left(s_k^N\right)-L\left(s_i^N\right)}
			\\
			& \times\sum_{1\leq j\leq N}
			L'\left(\tilde s_j^N\right)
			\left[
			\frac{1}{L\left(\tilde s_j^N\right)-L\left(s_k^N\right)}
			-\frac{1}{L\left(\tilde s_j^N\right)-L\left(s_i^N\right)}
			\right]z_i
			\\
			& +
			\sum_{1\leq j\leq N}
			\frac{L'\left(\tilde s_j^N\right)L'\left(s_k^N\right)}{\left(L\left(\tilde s_j^N\right)-L\left(s_k^N\right)\right)^2}
			z_k
			\\
			= &
			\left[
			\sum_{1\leq j\leq N}
			\frac{L'\left(\tilde s_j^N\right)}{L\left(\tilde s_j^N\right)-L\left(s_k^N\right)}
			-\frac{i\pi N}{\left|\partial\Omega\right|}
			\right]
			\sum_{\substack{1\leq i\leq N \\ i\neq k}}
			\frac{L'\left(s_i^N\right)}{L\left(s_k^N\right)-L\left(s_i^N\right)}
			z_i
			\\
			& - \sum_{\substack{1\leq i\leq N \\ i\neq k}}
			\frac{L'\left(s_i^N\right)}{L\left(s_k^N\right)-L\left(s_i^N\right)}
			\left[
			\sum_{1\leq j\leq N}
			\frac{L'\left(\tilde s_j^N\right)}{L\left(\tilde s_j^N\right)-L\left(s_i^N\right)}
			-\frac{i\pi N}{\left|\partial\Omega\right|}
			\right]z_i
			\\
			& +
			\left[
			\sum_{1\leq j\leq N}
			\frac{L'\left(\tilde s_j^N\right)L'\left(s_k^N\right)}{\left(L\left(\tilde s_j^N\right)-L\left(s_k^N\right)\right)^2}
			-\frac{\pi^2N^2}{\left|\partial\Omega\right|^2}\right]
			z_k + \frac{\pi^2N^2}{\left|\partial\Omega\right|^2}z_k.
		\end{align*}
	Therefore, employing \eqref{claim1}, \eqref{claim2} and \eqref{claim3}, it follows that
		\begin{align*}
			\left\| \left\{
			\frac{\left|\partial\Omega\right|^2}{N^2}
			\sum_{1\leq i,j\leq N}
			\frac{L'\left(\tilde s_j^N\right)}{L\left(\tilde s_j^N\right)-L\left(s_k^N\right)}
			\frac{L'\left(s_i^N\right)}{L\left(\tilde s_j^N\right)-L\left(s_i^N\right)}
			z_i
			-\pi^2 z_k
			\right\}_{1\leq k\leq N}
			\right\|_{\ell^2}
			\hspace{-10mm}&
			\\
			& \leq \frac{C}{N} \| z \|_{\ell^2}.
		\end{align*}
	Clearly, exchanging the symmetric roles of $(s_{1}^N,\dots , s_{N}^N)$ and $(\tilde s_{1}^N, \dots , \tilde s_{N}^N)$ yields the equivalent estimate
		\begin{align*}
			\left\| \left\{
			\frac{\left|\partial\Omega\right|^2}{N^2}
			\sum_{1\leq i,j\leq N}
			\frac{L'\left(s_j^N\right)}{L\left(s_j^N\right)-L\left(\tilde s_k^N\right)}
			\frac{L'\left(\tilde s_i^N\right)}{L\left(s_j^N\right)-L\left(\tilde s_i^N\right)}
			z_i
			-\pi^2 z_k
			\right\}_{1\leq k\leq N}
			\right\|_{\ell^2}
			\hspace{-10mm}&
			\\
			& \leq \frac{C}{N} \| z \|_{\ell^2}.
		\end{align*}
	
	Finally, noticing that, for all $s,s_*\in\left[0,\left|\partial\Omega\right|\right]$,
	\begin{equation}\label{real im}
		\begin{aligned}
			\frac{L'\left(s\right)}{L\left(s\right)-L\left(s_*\right)}
			=
			\frac{l(s)-l(s_*)}{\left|l(s)-l(s_*)\right|^2}\cdot \tau\left(l(s)\right)
			+i
			\frac{l(s)-l(s_*)}{\left|l(s)-l(s_*)\right|^2}\cdot n\left(l(s)\right),
		\end{aligned}
	\end{equation}
	it is readily seen that (we suggest the reader to compare these identities with \eqref{complex representation})
		\begin{align*}
			\left\{\sum_{1\leq i,j\leq N}
			\frac{L'\left(\tilde s_j^N\right)}{L\left(\tilde s_j^N\right)-L\left(s_k^N\right)}
			\frac{L'\left(s_i^N\right)}{L\left(\tilde s_j^N\right)-L\left(s_i^N\right)}
			z_i
			\right\}_{1\leq k\leq N}
			\hspace{-10mm}&
			\\
			& =
			-\left( B_N^*+i A_N^*\right)\left(\tilde B_N^*+i\tilde A_N^*\right) z,
		\end{align*}
	and
		\begin{align*}
			\left\{
			\sum_{1\leq i,j\leq N}
			\frac{L'\left(s_j^N\right)}{L\left(s_j^N\right)-L\left(\tilde s_k^N\right)}
			\frac{L'\left(\tilde s_i^N\right)}{L\left(s_j^N\right)-L\left(\tilde s_i^N\right)}
			z_i
			\right\}_{1\leq k\leq N}
			\hspace{-10mm}&
			\\
			& =
			-\left( \tilde B_N^*+i \tilde A_N^*\right)\left( B_N^*+i A_N^*\right) z.
		\end{align*}
	Therefore, identifying the real and imaginary parts in the preceding estimates yields
		\begin{align*}
			\left\|
			\frac{\left|\partial\Omega\right|^2}{N^2}
			\left( B_N^*\tilde B_N^*- A_N^*\tilde A_N^*\right) z
			+\pi^2 z
			\right\|_{\ell^2}
			\hspace{-10mm}&
			\\
			& +
			\left\|
			\frac{\left|\partial\Omega\right|^2}{N^2}
			\left( A_N^*\tilde B_N^*+B_N^*\tilde A_N^*\right) z
			\right\|_{\ell^2}
			\leq \frac{C}{N} \| z \|_{\ell^2},
			\\
			\left\|
			\frac{\left|\partial\Omega\right|^2}{N^2}
			\left( \tilde B_N^*B_N^*- \tilde A_N^* A_N^*\right) z
			+\pi^2 z
			\right\|_{\ell^2}
			\hspace{-10mm}&
			\\
			& +
			\left\|
			\frac{\left|\partial\Omega\right|^2}{N^2}
			\left( \tilde A_N^* B_N^*+\tilde B_N^* A_N^*\right) z
			\right\|_{\ell^2}
			\leq \frac{C}{N} \| z \|_{\ell^2}.
		\end{align*}
	A standard duality argument concludes the proof of \eqref{PB approx}.
	
	As for \eqref{mean approx}, by combining \eqref{claim1} with \eqref{real im} and then exchanging the symmetric roles of $(s_{1}^N,\dots , s_{N}^N)$ and $(\tilde s_{1}^N, \dots , \tilde s_{N}^N)$, we easily obtain
	\begin{equation}\label{average adjoint}
		\begin{aligned}
			\left\|
			\frac{\left|\partial\Omega\right|}{N}\left(B^*_N+iA^*_N\right)\mathbf{1}
			-i\pi\mathbf{1}
			\right\|_{\ell^\infty} & =\mathcal{O}\left(N^{-\kappa}\right),
			\\
			\left\|
			\frac{\left|\partial\Omega\right|}{N}\left(\tilde B^*_N+i\tilde A^*_N\right)\mathbf{1}
			-i\pi\mathbf{1}
			\right\|_{\ell^\infty} & =\mathcal{O}\left(N^{-\kappa}\right),
		\end{aligned}
	\end{equation}
	where $\mathbf{1}=(1,\ldots,1)\in\mathbb{R}^N$. Consequently, by duality, we find that
		\begin{align*}
			\left|\ip{\frac{\left|\partial\Omega\right|}{N}\left(B_N+iA_N\right)z-i\pi z}\right|
			& \leq
			\left\|
			\frac{\left|\partial\Omega\right|}{N}\left(B^*_N+iA^*_N\right)\mathbf{1}
			-i\pi\mathbf{1}
			\right\|_{\ell^\infty}\left\|z\right\|_{\ell^1}
			\\
			& \leq \mathcal{O}\left(N^{-\kappa}\right)
			\left\|z\right\|_{\ell^1},
			\\
			\left|\ip{\frac{\left|\partial\Omega\right|}{N}\left(\tilde B_N+i\tilde A_N\right)z-i\pi z}\right|
			& \leq
			\left\|
			\frac{\left|\partial\Omega\right|}{N}\left(\tilde B^*_N+i\tilde A^*_N\right)\mathbf{1}
			-i\pi\mathbf{1}
			\right\|_{\ell^\infty}\left\|z\right\|_{\ell^1}
			\\
			& \leq \mathcal{O}\left(N^{-\kappa}\right)
			\left\|z\right\|_{\ell^1},
		\end{align*}
	which completes the proof of the proposition.
\end{proof}

\subsection{Approximation of spectral properties}

We will also need a discretized version of the estimate \eqref{gelfand} on the spectral radius of $A$ acting on zero mean elements, which is precisely the content of the coming lemma. To this end, recall from \eqref{mean} that the operator $A$ preserves the mean value of its argument. In order that the relevant discretized operators enjoy similar properties, we introduce a correction to $\tilde A_NA_N$ and $A_N\tilde A_N$. More precisely, we introduce the operators $D_N,\tilde D_N:\mathbb{R}^N\to\mathbb{R}^N$ defined by, for all $z\in\mathbb{R}^N$,
		\begin{align*}
			D_N z & = \tilde A_N A_N z - \ip{\tilde A_N A_N z - \frac{\pi^2 N^2}{|\partial\Omega|^2}z}\mathbf{1}
			\\
			& = \tilde A_N A_N z - \ip{\tilde A_N A_N z - \frac{\pi N}{|\partial\Omega|}A_Nz}\mathbf{1}
			- \ip{\frac{\pi N}{|\partial\Omega|}A_Nz - \frac{\pi^2 N^2}{|\partial\Omega|^2}z}\mathbf{1},
			\\
			\tilde D_N z & = A_N \tilde A_N z - \ip{A_N \tilde A_N z - \frac{\pi^2 N^2}{|\partial\Omega|^2}z}\mathbf{1}
			\\
			& = A_N \tilde A_N z - \ip{A_N \tilde A_N z - \frac{\pi N}{|\partial\Omega|}\tilde A_Nz}\mathbf{1}
			- \ip{\frac{\pi N}{|\partial\Omega|}\tilde A_Nz - \frac{\pi^2 N^2}{|\partial\Omega|^2}z}\mathbf{1},
		\end{align*}
	where $\mathbf{1}=(1,\ldots,1)\in\mathbb{R}^N$. Observe that $\frac{|\partial\Omega|^2}{N^2}\ip{D_Nz}=\frac{|\partial\Omega|^2}{N^2}\ip{\tilde D_Nz}=\pi^2\ip{z}$. We also define the subspace $\ell^p_0\subset\ell^p$, for any $1\leq p \leq \infty$, by
	\begin{equation*}
		\ell^p_0=\set{z\in \ell^p}{\ip{z}=0}.
	\end{equation*}

\begin{lemma}\label{inverse D}
	For any $N\geq 2$, consider a well distributed mesh $(s_{1}^N,\dots , s_{N}^N)\in \left[0,\left|\partial\Omega\right|\right)^N$, $(\tilde s_{1}^N, \dots , \tilde s_{N}^N)\in \left[0,\left|\partial\Omega\right|\right)^N$. Then, there exist $N_*,k_*\geq 1$ and $\delta>0$ such that
		\begin{align*}
			\left\|
			\left( \frac{\left|\partial\Omega\right|^2}{N^2} D_N\right)^k
			\right\|_{\mathcal{L}\left(\ell^2_0\right)}^\frac 1k
			\leq\left(\pi-\delta\right)^{2}
			\quad\text{and}\quad
			\left\|
			\left( \frac{\left|\partial\Omega\right|^2}{N^2}\tilde D_N\right)^k
			\right\|_{\mathcal{L}\left(\ell^2_0\right)}^\frac 1k
			\leq\left(\pi-\delta\right)^{2},
		\end{align*}
	for all $k\geq k_*$ and $N\geq N_*$.
	
	In particular, provided $N$ is sufficiently large, the Neumann series
		\begin{align*}
			\left(\frac{\left|\partial\Omega\right|^2}{N^2} D_N-\pi^2\right)^{-1}
			& =
			-\pi^{-2}\sum_{k=0}^\infty \left(\frac{\left|\partial\Omega\right|^2}{\pi^2 N^2} D_N\right)^{k},
			\\
			\left(\frac{\left|\partial\Omega\right|^2}{N^2} \tilde D_N-\pi^2\right)^{-1}
			& =
			-\pi^{-2}\sum_{k=0}^\infty \left(\frac{\left|\partial\Omega\right|^2}{\pi^2 N^2} \tilde D_N\right)^{k},
		\end{align*}
	are absolutely convergent in $\mathcal{L}\left(\ell^2_0\right)$ and the inverse operators they define are bounded in $\mathcal{L}\left(\ell^2_0\right)$ uniformly in $N$.
\end{lemma}

\begin{proof}
	For each $k\geq 1$, we denote by $K_k(x,y)$ the kernel of $A^k$. Note that $K_k$ is smooth and satisfies, for all $x,y\in\partial\Omega$,
		\begin{align*}
			K_k(x,y)= &
			\int_{\partial\Omega\times\dots\times\partial\Omega}
			\frac{x-y_1}{|x-y_1|^2}\cdot n(x)
			\\
			&\times
			\left(\prod_{j=1}^{k-2}\frac{y_j-y_{j+1}}{|y_j-y_{j+1}|^2}\cdot n(y_j)\right)
			\frac{y_{k-1}-y}{|y_{k-1}-y|^2}\cdot n(y_{k-1})
			dy_1\dots dy_{k-1}.
		\end{align*}
	In particular, for even indices, we may write
		\begin{align*}
			K_{2k}(x,y) \hspace{-10mm}&
			\\
			= &
			\int_{\partial\Omega\times\dots\times\partial\Omega}
			\left(
			\int_{\partial\Omega}\frac{x-y_1}{|x-y_1|^2}\cdot n(x)
			\frac{y_1-y_2}{|y_1-y_2|^2}\cdot n(y_1)dy_1
			\right)
			\\
			&\times \prod_{j=1}^{k-2}
			\left(
			\int_{\partial\Omega}\frac{y_{2j}-y_{2j+1}}{|y_{2j}-y_{2j+1}|^2}\cdot n(y_{2j})
			\frac{y_{2j+1}-y_{2j+2}}{|y_{2j+1}-y_{2j+2}|^2}\cdot n(y_{2j+1})dy_{2j+1}
			\right)
			\\
			&\times
			\left(
			\int_{\partial\Omega}\frac{y_{2k-2}-y_{2k-1}}{|y_{2k-2}-y_{2k-1}|^2}\cdot n(y_{2k-2})
			\frac{y_{2k-1}-y}{|y_{2k-1}-y|^2}\cdot n(y_{2k-1})dy_{2k-1}
			\right)
			\prod_{j=1}^{k-1}dy_{2j}.
		\end{align*}
	Therefore, by smoothness of the kernel $\frac{x-y}{|x-y|^2}\cdot n(x)$, approximating the above integrals by their Riemann sums yields, as $N\to\infty$, that $K_{2k}(x,y)$ is arbitrarily close (which is symbolized here by $\sim$) in $L_{x,y}^\infty\left(\partial\Omega\times\partial\Omega\right)$ to
		\begin{align*}
			K_{2k} (x,y) \hspace{-9mm}&
			\\\sim &
			\int_{\partial\Omega\times\dots\times\partial\Omega}
			\left(
			\frac{\left|\partial\Omega\right|}{N}\sum_{i=1}^N \frac{x-l\left(\tilde s_i^N\right)}{\left|x-l\left(\tilde s_i^N\right)\right|^2}\cdot n(x)
			\frac{l\left(\tilde s_i^N\right)-y_2}{\left|l\left(\tilde s_i^N\right)-y_2\right|^2}\cdot n\left(l\left(\tilde s_i^N\right)\right)
			\right)
			\\
			&\times \prod_{j=1}^{k-2}
			\left(
			\frac{\left|\partial\Omega\right|}{N}\sum_{i=1}^N \frac{y_{2j}-l\left(\tilde s_i^N\right)}{\left|y_{2j}-l\left(\tilde s_i^N\right)\right|^2}\cdot n\left(y_{2j}\right)
			\frac{l\left(\tilde s_i^N\right)-y_{2j+2}}{\left|l\left(\tilde s_i^N\right)-y_{2j+2}\right|^2}\cdot n\left(l\left(\tilde s_i^N\right)\right)
			\right)
			\\
			&\times
			\left(
			\frac{\left|\partial\Omega\right|}{N}\sum_{i=1}^N \frac{y_{2k-2}-l\left(\tilde s_i^N\right)}{\left|y_{2k-2}-l\left(\tilde s_i^N\right)\right|^2}\cdot n\left(y_{2k-2}\right)
			\frac{l\left(\tilde s_i^N\right)-y}{\left|l\left(\tilde s_i^N\right)-y\right|^2}\cdot n\left(l\left(\tilde s_i^N\right)\right)
			\right)
			\prod_{j=1}^{k-1}dy_{2j}
			\\
			\sim &
			\left(\frac{\left|\partial\Omega\right|}{N}\right)^{2k-1}
			\\
			& \times \sum_{j_1,\ldots,j_{k-1}=1}^N
			\left(
			\sum_{i=1}^N \frac{x-l\left(\tilde s_i^N\right)}{\left|x-l\left(\tilde s_i^N\right)\right|^2}\cdot n(x)
			\frac{l\left(\tilde s_i^N\right)-l\left(s_{j_1}^N\right)}{\left|l\left(\tilde s_i^N\right)-l\left(s_{j_1}^N\right)\right|^2}\cdot n\left(l\left(\tilde s_i^N\right)\right)
			\right)
			\\
			&\times \prod_{n=1}^{k-2}
			\left(
			\sum_{i=1}^N \frac{l\left(s_{j_n}^N\right)-l\left(\tilde s_i^N\right)}{\left|l\left(s_{j_n}^N\right)-l\left(\tilde s_i^N\right)\right|^2}\cdot n\left(l\left(s_{j_n}^N\right)\right)
			\frac{l\left(\tilde s_i^N\right)-l\left(s_{j_{n+1}}^N\right)}{\left|l\left(\tilde s_i^N\right)-l\left(s_{j_{n+1}}^N\right)\right|^2}\cdot n\left(l\left(\tilde s_i^N\right)\right)
			\right)
			\\
			&\times
			\left(
			\sum_{i=1}^N \frac{l\left(s_{j_{k-1}}^N\right)-l\left(\tilde s_i^N\right)}{\left|l\left(s_{j_{k-1}}^N\right)-l\left(\tilde s_i^N\right)\right|^2}\cdot n\left(l\left(s_{j_{k-1}}^N\right)\right)
			\frac{l\left(\tilde s_i^N\right)-y}{\left|l\left(\tilde s_i^N\right)-y\right|^2}\cdot n\left(l\left(\tilde s_i^N\right)\right)
			\right).
		\end{align*}
	Further discretizing in $x=l(s)$ and $y=l(s_*)$, with $s,s_*\in\left[0,\left|\partial\Omega\right|\right]$, we deduce that, as $N\to\infty$,
	\begin{equation}\label{k approx}
		\left(K_{2k}(x,y)
		-
		\left(\frac{\left|\partial\Omega\right|}{N}\right)^{2k-1}
		\sum_{i,j=1}^N
		\mathds{1}_{\left[\theta_{i}^N, \theta_{i+1}^N\right)}(s)
		\left(\left(\tilde A_N A_N\right)^k\right)_{ij}
		\mathds{1}_{\left[\theta_{j}^N, \theta_{j+1}^N\right)}(s_*)
		\right)\to 0,
	\end{equation}
	in $L_{x,y}^\infty\left(\partial\Omega\times\partial\Omega\right)$, where the $\theta_i^N$'s are defined in \eqref{mesh}.
	
	Now, in view of Corollary \ref{boundedness} and Proposition \ref{important}, it holds that, denoting by $e^j$ the $j^\text{th}$ vector of the canonical basis of $\mathbb{R}^N$ and interpreting matrices in $\mathbb{R}^{N\times N}$ as vectors in $\mathbb{R}^{N^2}$ measured in the $\ell^\infty$-norm,
		\begin{align*}
			\left\|D_N-\tilde A_N A_N\right\|_{\ell^\infty} \hspace{-10mm}&
			\\
			& =
			\sup_{1\leq j\leq N}
			\left|\ip{\tilde A_N A_N e^j - \frac{\pi N}{|\partial\Omega|}A_Ne^j}
			+ \ip{\frac{\pi N}{|\partial\Omega|}A_Ne^j - \frac{\pi^2 N^2}{|\partial\Omega|^2}e^j}\right|
			\\
			& \leq \frac CN\left\|A_N e^j\right\|_{\ell^1} + C\left\|e^j\right\|_{\ell^1}
			= \mathcal{O}\left(N^{-1}\right),
			\\
			\left\|D_N-\tilde A_N A_N\right\|_{\mathcal{L}\left(\ell^2\right)} \hspace{-10mm}&
			\\
			& =
			\sup_{\substack{\|z\|_{\ell^2}=1}}
			\left|\ip{\tilde A_N A_N z - \frac{\pi N}{|\partial\Omega|}A_Nz}
			+ \ip{\frac{\pi N}{|\partial\Omega|}A_Nz - \frac{\pi^2 N^2}{|\partial\Omega|^2}z}\right|
			\\
			& \leq \frac CN\left\|A_N\right\|_{\mathcal{L}\left(\ell^2\right)} + C
			=\mathcal{O}\left(1\right).
		\end{align*}
	Further writing, for each $k\geq 1$,
	\begin{equation*}
		D_N^k-\left(\tilde A_N A_N\right)^k=\sum_{j=0}^{k-1}D_N^j\left(D_N-\tilde A_N A_N\right)\left(\tilde A_N A_N\right)^{k-1-j},
	\end{equation*}
	we obtain
		\begin{align*}
			\left\|D_N^k-\left(\tilde A_N A_N\right)^k\right\|_{\ell^\infty} \hspace{-20mm} &
			\\
			& \leq
			N^{2k-2}
			\sum_{j=0}^{k-1}
			N^{-j}
			\left\|D_N\right\|_{\ell^\infty}^j\left\|D_N-\tilde A_N A_N\right\|_{\ell^\infty}\left\|\tilde A_N\right\|_{\ell^\infty}^{k-1-j}\left\|A_N\right\|_{\ell^\infty}^{k-1-j}
			\\
			& =\mathcal{O}\left(N^{2k-3}\right),
			\\
			\left\|D_N^k-\left(\tilde A_N A_N\right)^k\right\|_{\mathcal{L}\left(\ell^2\right)} \hspace{-20mm} &
			\\
			& \leq
			\sum_{j=0}^{k-1}
			\left\|D_N\right\|_{\mathcal{L}\left(\ell^2\right)}^j\left\|D_N-\tilde A_N A_N\right\|_{\mathcal{L}\left(\ell^2\right)}\left\|\tilde A_N\right\|_{\mathcal{L}\left(\ell^2\right)}^{k-1-j}\left\|A_N\right\|_{\mathcal{L}\left(\ell^2\right)}^{k-1-j}
			\\
			& =\mathcal{O}\left(N^{2k-2}\right).
		\end{align*}
	Therefore, by \eqref{k approx}, since
		\begin{align*}
			\sup_{s,s_*\in\left[0,|\partial\Omega|\right]}\left|\sum_{i,j=1}^N
			\mathds{1}_{\left[\theta_{i}^N, \theta_{i+1}^N\right)}(s)
			\left(D_N^k-\left(\tilde A_N A_N\right)^k\right)_{ij}
			\mathds{1}_{\left[\theta_{j}^N, \theta_{j+1}^N\right)}(s_*)
			\right| \hspace{-20mm}&
			\\
			& \leq
			\left\|D_N^k-\left(\tilde A_N A_N\right)^k\right\|_{\ell^\infty},
		\end{align*}
	we find that, as $N\to\infty$,
	\begin{equation*}
		\left(K_{2k}(x,y)
		-
		\left(\frac{\left|\partial\Omega\right|}{N}\right)^{2k-1}
		\sum_{i,j=1}^N
		\mathds{1}_{\left[\theta_{i}^N, \theta_{i+1}^N\right)}(s)
		\left(D_N^k\right)_{ij}
		\mathds{1}_{\left[\theta_{j}^N, \theta_{j+1}^N\right)}(s_*)
		\right)\to 0,
	\end{equation*}
	in $L_{x,y}^\infty\left(\partial\Omega\times\partial\Omega\right)$.
	
	It follows that, for any fixed $k\geq 1$ and $\eps>0$, provided $N$ is sufficiently large,
		\begin{align*}
			\left\|A^{2k}\right\|_{\mathcal{L}\left(L^2_0\right)}+\eps \hspace{-20mm} &
			\\
			& \geq
			\sup_{\varphi\in L^2_0\left(\partial\Omega\right)}
			\frac{\left\|\left(\frac{\left|\partial\Omega\right|}{N}\right)^{2k-1}
			\sum_{i,j=1}^N
			\mathds{1}_{\left[\theta_{i}^N, \theta_{i+1}^N\right)}(s)
			\left(D_N^k\right)_{ij}
			\int_{\theta_{j}^N}^{\theta_{j+1}^N}\varphi(l(s_*))ds_*\right\|_{L^2_s}}
			{\left\|\varphi(l(s))\right\|_{L^2_s}}
			\\
			& =
			\sup_{\varphi\in L^2_0\left(\partial\Omega\right)}
			\frac{\left(\frac{\left|\partial\Omega\right|}{N}\sum_{i=1}^N \left( \left(\frac{\left|\partial\Omega\right|}{N}\right)^{2k-1}
			\sum_{j=1}^N
			\left(D_N^k\right)_{ij}
			\int_{\theta_{j}^N}^{\theta_{j+1}^N}\varphi(l(s_*))ds_*\right)^2\right)^\frac12}
			{\left\|\varphi(l(s))\right\|_{L^2_s}}
			\\
			& \geq
			\sup_{\substack{z\in\mathbb{R}^N \\ \ip{z}=0}}
			\frac{\left|\partial\Omega\right|^\frac 12
			\left\|
			\left(\frac{\left|\partial\Omega\right|}{N}\right)^{2k}D_N^k z
			\right\|_{\ell^2}}
			{\left\|\sum_{i=1}^N z_i \mathds{1}_{\left[\theta_{i}^N, \theta_{i+1}^N\right)}(s)\right\|_{L^2_s}}
			\\
			& =
			\sup_{\substack{z\in\mathbb{R}^N \\ \ip{z}=0}}
			\frac{
			\left\|
			\left(\frac{\left|\partial\Omega\right|}{N}\right)^{2k}D_N^k z
			\right\|_{\ell^2}}
			{\left\|z\right\|_{\ell^2}}
			=\left\|\left(\frac{\left|\partial\Omega\right|}{N}\right)^{2k}D_N^k\right\|_{\mathcal{L}\left(\ell^2_0\right)}.
		\end{align*}
	Further deducing from estimate \eqref{gelfand} that there exist $k_0\geq 1$ and $\delta>0$ such that $\left\|A^{2k_0}\right\|_{\mathcal{L}\left(L^2_0\right)}^\frac 1{2k_0} \leq \pi-3\delta$, we infer that, setting $\eps>0$ sufficiently small,
	\begin{equation*}
		\left\|\left(\frac{\left|\partial\Omega\right|}{N}\right)^{2k_0}D_N^{k_0}\right\|_{\mathcal{L}\left(\ell^2_0\right)}
		\leq
		\left\|A^{2k_0}\right\|_{\mathcal{L}\left(L^2_0\right)}+\eps
		\leq \left(\pi -3\delta\right)^{2k_0} + \eps \leq \left(\pi -2\delta\right)^{2k_0},
	\end{equation*}
	for $N$ sufficiently large.

	Now, for any $k\geq k_0$, we write $k=pk_0+q$ with positive integral numbers and $0\leq q\leq k_0-1$. Then, we obtain
		\begin{align*}
			\left\|\left(\left(\frac{\left|\partial\Omega\right|}{N}\right)^{2}D_N\right)^{k}\right\|_{\mathcal{L}\left(\ell^2_0\right)}
			& \leq
			\left\|\left(\left(\frac{\left|\partial\Omega\right|}{N}\right)^{2}D_N\right)^{k_0}\right\|_{\mathcal{L}\left(\ell^2_0\right)}^p
			\left\|\left(\frac{\left|\partial\Omega\right|}{N}\right)^{2}D_N\right\|_{\mathcal{L}\left(\ell^2_0\right)}^q
			\\
			& \leq \left(\pi -2\delta\right)^{2\left(k-q\right)}
			\left\|\left(\frac{\left|\partial\Omega\right|}{N}\right)^{2}D_N\right\|_{\mathcal{L}\left(\ell^2_0\right)}^q.
		\end{align*}
	Note that the preceding step could not possibly be performed for $\tilde A_NA_N$, for this operator does not in general preserve the subspace $\ell^2_0\subset\ell^2$, which justifies the introduction of $D_N$. Further using that $N^{-2}D_N$ is a bounded operator over $\ell^2$ uniformly in $N$, we arrive at, for some fixed constant $C_*>0$ independent of $N$ and $k$, and for sufficiently large $k$,
	\begin{equation*}
		\left\|\left(\left(\frac{\left|\partial\Omega\right|}{N}\right)^{2}D_N\right)^{k}\right\|_{\mathcal{L}\left(\ell^2_0\right)}
		\leq C_* \left(\pi -2\delta\right)^{2k}\leq \left(\pi -\delta\right)^{2k},
	\end{equation*}
	which, upon exchanging the roles of $(s_{1}^N,\dots , s_{N}^N)$ and $(\tilde s_{1}^N, \dots , \tilde s_{N}^N)$ to obtain an equivalent estimate on $\tilde D_N$, concludes the proof of the lemma.
\end{proof}

\subsection{Solving the discrete system \eqref{point toy}}

Combining the preceding results allows us to get the existence and the uniqueness of the solution to \eqref{point toy}. For mere convenience of notation, we introduce the matrix
\begin{equation*}
	B_{N-1,N} := 
	\left(
	\frac{l\left(\tilde s_{i}^N\right) - l\left(s_{j}^N\right)}
	{\left|l\left(\tilde s_{i}^N\right) - l\left(s_{j}^N\right)\right|^2}\cdot \tau\left(l\left(\tilde s_{i}^N\right)\right)
	\right)_{\substack{1\leq i\leq N-1 \\ 1\leq j\leq N }}.
\end{equation*}

\begin{proposition}\label{inverse perfect}
	For any $N\geq 2$, consider a well distributed mesh $(s_{1}^N,\dots , s_{N}^N)\in \left[0,\left|\partial\Omega\right|\right)^N$, $(\tilde s_{1}^N, \dots , \tilde s_{N}^N)\in \left[0,\left|\partial\Omega\right|\right)^N$. Then, provided $N$ is sufficiently large, the following problem:
	\begin{equation}\label{prob}
		z\in \R^N,\quad \frac{\left|\partial\Omega\right|}N B_{N-1,N}z=v, \quad \langle z \rangle = \gamma,
	\end{equation}
	has a unique solution for any given $v\in \R^{N-1}$ and $\gamma \in \R$. Moreover, this solution satisfies:
	\begin{equation}\label{est l1}
	\| z \|_{\ell^1}\leq \| z \|_{\ell^2}
	\leq
	C\left(\| v \|_{\ell^2} + |\gamma| + \sqrt N \left|\ip{v}\right|\right)
	\leq
	C\left(\| v \|_{\ell^\infty} + |\gamma| + \sqrt N \left|\ip{v}\right|\right),
	\end{equation}
	for some independent constant $C>0$.
\end{proposition}

\begin{proof}
	Let the operators $E_N,\tilde E_N:\ell^2\to\ell^2_0$ be defined by
	\begin{equation*}
		E_Nz=B_Nz-\ip{B_Nz}\mathbf{1}
		\quad\text{and}\quad
		\tilde E_Nz=\tilde B_Nz-\ip{\tilde B_Nz}\mathbf{1},
	\end{equation*}
	for all $z\in\mathbb{R}^N$, where $\mathbf{1}=(1,\ldots,1)\in\mathbb{R}^N$. In view of Corollary \ref{boundedness}, $N^{-1}E_N$ and $N^{-1}\tilde E_N$ are bounded uniformly in both $\mathcal{L}\left(\ell^2\right)$ and $\mathcal{L}\left(\ell^2_0\right)$.
	
	Next, by \eqref{mean approx} and \eqref{PB approx} from Proposition \ref{important}, we obtain that
		\begin{align*}
			\left\|
			\frac{\left|\partial\Omega\right|^2}{N^2}
			\tilde E_NE_N z
			-
			\left(
			\frac{\left|\partial\Omega\right|^2}{N^2} D_N - \pi^2
			\right) z
			\right\|_{\ell^2}
			\hspace{-40mm}&
			\\
			\leq & \left\|
			\frac{\left|\partial\Omega\right|^2}{N^2}
			\tilde B_NB_N z
			-
			\left(
			\frac{\left|\partial\Omega\right|^2}{N^2} D_N - \pi^2
			\right) z
			\right\|_{\ell^2}
			\\
			& +\frac{\left|\partial\Omega\right|^2}{N^2}
			\left(
			\left|\ip{B_N z}\right|\left\|\tilde B_N\mathbf{1}\right\|_{\ell^2}
			+\left|\ip{\tilde B_N B_N z}\right|
			+\left|\ip{\tilde B_N z}\right|\left|\ip{B_N z}\right|
			\right)
			\\
			\leq & \left\|
			\frac{\left|\partial\Omega\right|^2}{N^2}
			\left( \tilde B_NB_N- \tilde A_N A_N\right) z
			+\pi^2 z
			\right\|_{\ell^2} + \frac{C}{N^2}\left\|z\right\|_{\ell^2}
			\\
			& +
			\frac{\left|\partial\Omega\right|}{N}
			\left|
			\ip{\frac{|\partial\Omega|}{N}\tilde A_N A_N z - \pi A_Nz}
			\right|
			+
			\pi
			\left|
			\ip{\frac{|\partial\Omega|}{N}A_Nz - \pi z}
			\right|
			\\
			\leq & \frac{C}{N}\left\|z\right\|_{\ell^2},
		\end{align*}
	where $C>0$ denotes various independent constants. Therefore, by Lemma \ref{inverse D}, we infer that, provided $N$ is sufficiently large,
	\begin{equation*}
		\left\|
		\left(
		\frac{\left|\partial\Omega\right|^2}{N^2} D_N - \pi^2
		\right)^{-1}
		\frac{\left|\partial\Omega\right|^2}{N^2}
		\tilde E_N E_N z
		- z
		\right\|_{\ell^2_0}
		\leq \frac 12 \left\|z\right\|_{\ell^2_0},
	\end{equation*}
	for all $z\in\ell^2_0$, which allows us to deduce, using yet another absolutely convergent Neumann series, that the operator $\left(
		\frac{\left|\partial\Omega\right|^2}{N^2} D_N - \pi^2
		\right)^{-1}
		\frac{\left|\partial\Omega\right|^2}{N^2}
		\tilde E_N E_N$ has an inverse in $\mathcal{L}\left(\ell^2_0\right)$, which is uniformly bounded in $N$ and given by
		\begin{align*}
			\left(\left(
			\frac{\left|\partial\Omega\right|^2}{N^2} D_N - \pi^2
			\right)^{-1}
			\frac{\left|\partial\Omega\right|^2}{N^2}
			\tilde E_N E_N\right)^{-1} \hspace{-20mm}&
			\\
			& =\sum_{k=0}^\infty
			\left(1-
			\left(
			\frac{\left|\partial\Omega\right|^2}{N^2} D_N - \pi^2
			\right)^{-1}
			\frac{\left|\partial\Omega\right|^2}{N^2}
			\tilde E_N E_N
			\right)^k.
		\end{align*}
	It therefore follows that, for large $N$, the operator $\frac{\left|\partial\Omega\right|}{N}E_N\in\mathcal{L}\left(\ell^2_0\right)$ is invertible with an inverse uniformly bounded in $N$ (in the operator norm over $\ell^2_0$) and given by
		\begin{align*}
			\left(\frac{\left|\partial\Omega\right|}{N}E_N\right)^{-1}
			\hspace{-10mm}&
			\\
			& =
			\left(\left(
			\frac{\left|\partial\Omega\right|^2}{N^2} D_N - \pi^2
			\right)^{-1}
			\frac{\left|\partial\Omega\right|^2}{N^2}
			\tilde E_N E_N\right)^{-1}
			\left(
			\frac{\left|\partial\Omega\right|^2}{N^2} D_N - \pi^2
			\right)^{-1}
			\frac{\left|\partial\Omega\right|}{N}
			\tilde E_N
			\\
			& \in\mathcal{L}\left(\ell^2_0\right).
		\end{align*}

	Now, let us define
		\begin{align*}
			\Phi\ :\quad \R^N &\to \R^{N+1}\\
			 z &\mapsto  
			\begin{pmatrix}
				\frac{\left|\partial\Omega\right|}N B_{N} z \\ \langle z \rangle
			\end{pmatrix},
		\end{align*}
	and let us suppose that $\Phi z=0$ for some $z\in\mathbb{R}^N$. In particular, one has that $z\in\ell^2_0$ and $\frac{\left|\partial\Omega\right|}{N}E_Nz=0$, whence $z=0$. Therefore, $\Phi$ is an injective linear mapping. In particular, it is bijective from $\R^N$ onto $\operatorname{Im}\Phi$, so that $\dim \left(\operatorname{Im} \Phi\right)=N$ and there exist vectors $r_N=\left(r_N^1,\ldots,r_N^N,r_N^{N+1}\right)=\left(r_N',r_N^{N+1}\right)\in\mathbb{R}^{N+1}$ such that
	\begin{equation*}
		\operatorname{Im}\Phi=\set{u\in\mathbb{R}^{N+1}}{r_N\cdot u=0}.
	\end{equation*}
	Without any loss of generality, we impose that the $r_N$'s satisfy
	\begin{equation}\label{scaling}
		r_N\cdot
		\begin{pmatrix}
			\mathbf{1}\\0
		\end{pmatrix}
		=N\ip{r_N'}
		\geq 0
		\quad\text{and}\quad
		\left\|r_N'\right\|_{\ell^2}+\left|r_N^{N+1}\right|=1.
	\end{equation}

	Observe that there is a unique $z_N\in\mathbb{R}^N$ such that
	\begin{equation*}
		\Phi (z_N)
		=
		\begin{pmatrix}
			\mathbf{1}\\0
		\end{pmatrix}
		-\frac{r_N'\cdot\mathbf{1}}{r_N\cdot r_N}r_N \in \operatorname{Im}\Phi.
	\end{equation*}
	On the one hand, since $\frac{\left|\partial\Omega\right|}{N}E_N\in\mathcal{L}\left(\ell^2_0\right)$ is invertible, we have the estimate
		\begin{align*}
			\left\|z_N\right\|_{\ell^2}
			& \leq C\left\|\frac{\left|\partial\Omega\right|}{N}E_N\left(z_N-\ip{z_N}\mathbf{1}\right)\right\|_{\ell^2_0}
			+\left|\ip{z_N}\right|
			\leq C\left\|\frac{\left|\partial\Omega\right|}{N}E_N\left(z_N\right)\right\|_{\ell^2_0}
			+C\left|\ip{z_N}\right|
			\\
			& = C\left\|\mathbf{1}-\frac{r_N'\cdot\mathbf{1}}{r_N\cdot r_N}r_N'-\ip{\mathbf{1}-\frac{r_N'\cdot\mathbf{1}}{r_N\cdot r_N}r_N'}\mathbf{1}\right\|_{\ell^2_0}
			+C\left|\frac{r_N'\cdot\mathbf{1}}{r_N\cdot r_N}r_N^{N+1}\right|
			\\
			& = C\left\|\frac{r_N'\cdot\mathbf{1}}{r_N\cdot r_N}r_N'-\frac{\left(r_N'\cdot\mathbf{1}\right)^2}{N r_N\cdot r_N}\mathbf{1}\right\|_{\ell^2_0}
			+C\left|\frac{r_N'\cdot\mathbf{1}}{r_N\cdot r_N}r_N^{N+1}\right|
			\\
			& \leq C\frac{\left\|r_N'\right\|_{\ell^2}^2 + \left\|r_N'\right\|_{\ell^2}\left|r_N^{N+1}\right|}{\left\|r_N\right\|_{\ell^2}^2}
			\leq \frac{C}{\left\|r_N\right\|_{\ell^2}^2},
		\end{align*}
	where $C>0$ denotes various constants independent of $N$, while, on the other hand, using \eqref{mean approx}, we find that
		\begin{align*}
			\frac{1}{\left\|r_N\right\|_{\ell^2}^2}
			\left\|
			r_N-\ip{r_N'}
			\begin{pmatrix}
				\mathbf{1}\\0
			\end{pmatrix}
			\right\|_{\ell^2}^2
			&
			=
			\left|
			1-\frac{\left(r_N'\cdot\mathbf{1}\right)^2}{Nr_N\cdot r_N}
			\right|
			\\
			& =
			\left|\ip{\mathbf{1}-\frac{r_N'\cdot\mathbf{1}}{r_N\cdot r_N}r_N'}\right|
			\\
			& =
			\left|\ip{\frac{\left|\partial\Omega\right|}{N}B_N(z_N)}\right|
			\leq \frac{C}{N^2} \left\|z_N\right\|_{\ell^1}.
		\end{align*}
	Combining the preceding estimates yields
	\begin{equation}\label{r to 1 l2}
			\left\|
			r_N-\ip{r_N'}
			\begin{pmatrix}
				\mathbf{1}\\0
			\end{pmatrix}
			\right\|_{\ell^2}^2
			=\mathcal{O}\left(N^{-2}\right),
	\end{equation}
	whence, using that $0\leq \ip{r_N'}\leq 1$,
	\begin{equation}\label{rN mean}
		\left|\left\|r_N\right\|_{\ell^2}-\ip{r_N'}\right|=\mathcal{O}\left(N^{-1}\right),
	\end{equation}
	and, since $\left\|z\right\|_{\ell^\infty}\leq \sqrt N\left\|z\right\|_{\ell^2}$, for any $z\in\mathbb{R}^N$,
	\begin{equation*}
			\left\|
			r_N-\ip{r_N'}
			\begin{pmatrix}
				\mathbf{1}\\0
			\end{pmatrix}
			\right\|_{\ell^\infty}
			=\mathcal{O}\left(N^{-\frac 12}\right).
	\end{equation*}
	In particular, we further deduce that $r_N^{N+1}=\mathcal{O}\left(N^{-\frac 12}\right)$. It therefore holds, by \eqref{scaling}, that $\left\|r_N'\right\|_{\ell^2}=1+\mathcal{O}\left(N^{-\frac 12}\right)$ and $\left\|r_N\right\|_{\ell^2}=1+\mathcal{O}\left(N^{-\frac 12}\right)$ so that $\ip{r_N'}=1+\mathcal{O}\left(N^{-\frac 12}\right)$, as well, by \eqref{rN mean}. On the whole, we conclude that
	\begin{equation*}
			\left\|
			r_N-
			\begin{pmatrix}
				\mathbf{1}\\0
			\end{pmatrix}
			\right\|_{\ell^\infty}
			=\mathcal{O}\left(N^{-\frac 12}\right).
	\end{equation*}
	In particular, considering sufficiently large values of $N$, we henceforth assume that all components of $r_N'$ are uniformly bounded away from zero.

	Now, let $v\in \R^{N-1}$ and $\gamma \in \R$ be fixed. There exists a unique $v_{N}$ such that
	$
	\begin{pmatrix}
	v \\ v_{N}\\ \gamma
	\end{pmatrix}
	\in \operatorname{Im}\Phi
	$, namely $v_{N}=-\frac{1}{r_N^N}\sum_{i=1}^{N-1}r_N^{i}v_{i} - \frac{r_N^{N+1}}{r_N^N}\gamma$. With this $v_{N}$, we then deduce the existence of $z\in \R^N$ such that $\Phi(z)=\begin{pmatrix}
	v\\v_{N}\\\gamma
	\end{pmatrix}
	$. In particular $z$ is a solution to \eqref{prob} and, by invertibility of $\frac{\left|\partial\Omega\right|}{N}E_N\in\mathcal{L}\left(\ell^2_0\right)$, it holds that
		\begin{align*}
			\left\|z-\gamma\mathbf{1}\right\|_{\ell^2_0}
			& =
			\left\|z-\ip{z}\mathbf{1}\right\|_{\ell^2_0}
			\leq C\left\|\frac{\left|\partial\Omega\right|}{N}E_N\left(z-\ip{z}\mathbf{1}\right)\right\|_{\ell^2_0}
			\\
			& =
			C\left\|
			\begin{pmatrix}
				v \\ v_N
			\end{pmatrix}
			-\gamma\frac{\left|\partial\Omega\right|}{N}B_N \mathbf{1}
			-
			\ip{\begin{pmatrix}
				v \\ v_N
			\end{pmatrix}}
			+\gamma\ip{\frac{\left|\partial\Omega\right|}{N}B_N \mathbf{1}}
			\right\|_{\ell^2_0}
			\\
			& \leq C
			\left(
			\left(\frac 1N\sum_{i=1}^N |v_i|^2\right)^\frac 12 + |\gamma|
			\right)
			\\
			& = C
			\left(
			\left(\frac 1N\sum_{i=1}^{N-1} |v_i|^2
			+\frac 1N\left|\frac{1}{r_N^N}\sum_{i=1}^{N-1}r_N^{i}v_{i} + \frac{r_N^{N+1}}{r_N^N}\gamma\right|^2\right)^\frac 12 + |\gamma|
			\right)
			\\
			& \leq C
			\left(\left\|v\right\|_{\ell^2} +
			\frac 1{\sqrt N}\left|\sum_{i=1}^{N-1}r_N^{i}v_{i}\right| + |\gamma|
			\right).
		\end{align*}
	As $\| z \|_{\ell^2}-|\gamma| \leq  \| z -\gamma\mathbf{1}\|_{\ell^2_0}$ and
		\begin{align*}
			\frac 1{\sqrt N}\left|\sum_{i=1}^{N-1}r_N^{i}v_{i}\right|
			& \leq
			\frac 1{\sqrt N}\left|\sum_{i=1}^{N-1}\left(r_N^{i}-\ip{r_N'}\right)v_{i}\right|
			+\frac{N-1}{\sqrt N}\left|\ip{r_N'}\ip{v}\right|
			\\
			& \leq
			\sqrt{\frac{N-1}{N}}\left\|v\right\|_{\ell^2}
			\left(\sum_{i=1}^{N-1}\left(r_N^{i}-\ip{r_N'}\right)^2\right)^\frac 12
			+\frac{N-1}{\sqrt N}\left|\ip{r_N'}\ip{v}\right|
			\\
			& =
			\mathcal{O}\left(N^{-\frac 12}\right)\left\|v\right\|_{\ell^2} + \mathcal{O}\left(N^\frac 12\right)\left|\ip{v}\right|,
		\end{align*}
	where we have used \eqref{r to 1 l2}, we conclude that
	\begin{equation*}
		\| z \|_{\ell^2} \leq C\left( \| v \|_{\ell^2} + |\gamma| + \sqrt N \left|\ip{v}\right| \right).
	\end{equation*}

	Finally, concerning the uniqueness of a solution to \eqref{prob}, let us consider $z$ and $\tilde z$ two solutions of \eqref{prob}. Then, $\Phi(z-\tilde z)=\begin{pmatrix}
	0_{\R^{N-1}}\\ x\\ 0
	\end{pmatrix}$ (for some $x\in \R$) belongs to $\operatorname{Im}\Phi$ if only if $x= 0$. By injectivity of $\Phi$, we conclude that necessarily $z=\tilde z$, thereby completing the proof of the proposition.
\end{proof}

\section{Weak convergence of discretized singular integral operators}\label{sect:conv}

The results in this section will serve to show that $(u_{R}-u_{\rm app}^N)\cdot n\vert_{\partial \Omega}$ vanishes in a weak sense.

The coming proposition establishes some weak convergence of the discretization of the singular integral operator $B$ defined in \eqref{AB}.

\begin{proposition}\label{prop 32}
	For any $N\geq 2$, consider a well distributed mesh $(s_{1}^N,\dots , s_{N}^N)\in \left[0,\left|\partial\Omega\right|\right)^N$, $(\tilde s_{1}^N, \dots , \tilde s_{N}^N)\in \left[0,\left|\partial\Omega\right|\right)^N$ satisfying \eqref{mesh2} and, according to Proposition \ref{inverse perfect}, consider the solution $\gamma^N=(\gamma_{1}^N,\dots, \gamma_{N}^N)\in\mathbb{R}^N$ to the system \eqref{point toy} for some periodic function $f\in C^{\kappa}\left(\left[0,\left|\partial\Omega\right|\right]\right)$, where $\kappa \geq 2$ is introduced in \eqref{mesh2}, with zero mean value $\int_0^{\left|\partial\Omega\right|} f(s)ds=0$ and some $\gamma\in\mathbb{R}$. We define the approximations
	\begin{equation}\label{f app}
		\begin{aligned}
			f_{\rm app}^N(s) & :=
			\frac1{N}\sum_{j=1}^N \gamma_{j}^N \frac{l\left(s\right) - l\left(s_{j}^N\right)}
			{\left|l\left(s\right) - l\left(s_{j}^N\right)\right|^2}\cdot \tau\left(l\left(s\right)\right),
			\\
			g_{\rm app}^N(s) & :=
			\frac1{N}\sum_{j=1}^N \gamma_{j}^N \frac{l\left(s\right) - l\left(s_{j}^N\right)}
			{\left|l\left(s\right) - l\left(s_{j}^N\right)\right|^2}\cdot n\left(l\left(s\right)\right).
		\end{aligned}
	\end{equation}

	Then, for any periodic test function $\varphi\in C^{\infty}\left(\left[0,\left|\partial\Omega\right|\right]\right)$,
		\begin{align*}
			\left|\int_{0}^{\left|\partial\Omega\right|} (f_{\rm app}^N - f )\varphi\right|
			& \leq
			\frac{C}{N^{\kappa}}
			\left(\left\|f\right\|_{C^{\kappa}}+|\gamma|
			\right)
			\left\|\varphi\right\|_{C^{\kappa+1}},
			\\
			\left|\int_{0}^{\left|\partial\Omega\right|} (g_{\rm app}^N - AB^{-1}f -\pi\gamma H\cdot\tau )\varphi\right|
			& \leq
			\frac{C}{N^{\kappa}}
			\left(\left\|f\right\|_{C^{\kappa}}+|\gamma|
			\right)
			\left\|\varphi\right\|_{L^2},
		\end{align*}
	where we identify the variable $x$ with the variable $s$ whenever $x=l(s)\in\partial\Omega$, the singular integrals are defined in the sense of Cauchy's principal value and $H$ is given by the limiting values from $\Omega$ of the harmonic vector field defined by \eqref{harmonic}.
\end{proposition}

\begin{proof}
	Let $\varphi \in C^{\infty}\left(\left[0,\left|\partial\Omega\right|\right]\right)$ be a periodic test function. Then, we decompose
		\begin{align*}
			\int_{0}^{\left|\partial\Omega\right|} (f_{\rm app}^N - f )\varphi
			=& \left(
			\int_{0}^{\left|\partial\Omega\right|} f_{\rm app}^N \varphi - \frac{\left|\partial\Omega\right|}{N} \sum_{i=1}^{N} f_{\rm app}^N(\tilde s_{i}^N) \varphi(\tilde s_{i}^N)
			\right)\\
			&- \left(
			\int_{0}^{\left|\partial\Omega\right|} f \varphi- \frac{\left|\partial\Omega\right|}{N} \sum_{i=1}^{N} f(\tilde s_{i}^N) \varphi(\tilde s_{i}^N)
			\right)\\
			&+
			\frac{\left|\partial\Omega\right|}{N} \sum_{i=1}^{N-1} \left( f_{\rm app}^N(\tilde s_{i}^N)- f(\tilde s_{i}^N) \right) \varphi(\tilde s_{i}^N)
			\\
			&+
			\frac{\left|\partial\Omega\right|}{N} \left( f_{\rm app}^N(\tilde s_{N}^N)- f(\tilde s_{N}^N) \right) \varphi(\tilde s_{N}^N)
			\\
			=&: D_1 + D_2 + D_3 + D_4.
		\end{align*}
	It is readily seen that $D_3$ is null, for $f_{\rm app}^N(\tilde s_{i}^N) = f(\tilde s_{i}^N)$, for all $i=1,\ldots,N-1$, by construction (see \eqref{point toy}).

	Next, note that $D_2$ is the error of approximation of the integral $\int_{0}^{\left|\partial\Omega\right|} f \varphi$ by its Riemann sum. Therefore, a direct application of Corollary \ref{riemann2} yields
	\begin{equation}\label{d2}
		|D_2|\leq \frac{C}{N^{\kappa}} \left\|f\varphi\right\|_{C^{\kappa}}
		\leq\frac{C}{N^{\kappa}} \left\|f\right\|_{C^{\kappa}}\left\|\varphi\right\|_{C^{\kappa}}.
	\end{equation}

	As for the term $D_1$, it is first rewritten, exploiting \eqref{B1} and \eqref{average adjoint}, as
		\begin{align*}
			D_1  =& \int_{0}^{\left|\partial\Omega\right|} f_{\rm app}^N \varphi - \frac{\left|\partial\Omega\right|}{N} \sum_{i=1}^{N} f_{\rm app}^N(\tilde s_{i}^N) \varphi(\tilde s_{i}^N)
			\\
			= & \frac1{N}\sum_{j=1}^N \gamma_{j}^N
			\int_0^{\left|\partial\Omega\right|}
			\frac{l\left(s\right) - l\left(s_{j}^N\right)}
			{\left|l\left(s\right) - l\left(s_{j}^N\right)\right|^2}\cdot \tau\left(l\left(s\right)\right)
			\varphi(s)ds
			\\
			& - \frac{\left|\partial\Omega\right|}{N^2} \sum_{i,j=1}^{N} \gamma_{j}^N
			\frac{l\left(\tilde s_{i}^N\right) - l\left(s_{j}^N\right)}
			{\left|l\left(\tilde s_{i}^N\right) - l\left(s_{j}^N\right)\right|^2}\cdot \tau\left(l\left(\tilde s_{i}^N\right)\right)
			\varphi\left(\tilde s_{i}^N\right)
			\\
			= & \int_0^{\left|\partial\Omega\right|} F(s){d}s
			- \frac{\left|\partial\Omega\right|}{N} \sum_{i=1}^N
			F(\tilde s_i^N)
			\\
			& + \frac1{N}\sum_{j=1}^N \gamma_{j}^N
			\underbrace{\left(\int_0^{\left|\partial\Omega\right|}
			\frac{l\left(s\right) - l\left(s_{j}^N\right)}
			{\left|l\left(s\right) - l\left(s_{j}^N\right)\right|^2}\cdot \tau\left(l\left(s\right)\right)
			ds\right)}_{=0\quad\text{by \eqref{B1}}}
			\varphi\left(s_j^N\right)
			\\
			& - \frac 1N \sum_{j=1}^{N} \gamma_{j}^N
			\underbrace{
			\left(\frac{\left|\partial\Omega\right|}{N}\sum_{i=1}^N\frac{l\left(\tilde s_{i}^N\right) - l\left(s_{j}^N\right)}
			{\left|l\left(\tilde s_{i}^N\right) - l\left(s_{j}^N\right)\right|^2}\cdot \tau\left(l\left(\tilde s_{i}^N\right)\right)
			\right)
			}_{=\mathcal{O}\left(N^{-\kappa}\right)\quad \text{in }\ell^\infty\text{ by \eqref{average adjoint}}}
			\varphi\left(s_j^N\right)
			\\
			 =&\int_0^{\left|\partial\Omega\right|} F(s){d}s
			- \frac{\left|\partial\Omega\right|}{N} \sum_{i=1}^N
			F(\tilde s_i^N)
			+\mathcal{O}\left(\frac {\left\|\gamma^N\right\|_{\ell^1}\left\|\varphi\right\|_{L^\infty}}{N^\kappa}\right),
		\end{align*}
	where
	\begin{equation*}
		F(s)=
		\frac1{N}\sum_{j=1}^N \gamma_{j}^N
		\frac{l\left(s\right) - l\left(s_{j}^N\right)}
		{\left|l\left(s\right) - l\left(s_{j}^N\right)\right|^2}\cdot \tau\left(l\left(s\right)\right)
		\left(\varphi(s)-\varphi\left(s_j^N\right)\right).
	\end{equation*}
	Note that the integrand $s\mapsto \frac{l\left(s\right) - l\left(s_{j}^N\right)}
		{\left|l\left(s\right) - l\left(s_{j}^N\right)\right|^2}\cdot \tau\left(l\left(s\right)\right)
		\left(\varphi(s)-\varphi\left(s_j^N\right)\right)$ above is now regular, thus assuring that the corresponding Riemann sums converge. It therefore follows from Corollary \ref{riemann2} that
		\begin{align*}
			|D_1|
			\leq & \frac{C}{N^{\kappa}}\left\|F\right\|_{C^{\kappa}}
			+\frac{C}{N^\kappa}\left\|\gamma^N\right\|_{\ell^1}\left\|\varphi\right\|_{L^\infty}
			\\
			\leq & \frac{C}{N^{\kappa}}\left\|\gamma^N\right\|_{\ell^1}
			\\
			& \times\sup_{s_*\in \left[0,\left|\partial\Omega\right|\right]}
			\left\|\frac{(s-s_*)(l\left(s\right) - l\left(s_*\right))}
			{\left|l\left(s\right) - l\left(s_*\right)\right|^2}\cdot \tau\left(l\left(s\right)\right)
			\right\|_{C^{\kappa}_s
			\left(\left[0,\left|\partial\Omega\right|\right]\right)}\left\|\varphi'\right\|_{C^{\kappa}}
			\\
			& +\frac{C}{N^\kappa}\left\|\gamma^N\right\|_{\ell^1}\left\|\varphi\right\|_{L^\infty}
			\\
			\leq & \frac{C}{N^{\kappa}}\left\|\gamma^N\right\|_{\ell^1}\left\|\varphi\right\|_{C^{\kappa+1}}.
		\end{align*}
	Then, further utilizing estimate \eqref{est l1}, Corollary \ref{riemann2}, that $\kappa\geq\frac 12$ and the fact that $f$ has zero mean value, we infer
	\begin{equation}\label{d1}
		\begin{aligned}
			|D_1|
			& \leq \frac{C}{N^{\kappa}}
			\left(\left\|f\right\|_{L^\infty}+|\gamma|+{\sqrt N}\left|\frac 1{N-1}\sum_{i=1}^{N-1}f(\tilde s_i^N)\right|\right)
			\left\|\varphi\right\|_{C^{\kappa+1}}
			\\
			& \leq \frac{C}{N^{\kappa}}
			\left(\left\|f\right\|_{L^\infty}+|\gamma|
			+{\sqrt N}\left|\int_0^{\left|\partial\Omega\right|}f(s){d}s - \frac {\left|\partial\Omega\right|}{N}\sum_{i=1}^{N}f(\tilde s_i^N)\right|\right)
			\left\|\varphi\right\|_{C^{\kappa+1}}
			\\
			& \leq \frac{C}{N^{\kappa}}
			\left(\left\|f\right\|_{C^{\kappa}}+|\gamma|
			\right)
			\left\|\varphi\right\|_{C^{\kappa+1}}.
		\end{aligned}
	\end{equation}

	Finally, regarding $D_4$, recalling that, by \eqref{point toy},
		\begin{align*}
			\sum_{i=1}^{N-1}f(\tilde s_{i}^N) & =
			\frac1{N}\sum_{i=1}^{N-1}\sum_{j=1}^N \gamma_{j}^N \frac{l\left(\tilde s_{i}^N\right) - l\left(s_{j}^N\right)}
			{\left|l\left(\tilde s_{i}^N\right) - l\left(s_{j}^N\right)\right|^2}\cdot \tau\left(l\left(\tilde s_{i}^N\right)\right)
			\\
			& =
			\ip{B_N\gamma^N}-
			\frac1{N}\sum_{j=1}^N \gamma_{j}^N \frac{l\left(\tilde s_{N}^N\right) - l\left(s_{j}^N\right)}
			{\left|l\left(\tilde s_{N}^N\right) - l\left(s_{j}^N\right)\right|^2}\cdot \tau\left(l\left(\tilde s_{N}^N\right)\right)
			\\
			& =
			\ip{B_N\gamma^N}-f_{\rm app}^N\left(\tilde s_N^N\right),
		\end{align*}
	we find
		\begin{align*}
			D_4 & =\frac{\left|\partial\Omega\right|}{N} \left( f_{\rm app}^N(\tilde s_{N}^N)- f(\tilde s_{N}^N) \right) \varphi(\tilde s_{N}^N)
			\\
			& =\left(\ip{\frac{\left|\partial\Omega\right|}{N}B_N\gamma^N}
			-\frac{\left|\partial\Omega\right|}{N} \sum_{i=1}^N f(\tilde s_{i}^N)\right)\varphi(\tilde s_{N}^N) \\
			& =\left(\ip{\frac{\left|\partial\Omega\right|}{N}B_N\gamma^N}
			+\int_0^{\left|\partial\Omega\right|}f(s)ds-\frac{\left|\partial\Omega\right|}{N} \sum_{i=1}^N f(\tilde s_{i}^N)\right)
			\varphi(\tilde s_{N}^N).
		\end{align*}
	Hence, utilizing \eqref{mean approx} and Corollary \ref{riemann2} again,
	\begin{equation*}
		\left|D_4\right|
		\leq \left(\frac{C}{N^\kappa}\left\|\gamma^N\right\|_{\ell^1}
		+ \frac{C}{N^{\kappa}}\|f \|_{C^{\kappa}}\right)\left\|\varphi\right\|_{L^\infty}.
	\end{equation*}
	Therefore, repeating the control of $\left\|\gamma^N\right\|_{\ell^1}$ performed in \eqref{d1} and based on \eqref{est l1}, we arrive at
	\begin{equation}\label{d4}
		\left|D_4\right|
		\leq
		\frac{C}{N^{\kappa}}
		\left(\left\|f\right\|_{C^{\kappa}}+|\gamma|
		\right)
		\left\|\varphi\right\|_{L^\infty}.
	\end{equation}

	On the whole, since $D_3=0$, combining \eqref{d2}, \eqref{d1} and \eqref{d4}, we deduce that
	\begin{equation}\label{fappestimate}
		\left|\int_{0}^{\left|\partial\Omega\right|} (f_{\rm app}^N - f )\varphi\right|
		\leq
		\frac{C}{N^{\kappa}}
		\left(\left\|f\right\|_{C^{\kappa}}+|\gamma|
		\right)
		\left\|\varphi\right\|_{C^{\kappa+1}},
	\end{equation}
	which concludes the convergence estimate for $f_{\rm app}^N$.

	As for $g_{\rm app}^N$, we first write
	\begin{equation*}
		\int_{0}^{\left|\partial\Omega\right|} g_{\rm app}^N \varphi
		=
		\frac1{N}\sum_{j=1}^N \gamma_{j}^N
		A^*\varphi\left(l\left(s_j^N\right)\right),
	\end{equation*}
	where we identify $\varphi(x)$ with $\varphi(s)$ whenever $x=l(s)\in\partial\Omega$.
	
	Now, recall from Section \ref{vortex sheet} that $B\in\mathcal{L}\left(L^2_0\right)$ is invertible. Moreover, it is readily seen, by \eqref{adjoint A} and \eqref{adjoint B}, that the adjoint operators of $A$ and $B$ over $L^2_0\left(\partial\Omega\right)$ are respectively given by
		\begin{align*}
			A^\# \varphi(x) & = A^*\varphi(x) - \frac{1}{\left|\partial\Omega\right|}\int_{\partial\Omega} A^*\varphi(y) dy,
			\\
			B^\# \varphi(x) & = B^*\varphi(x) - \frac{1}{\left|\partial\Omega\right|}\int_{\partial\Omega} B^*\varphi(y) dy.
		\end{align*}
	It follows that $B^\#\in\mathcal{L}\left(L^2_0\right)$ is invertible, as well.
	
	We then obtain, by \eqref{harmonic4} and utilizing the adjointness \eqref{adjoint A} and \eqref{adjoint B} of $A$ and $B$ in the $L^2\left(\partial\Omega\right)$ structure,
		\begin{align*}
			\int_{0}^{\left|\partial\Omega\right|} g_{\rm app}^N \varphi
			= &
			\frac1{N}\sum_{j=1}^N \gamma_{j}^N
			A^\#\varphi\left(l\left(s_j^N\right)\right)
			+ \frac{\gamma}{\left|\partial\Omega\right|}\int_{\partial\Omega} A^*\varphi
			\\
			= &
			\frac1{N}\sum_{j=1}^N \gamma_{j}^N
			\left[B^*\left(B^\#\right)^{-1}A^\#\varphi\right]\left(l\left(s_j^N\right)\right)
			\\
			& - \frac{\gamma}{\left|\partial\Omega\right|}\int_{\partial\Omega}
			B^*\left(B^\#\right)^{-1}A^\#\varphi
			+ \frac{\gamma}{\left|\partial\Omega\right|}\int_{\partial\Omega} A^*\varphi
			\\
			= &
			\int_{0}^{\left|\partial\Omega\right|} f_{\rm app}^N
			\left(B^\#\right)^{-1}A^\#\varphi
			- \frac{\gamma}{\left|\partial\Omega\right|}\int_{\partial\Omega}
			B^{-1}B1 A^\#\varphi
			+ \frac{\gamma}{\left|\partial\Omega\right|}\int_{\partial\Omega} A^*\varphi
			\\
			= &
			\int_{0}^{\left|\partial\Omega\right|} f_{\rm app}^N
			\left(B^\#\right)^{-1}A^\#\varphi
			+ \frac{\gamma}{\left|\partial\Omega\right|}\int_{\partial\Omega}
			\underbrace{\left[1-B^{-1}B1\right]}_{=|\partial\Omega|H\cdot\tau \text{ by \eqref{harmonic4}}} A^*\varphi
			\\
			& + \frac{\gamma}{\left|\partial\Omega\right|^2}
			\underbrace{\left[\int_{\partial\Omega}
			B^{-1}B1\right]}_{=0}\left[\int_{\partial\Omega}A^*\varphi\right]
			\\
			= &
			\int_{0}^{\left|\partial\Omega\right|} f_{\rm app}^N
			\left(B^\#\right)^{-1}A^\#\varphi
			+ \gamma \int_{\partial\Omega}\underbrace{A[H\cdot \tau]}_{=\pi H\cdot\tau \text{ by \eqref{harmonic4}}} \varphi
			\\
			= &
			\int_{0}^{\left|\partial\Omega\right|} \left(f_{\rm app}^N-f\right)
			\left(B^\#\right)^{-1}A^\#\varphi
			+\int_{0}^{\left|\partial\Omega\right|} AB^{-1}f
			\varphi
			+ \pi\gamma\int_{\partial\Omega}
			H\cdot\tau \varphi,
		\end{align*}
	where, as already emphasized, the values of $H$ on $\partial\Omega$ are given by its limiting values from inside $\Omega$. Therefore, according to \eqref{fappestimate}, we obtain that
		\begin{align*}
			\left|\int_{0}^{\left|\partial\Omega\right|} (g_{\rm app}^N - AB^{-1}f-\pi\gamma H\cdot\tau )\varphi\right|
			\hspace{-10mm} &
			\\
			& \leq
			\frac{C}{N^{\kappa}}
			\left(\left\|f\right\|_{C^{\kappa}}+|\gamma|
			\right)
			\left\|\left(B^\#\right)^{-1}A^\#\varphi\right\|_{C^{\kappa+1}}.
		\end{align*}
	
	There only remains to estimate the regularity of $\psi=\left(B^\#\right)^{-1}A^\#\varphi\in L^2_0\left(\partial\Omega\right)$. To this end, we rewrite
	\begin{equation*}
		B^*\psi = A^*\varphi + \frac{1}{\left|\partial\Omega\right|}\int_{\partial\Omega}\left(B^*\psi-A^*\varphi\right),
	\end{equation*}
	whence, utilizing \eqref{B1} and \eqref{PBadjoint}, we infer that
	\begin{equation*}
		\pi^2\psi = A^{*2}\psi + A^*B^*\varphi.
	\end{equation*}
	Further recalling that $A^*$ is a regularizing operator, for it has a smooth kernel, we deduce that
	\begin{equation*}
		\left\|\left(B^\#\right)^{-1}A^\#\varphi\right\|_{C^{\kappa+1}}
		=
		\left\|\psi\right\|_{C^{\kappa+1}}
		\leq C\left\|\varphi\right\|_{L^2},
	\end{equation*}
	which concludes the proof of the proposition.
\end{proof}

\begin{remark}
	Note that the use of \eqref{average adjoint}, which is a consequence of Lemma \ref{technical cot}, to estimate $D_1$ in the above proof is the reason why it is not possible to relax condition \eqref{mesh2} for a well distributed mesh if one aims at a convergence rate $\mathcal{O}\left(N^{-\kappa}\right)$ in Theorem \ref{main theo}.
\end{remark}

\section{Proof of Theorem \ref{main theo}}\label{proof of main}

We proceed now to the demonstration of our main result---Theorem \ref{main theo}---on the approximation of the boundary of an exterior domain by point vortices in system \eqref{elliptic2}.

First, for given $\omega\in C_c^{0,\alpha}$ and $\gamma\in\mathbb{R}$, recall that the full plane flow $u_{P}\in C^1\left(\overline \Omega\right)$ is obtained from \eqref{uP} and that the $\left|\partial\Omega\right|$-periodic function $f\in C^\infty\left([0,\left|\partial\Omega\right|]\right)$, which has zero mean value, is defined by $f(s)=2\pi[u_P\cdot n]\left(l(s)\right)$, for all $s\in\left[0,\left|\partial\Omega\right|\right]$. Therefore, with this given $f$, according to Proposition \ref{inverse perfect}, we find a unique solution $\gamma^N\in\mathbb{R}^N$ of \eqref{point toy}.

Next, the approximate flow $u_{\rm app}^N$ is introduced by \eqref{approx} and verifies
	\begin{align*}
		u_{\rm app}^N(x)\cdot n(x) & = - \frac1{2\pi} f_{\rm app}^N(s),
		\\
		u_{\rm app}^N(x)\cdot \tau(x) & = \frac1{2\pi} g_{\rm app}^N(s),
	\end{align*}
where $x=l(s)\in\partial\Omega$ and $f_{\rm app}^N$, $g_{\rm app}^N$ are defined by \eqref{f app}.
Utilizing identity \eqref{vortex identity} to rewrite the discrete Biot--Savart kernel of $u_{\rm app}^N$, we find that
	\begin{align*}
		u_{\rm app}^N(x)
		= & \frac1{2\pi^2} \sum_{j=1}^N \frac{\gamma_{j}^N}N
		\\
		& \times \left(
		\int_{\partial\Omega}\frac{x_{j}^N-z}{\left|x_{j}^N-z\right|^2} \cdot \tau(z) \frac{x-z}{|x-z|^2}dz
		-
		\int_{\partial\Omega}\frac{x_{j}^N-z}{\left|x_{j}^N-z\right|^2}\cdot n(z)\frac{(x-z)^\perp}{|x-z|^2}
		dz
		\right)
		\\
		= &
		-\frac1{2\pi^2}
		\int_0^{\left|\partial\Omega\right|}f_{\rm app}^N(s_*) \frac{x-l(s_*)}{\left|x-l(s_*)\right|^2}ds_*
		\\
		& +\frac1{2\pi^2}
		\int_0^{\left|\partial\Omega\right|}g_{\rm app}^N(s_*)\frac{(x-l(s_*))^\perp}{|x-l(s_*)|^2}
		ds_*, \quad\text{on }\Omega.
	\end{align*}

Furthermore, recall that, according to \eqref{harmonic3} and \eqref{vortex identity 2}, the remainder flow $u_R$ can be expressed as
	\begin{align*}
		u_R(x)
		= & \frac 1{2\pi^2}\int_{\partial\Omega} \frac{(x-y)^\perp}{|x-y|^2}
		\left[AB^{-1}f + \gamma\pi H\cdot\tau\right](y)dy
		\\
		& -
		\frac 1{2\pi^2}\int_{\partial\Omega} \frac{x-y}{|x-y|^2}
		f(y)dy,\quad\text{on }\Omega,
	\end{align*}
whereby
	\begin{align*}
		\left(u_R-u_{\rm app}^N\right)(x)
		= & \frac 1{2\pi^2}\int_{\partial \Omega} \frac{x-y}{|x-y|^2} \left(f^N_{\rm app}-f\right)(y)dy
		\\
		& +
		\frac 1{2\pi^2}\int_{\partial \Omega} \frac{\left(x-y\right)^\perp}{|x-y|^2} \left(AB^{-1}f+ \gamma\pi H\cdot\tau-g^N_{\rm app}\right)(y)dy
		, \quad\text{on }\Omega.
	\end{align*}

Therefore, in view of Proposition \ref{prop 32}, we deduce that, for any fixed $x\in\Omega$,
	\begin{align*}
		\left|\left(u_R-u_{\rm app}^N\right)(x)\right|
		& \leq \frac{C}{N^\kappa}
		\left(\left\|f\right\|_{C^{\kappa}}+|\gamma|
		\right)
		\left\|\frac{x-y}{|x-y|^2}\right\|_{C_y^{\kappa+1}}
		\\
		& \leq \frac{C}{N^\kappa}
		\left(\left\|f\right\|_{C^{\kappa}}+|\gamma|
		\right)
		\sup_{y\in\partial \Omega}\left(\frac{1}{|x-y|}+\frac{1}{|x-y|^{\kappa+2}}\right),
	\end{align*}
where the constant $C>0$ is independent of $x$, $\omega$ and $\gamma$. Since the support of $\omega$ is bounded away from $\partial\Omega$, it holds that
	\begin{align*}
		\left\|f\right\|_{C^{\kappa}}
		&
		\leq C \sup_{x\in\partial\Omega}
		\int_{\mathbb{R}^2}
		\left(\frac{1}{|x-y|}+\frac{1}{|x-y|^{\kappa+1}}\right)\left|\omega(y)\right|dy
		\\
		&
		\leq \frac C{\operatorname{dist}\left(\operatorname{supp}\omega,\partial\Omega\right)}\left(1 +\frac 1{\operatorname{dist}\left(\operatorname{supp}\omega,\partial\Omega\right)^\kappa}\right)
		\left\|\omega\right\|_{L^1\left(\mathbb{R}^2\right)}.
	\end{align*}
It follows that, for any closed set $K\subset\Omega$,
	\begin{align*}
		\| u_{R} - u_{\rm app}^N \|_{L^\infty(K)} \hspace{-5mm}&
		\\
		\leq & \frac{C}{N^\kappa}
		\left(
		\frac 1{\operatorname{dist}\left(K,\partial\Omega\right)} + \frac 1{\operatorname{dist}\left(K,\partial\Omega\right)^{\kappa+2}}
		\right)
		\\
		& \times \left(\left(
		\frac 1{\operatorname{dist}\left(\operatorname{supp}\omega,\partial\Omega\right)} + \frac 1{\operatorname{dist}\left(\operatorname{supp}\omega,\partial\Omega\right)^{\kappa+1}}
		\right)\left\|\omega\right\|_{L^1\left(\mathbb{R}^2\right)} + |\gamma|\right),
	\end{align*}
which, extending the above estimate to all $\omega\in L^1_c\left(\Omega\right)$ by a standard density argument, concludes the proof of the theorem. \qed

\section{Proofs of dynamic theorems}\label{dynamic proofs}

We give now complete justifications of Theorems \ref{wellposedness} and \ref{main convergence}. Both results heavily rely on the static approximation result Theorem \ref{main theo}.

\subsection{Wellposedness of vortex approximation}\label{dynamic proofs 1}

We have already mentioned, earlier in Section \ref{dynamic thm}, that a classical estimate (see e.g.\ \cite[Lemma 4.2]{lacave}) shows that the support in $x$ of the classical solution $\omega$ to \eqref{dynamic1} remains uniformly bounded away from the boundary $\partial\Omega$, i.e.\ $\omega\in C^1_c\left([0,t_1]\times \Omega\right)$. We will need to adapt this estimate to the vortex approximation \eqref{dynamic5} and therefore, for later use, we begin by reproducing here this control on the support of $\omega$.

\begin{lemma}\label{classical supp estimate}
	Let $\omega\in C^1_c\left([0,t_1]\times \overline\Omega\right)$ be the unique classical solution to \eqref{dynamic1}. Then, it holds that, for all $t\in [0,t_1]$,
	\begin{equation}\label{support away boundary}
		\operatorname{dist}\left(\operatorname{supp}\omega,\partial\Omega\right)
		\geq C'e^{-C\int_0^t\left\|u\right\|_{C^1(\Omega)}}\operatorname{dist}\left(\operatorname{supp}\omega_0,\partial\Omega\right),
	\end{equation}
	for some independent constants $C,C'>0$. In particular, it follows that $\omega\in C^1_c\left([0,t_1]\times \Omega\right)$.
\end{lemma}

\begin{proof}
	Recall that a \emph{characteristic curve} (or simply a \emph{characteristic}) for \eqref{dynamic1} is a function
	\begin{equation*}
		X(s,x)\in C^1\left([0,t_2]\times\Omega; \Omega\right),
	\end{equation*}
	solving the differential equation
	\begin{equation}\label{char X}
		\frac{dX}{ds}=u\left(s,X\right),
	\end{equation}
	for some given initial data $X(0,x)=x$ and some existence time $0<t_2\leq t_1$ ($t_2$ may \emph{a priori} depend on $x$). Since the velocity field $u$ is of class $C^1$, standard results from the theory of ordinary differential equations (see \cite{coddington}, for instance) guarantee the existence, uniqueness and regularity of such characteristics for any given initial data.

	We show below that characteristics actually exist up to time $t_1$, i.e.\ that it is always possible to take $t_2=t_1$. This is a consequence of the fact that the flow $u$ is tangent to the boundary $\partial \Omega$, and it does not otherwise hold for general flows. Indeed, loosely speaking, it is easy to force $t_2<t_1$ by creating a constant flow that takes particles and crashes them into the boundary arbitrarily fast (take a constant vector field of large magnitude pointing from the initial data $x$ straight into $\partial\Omega$).
	
	It is also easy to see that a tangency condition on $u$ should be enough to guarantee that characteristics never hit the boundary. This follows from the uniqueness of solutions to the above differential equation. Indeed, it suffices to note that a tangent flow produces trajectories that remain on $\partial\Omega$ provided they start on the boundary. Therefore, a characteristic can only hit the boundary if it is on the boundary all along.
	
	The estimate we now provide is more precise and yields an accurate control on the distance from the characteristic to the boundary $\operatorname{dist}\left(X(s,x),\partial\Omega\right)$ by exploiting the tangency of $u$.

	Employing the conformal map $T:\Omega\to\left\{|x|>1\right\}$ used in \eqref{BS exterior} and defining the curves
	\begin{equation*}
		Y\left(s,T(x)\right):=T\left(X(s,x)\right)\in C^1\left([0,t_2]\times\Omega; \left\{|x|>1\right\}\right),
	\end{equation*}
	we obtain solutions of the differential equation
	\begin{equation*}
		\frac{dY}{ds}=DT\left(T^{-1}Y\right) u\left(s,T^{-1}Y\right),
	\end{equation*}
	for the initial data $Y(0,T(x))=T(x)$. Now, recalling that $DT^t(x)T(x)$ is normal to $\partial\Omega$ at $x\in\partial\Omega$ (see \eqref{harmonic field representation}; this property is easily obtained by differentiating the relation $\left|T(x)\right|=1$ on $\partial\Omega$) and that $u(s,x)$ is tangent to $\partial\Omega$ at the same location, we compute that
	\begin{equation}\label{dist boundary 2}
		\begin{aligned}
			\frac{d\left(|Y|-1\right)}{ds} = & DT\left(T^{-1}Y\right) u\left(s,T^{-1}Y\right)\cdot\frac{Y}{|Y|}
			= u^t\left(s,T^{-1}Y\right) DT^t\left(T^{-1}Y\right)\frac{Y}{|Y|}
			\\
			= & \left(u^t\left(s,T^{-1}Y\right)DT^t\left(T^{-1}Y\right)
			-u^t\left(s,T^{-1}\frac{Y}{|Y|}\right)DT^t\left(T^{-1}\frac{Y}{|Y|}\right)\right)\frac{Y}{|Y|}
			\\
			\geq & -C_T\left\|u\right\|_{C^1(\Omega)}\left|Y-\frac Y{|Y|}\right| = -C_T\left\|u\right\|_{C^1(\Omega)}\left(|Y|-1\right),
		\end{aligned}
	\end{equation}
	where the constant $C_T>0$ only depends on $T$. It follows that
	\begin{equation*}
		\left|Y(s,T(x))\right|-1\geq e^{-C_T\int_0^s\left\|u\right\|_{C^1(\Omega)}}\left(\left|T(x)\right|-1\right),
	\end{equation*}
	whereby, for any $x_0\in\partial\Omega$,
		\begin{align*}
			\left|X(s,x)-x_0\right|\hspace{-10mm}&
			\\
			\geq & C_T'\left|T(X(s,x))-T(x_0)\right|
			\\
			\geq & C_T'e^{-C_T\int_0^s\left\|u\right\|_{C^1(\Omega)}}\left(\left|T(x)\right|-1\right)
			= C_T'e^{-C_T\int_0^s\left\|u\right\|_{C^1(\Omega)}}\left|T(x)-\frac{T(x)}{|T(x)|}\right|
			\\
			\geq & C_T''e^{-C_T\int_0^s\left\|u\right\|_{C^1(\Omega)}}\left|x-T^{-1}\left(\frac{T(x)}{|T(x)|}\right)\right|
			\geq
			C_T''e^{-C_T\int_0^s\left\|u\right\|_{C^1(\Omega)}}\operatorname{dist}\left(x,\partial\Omega\right),
		\end{align*}
	with some constants $C_T',C_T''>0$ only depending on $T$. Further minimizing the above left-hand side over $x_0\in\partial\Omega$, we finally conclude that, for any $x\in\Omega$,
	\begin{equation}\label{dist boundary}
		\operatorname{dist}\left(X(s,x),\partial\Omega\right) \geq C_T''e^{-C_T\int_0^s\left\|u\right\|_{C^1(\Omega)}}\operatorname{dist}\left(x,\partial\Omega\right),
	\end{equation}
	which implies, in particular, that $X(t_2,x)\in\Omega$, for all $x\in\Omega$. Therefore, by continuing the characteristic beyond the existence time $t_2$ (which is always possible because $u$ is bounded on $[0,t_1]\times\Omega$; see \cite[Chapter~1, Theorem~4.1]{coddington}, for instance), we may always assume that $t_2=t_1$.

	Now, for any fixed $0\leq t\leq t_1$, it holds that the mapping $X(t,\cdot):\Omega\to\Omega$ is a $C^1$-diffeomorphism preserving the Lebesgue measure, for $u$ is solenoidal (see \cite[Chapter~1, Theorem~7.2]{coddington}). We denote its inverse by $X^{-1}(t,\cdot):\Omega\to\Omega$. In particular, recasting the transport equation \eqref{dynamic1} in Lagrangian coordinates
	\begin{equation*}
		\left\{
		\begin{aligned}
			& \frac{d}{dt}\omega\left(t,X(t,x)\right)=0,
			\\
			& \omega(t=0)=\omega_0,
		\end{aligned}
		\right.
	\end{equation*}
	we deduce that
	\begin{equation*}
		\omega\left(t,X(t,x)\right) = \omega_0(x)
		\quad\text{and}\quad
		\omega(t,x) = \omega_0\left(X^{-1}(t,x)\right),
	\end{equation*}
	for all $0\leq t\leq t_1$ and $x\in\Omega$. We therefore conclude from \eqref{dist boundary} that
		\begin{align*}
			\operatorname{dist}\left(\operatorname{supp}\omega,\partial\Omega\right)
			& =\inf_{\substack{x\in\Omega \\ \omega(t,x)\neq 0}}\operatorname{dist}\left(x,\partial\Omega\right)
			\\
			& \geq C_T''e^{-C_T\int_0^t\left\|u\right\|_{C^1(\Omega)}}\inf_{\substack{x\in\Omega \\ \omega_0\left(X^{-1}(t,x)\right)\neq 0}}\operatorname{dist}\left(X^{-1}(t,x),\partial\Omega\right)
			\\
			& = C_T''e^{-C_T\int_0^t\left\|u\right\|_{C^1(\Omega)}}\operatorname{dist}\left(\operatorname{supp}\omega_0,\partial\Omega\right),
		\end{align*}
	which, as announced, establishes that $\omega\in C^1_c\left([0,t_1]\times \Omega\right)$ and completes the proof of the lemma.
\end{proof}

We move on now to the justification of the wellposedness of the vortex approximation \eqref{dynamic5} asserted in Theorem \ref{wellposedness}. To this end, we begin with a few classical lemmas providing precise estimates on velocity flows. The next result is standard (see \cite[Lemmas 4.5 and 4.6]{MajdaBertozzi}, for instance).

\begin{lemma}
	For any vortex density $\omega\in C^1_c\left(\Omega\right)$, one has the estimates
	\begin{equation}\label{mu 1}
		\left\|K_{\mathbb{R}^2}\left[\omega\right]\right\|_{L^\infty(\mathbb{R}^2)}
		\leq
		C\left\|\omega\right\|_{L^1\cap L^\infty(\mathbb{R}^2)},
	\end{equation}
	and
	\begin{equation}\label{mu 3}
			\left\|\nabla K_{\mathbb{R}^2}\left[\omega\right]\right\|_{L^\infty(\mathbb{R}^2)}
			\leq C\left(
			1+
			\left\|\omega\right\|_{L^1\cap L^\infty(\mathbb{R}^2)}
			+
			\left\|\omega\right\|_{L^\infty(\mathbb{R}^2)}
			\log\left(1+\left\|\nabla\omega\right\|_{L^\infty \left(\mathbb{R}^2\right)}\right)
			\right),
	\end{equation}
	for some independent constant $C>0$.
\end{lemma}

\begin{proof}
	The first estimate \eqref{mu 1} is obtained straightforwardly by isolating the integrable singularity of the Biot--Savart kernel by a ball of fixed radius centered at the singularity and, then, by estimating the contributions of the integrand of $K_{\mathbb{R}^2}\left[\omega\right]$ within this ball and on its exterior separately.
	
	The second estimate \eqref{mu 3} is more delicate. To justify it, we first compute that, for any radii $0<R_0<R\leq 1$,
		\begin{align*}
			\nabla K_{\mathbb{R}^2}\left[\omega\right] \hspace{-11mm} &
			\\
			= &
			\frac 1{2\pi}\int_{\mathbb{R}^2} \frac{y^\perp}{|y|^2} \otimes \nabla\omega(x-y) dy
			\\
			= &
			\frac 1{2\pi}\int_{\left\{|y|\leq R_0\right\}} \frac{y^\perp}{|y|^2} \otimes \nabla\omega(x-y) dy
			+
			\frac 1{2\pi}\int_{\left\{|y|= R_0\right\}}
			\frac{1}{|y|^3}
			\begin{pmatrix}
				-y_1y_2 & -y_2^2
				\\
				y_1^2 & y_1y_2
			\end{pmatrix}
			\omega(x-y) dy
			\\
			&
			+
			\frac 1{2\pi}\int_{\left\{|y|> R_0\right\}}
			\frac{1}{|y|^4}
			\begin{pmatrix}
				2y_1y_2 & y_2^2-y_1^2
				\\
				y_2^2 - y_1^2 & -2y_1y_2
			\end{pmatrix}
			\omega(x-y) dy
			\\
			= &
			\frac 1{2\pi}\int_{\left\{|y|\leq R_0\right\}} \frac{y^\perp}{|y|^2} \otimes \nabla\omega(x-y) dy
			\\
			& +
			\frac 1{2\pi}\int_{\left\{|y|= R_0\right\}}
			\frac{1}{|y|^3}
			\begin{pmatrix}
				-y_1y_2 & -y_2^2
				\\
				y_1^2 & y_1y_2
			\end{pmatrix}
			\left(\omega(x-y)-\omega(x)\right) dy
			+
			\begin{pmatrix}
				0 & -\frac 12
				\\
				\frac 12 & 0
			\end{pmatrix}
			\omega(x)
			\\
			& +
			\frac 1{2\pi}\int_{\left\{R_0<|y|\leq R\right\}}
			\frac{1}{|y|^4}
			\begin{pmatrix}
				2y_1y_2 & y_2^2-y_1^2
				\\
				y_2^2 - y_1^2 & -2y_1y_2
			\end{pmatrix}
			\left(\omega(x-y)-\omega(x)\right) dy
			\\
			& +
			\frac 1{2\pi}\int_{\left\{|y|> R\right\}}
			\frac{1}{|y|^4}
			\begin{pmatrix}
				2y_1y_2 & y_2^2-y_1^2
				\\
				y_2^2 - y_1^2 & -2y_1y_2
			\end{pmatrix}
			\omega(x-y) dy,
		\end{align*}
	and then let $R_0\to 0$ to yield
	\begin{equation}\label{representation velocity}
		\begin{aligned}
			\nabla K_{\mathbb{R}^2}\left[\omega\right]  
			= &
			\begin{pmatrix}
				0 & -\frac 12
				\\
				\frac 12 & 0
			\end{pmatrix}
			\omega(x)\\
			&+
			\frac 1{2\pi}\int_{\left\{|y|\leq R\right\}}
			\frac{1}{|y|^4}
			\begin{pmatrix}
				2y_1y_2 & y_2^2-y_1^2
				\\
				y_2^2 - y_1^2 & -2y_1y_2
			\end{pmatrix}
			\left(\omega(x-y)-\omega(x)\right) dy
			\\
			& +
			\frac 1{2\pi}\int_{\left\{|y|> R\right\}}
			\frac{1}{|y|^4}
			\begin{pmatrix}
				2y_1y_2 & y_2^2-y_1^2
				\\
				y_2^2 - y_1^2 & -2y_1y_2
			\end{pmatrix}
			\omega(x-y) dy.
		\end{aligned}
	\end{equation}
	It follows that
		\begin{align*}
			\left\|\nabla K_{\mathbb{R}^2}\left[\omega\right]\right\|_{L^\infty(\mathbb{R}^2)}
			\leq & C\left(
			\left\|\omega\right\|_{L^1\cap L^\infty(\mathbb{R}^2)}
			+
			R\left\|\nabla\omega\right\|_{L^\infty(\mathbb{R}^2)}
			+
			\log\left(R^{-1}\right)
			\left\|\omega\right\|_{L^\infty(\mathbb{R}^2)}
			\right)
			\\
			\leq & C\left(
			1+
			\left\|\omega\right\|_{L^1\cap L^\infty(\mathbb{R}^2)}
			+
			\left\|\omega\right\|_{L^\infty(\mathbb{R}^2)}
			\log\left(1+\left\|\nabla\omega\right\|_{L^\infty \left(\mathbb{R}^2\right)}\right)
			\right),
		\end{align*}
	where we optimized the last estimate by setting $R=\frac 1{1+\left\|\nabla\omega\right\|_{L^\infty \left(\mathbb{R}^2\right)}}$, which completes the proof.
\end{proof}

\begin{lemma}
	For any vortex density $\omega\in C^1_c\left(\Omega\right)$ and any circulation $\gamma\in\mathbb{R}$, one has the estimates
	\begin{equation}\label{mu 2}
		\left\|u_{\rm app}^N\left[\omega,\gamma\right]\right\|_{L^\infty(K)}
		\leq
		\frac {C}{\operatorname{dist}\left(K,\partial\Omega\right)}
		\left(
		\left\|\omega\right\|_{L^1\cap L^\infty\left(\mathbb{R}^2\right)}
		+|\gamma|
		\right),
	\end{equation}
	and
	\begin{equation}\label{mu 4}
		\left\|\nabla u_{\rm app}^N\left[\omega,\gamma\right]\right\|_{L^\infty(K)}
		\leq
		\frac {C}{\operatorname{dist}\left(K,\partial\Omega\right)^2}
		\left(
		\left\|\omega\right\|_{L^1\cap L^\infty\left(\mathbb{R}^2\right)}
		+|\gamma|
		\right),
	\end{equation}
	for some independent constant $C>0$.
\end{lemma}

\begin{proof}
	Using \eqref{est l1} from Proposition \ref{inverse perfect} and the convergence of Riemann sums for $C^{0,\alpha}$ functions (see, e.g., \cite[Lemma 4.1]{ADLproc}), one has that, for any compact set $K\subset\Omega$,
		\begin{align*}
			\left\|u_{\rm app}^N\left[\omega,\gamma\right]\right\|_{L^\infty(K)}
			\hspace{-25mm}&
			\\
			\leq &
			\frac C{\operatorname{dist}\left(K,\partial\Omega\right)}\left\|\gamma^N\right\|_{\ell^1}
			\\
			\leq &
			\frac C{\operatorname{dist}\left(K,\partial\Omega\right)}
			\left(\left\|K_{\mathbb{R}^2}\left[\omega\right]\cdot n\right\|_{L^\infty\left(\partial\Omega\right)}
			+ |\gamma|
			+\sqrt N\left|\frac 1{N-1}\sum_{i=1}^{N-1}K_{\mathbb{R}^2}\left[\omega\right]\cdot n\left(\tilde x_i^N\right)\right|\right)
			\\
			\leq &
			\frac C{\operatorname{dist}\left(K,\partial\Omega\right)}
			\Bigg(\left\|K_{\mathbb{R}^2}\left[\omega\right]\cdot n\right\|_{L^\infty\left(\partial\Omega\right)}
			+|\gamma|
			\\
			& +\sqrt N\left|\frac {\left|\partial\Omega\right|}{N}\sum_{i=1}^{N}K_{\mathbb{R}^2}\left[\omega\right]\cdot n\left(\tilde x_i^N\right)
			-\int_{\partial\Omega}K_{\mathbb{R}^2}\left[\omega\right]\cdot n(x)dx\right|\Bigg)
			\\
			\leq &
			\frac C{\operatorname{dist}\left(K,\partial\Omega\right)}
			\left(
			\left\|K_{\mathbb{R}^2}\left[\omega\right]\cdot n\right\|_{C^{0,\frac 12}\left(\partial\Omega\right)}
			+|\gamma|
			\right)
			\\
			\leq &
			\frac C{\operatorname{dist}\left(K,\partial\Omega\right)}
			\left(
			\left\|\int_{\mathbb{R}^2}\left(\frac{1}{\left|x-y\right|}+\frac{1}{\left|x-y\right|^\frac 32}\right)\left|\omega(y)\right|dy\right\|_{L^\infty\left(\partial\Omega\right)}
			+|\gamma|\right)
			\\
			\leq &
			\frac {C}{\operatorname{dist}\left(K,\partial\Omega\right)}
			\left(
			\left\|\omega\right\|_{L^1\cap L^\infty\left(\mathbb{R}^2\right)}
			+|\gamma|
			\right),
		\end{align*}
	and, similarly,
		\begin{align*}
			\left\|\nabla u_{\rm app}^N\left[\omega,\gamma\right]\right\|_{L^\infty(K)}
			\leq &
			\frac C{\operatorname{dist}\left(K,\partial\Omega\right)^2}\left\|\gamma^N\right\|_{\ell^1}
			\\
			\leq &
			\frac {C}{\operatorname{dist}\left(K,\partial\Omega\right)^2}
			\left(
			\left\|\omega\right\|_{L^1\cap L^\infty\left(\mathbb{R}^2\right)}
			+|\gamma|
			\right),
		\end{align*}
	which completes the proof of the lemma.
\end{proof}

\begin{lemma}
	For any vortex density $\omega\in C^1_c\left(\Omega\right)$ and any circulation $\gamma\in\mathbb{R}$, considering the velocity flow $u$ defined by \eqref{BS exterior}, one has the estimates
	\begin{equation}\label{C1 estimate}
		\begin{aligned}
			\left\|u\right\|_{L^\infty\left(\Omega\right)}
			& \leq
			C\left(\left\|\omega\right\|_{L^1\cap L^\infty\left(\Omega\right)}
			+|\gamma|\right),
			\\
			\left\|\nabla u\right\|_{L^\infty\left(\Omega\right)}
			& \leq
			C\left(\left\|\omega\right\|_{L^1\cap L^\infty\left(\Omega\right)}
			+ \left\|\nabla\omega\right\|_{L^1\cap L^\infty\left(\Omega\right)}
			+|\gamma|\right),
		\end{aligned}
	\end{equation}
	and, for all $x,h\in\mathbb{R}^2$ such that $|h|\leq 1$ and $[x,x+h]\subset\Omega$,
	\begin{equation}\label{quasilipschitz}
		\left|u(x+h)-u(x)\right|
		\leq C
		\left(|\gamma|+\left(1+|x|^2\right)\left\|\omega\right\|_{L^1\cap L^\infty\left(\Omega\right)}\right)|h|\left(1-\log|h|\right),
	\end{equation}
	for some independent constant $C>0$.
\end{lemma}

\begin{proof}
	The first estimates \eqref{C1 estimate} are straightforward and deduced directly from \eqref{BS exterior}, employing \eqref{convolution T}.
	
	Let us therefore focus on the more refined control \eqref{quasilipschitz}. Using that
	\begin{equation*}
		\frac{|b-b^*|}{|a-b^*|}\leq \frac{|b-b^*|}{\left|\frac{b}{|b|}-b^*\right|}= 1+|b|,
	\end{equation*}
	for any $a,b\in\mathbb{R}^2$ such that $|a|,|b|> 1$, we first compute that
		\begin{align*}
			\left|\nabla_{x} G_{\Omega}(x,y)\right|
			\leq &
			\frac{\left|DT(x)\right|\left|T(y)-T(y)^*\right|}{\left|T(x)-T(y)\right|\left|T(x)-T(y)^*\right|}
			\leq
			\frac{\left|DT(x)\right|\left(1+\left|T(y)\right|\right)}{\left|T(x)-T(y)\right|}
			\\
			\leq & C_T
			\frac{1+|y|}{|x-y|}
			\leq C_T
			\left(1+\frac{1+|x|}{|x-y|}\right),
			\\
			\left|D_{x}^2 G_{\Omega}(x,y)\right|
			\leq &
			\frac{\left|D^2T(x)\right|\left|T(y)-T(y)^*\right|}{\left|T(x)-T(y)\right|\left|T(x)-T(y)^*\right|}
			\\ & +7
			\frac{\left|DT(x)\right|^2\left|T(y)-T(y)^*\right|\left(\left|T(x)-T(y)\right|+\left|T(x)-T(y)^*\right|\right)}{\left|T(x)-T(y)\right|^2\left|T(x)-T(y)^*\right|^2}
			\\
			\leq &
			\frac{\left|D^2T(x)\right|\left(1+\left|T(y)\right|\right)}{\left|T(x)-T(y)\right|}
			+ 14\frac{\left|DT(x)\right|^2\left(1+\left|T(y)\right|\right)^2}{\left|T(x)-T(y)\right|^2}
			\\
			\leq & C_T
			\left(\frac{1+|y|}{|x-y|}+\frac{1+|y|^2}{|x-y|^2}\right)
			\leq C_T'
			\left(1+\frac{1+|x|^2}{|x-y|^2}\right),
		\end{align*}
	for some constants $C_T,C_T'>0$ only depending on $T$.
	
	Then, utilizing these estimates on the Green function, we deduce the quasi-Lipschitz estimate, for all $x,h\in\mathbb{R}^2$ such that $|h|\leq 1$ and $[x,x+h]\subset\Omega$:
		\begin{align*}
			\left|u(x+h)-u(x)\right|
			\hspace{-22mm} &
			\\
			\lesssim &
			|\alpha||h|
			+
			\int_{\Omega} \left|\nabla_{x} G_{\Omega}(x+h,y)
			-\nabla_{x} G_{\Omega}(x,y)\right| \left|\omega(y)\right|dy
			\\
			\lesssim &
			\left(|\gamma|+\left\|\omega\right\|_{L^1\left(\Omega\right)}\right)|h|
			+
			\int_{\left\{|x-y|\leq 2|h|\right\}} \left(\left|\nabla_{x} G_{\Omega}(x+h,y)\right|
			+\left|\nabla_{x} G_{\Omega}(x,y)\right|\right) \left|\omega(y)\right|dy
			\\
			& +
			\int_{\left\{|x-y|> 2|h|\right\}} \left|\nabla_{x} G_{\Omega}(x+h,y)
			-\nabla_{x} G_{\Omega}(x,y)\right| \left|\omega(y)\right|dy
			\\
			\lesssim &
			\left(|\gamma|+\left\|\omega\right\|_{L^1\cap L^\infty\left(\Omega\right)}\right)|h|
			+(1+|x|)\left\|\omega\right\|_{L^\infty\left(\Omega\right)}
			\int_{\left\{|x-y|\leq 3|h|\right\}}
			\frac 1{|x-y|}
			dy
			\\
			& +|h|\sup_{z\in[x,x+h]}
			\int_{\left\{|x-y|> 2|h|\right\}} \left|D_{x}^2 G_{\Omega}(z,y)\right| \left|\omega(y)\right|dy
			\\
			\lesssim &
			\left(|\gamma|+(1+|x|)\left\|\omega\right\|_{L^1\cap L^\infty\left(\Omega\right)}\right)|h|
			+|h|\left(1+|x|^2\right)
			\int_{\left\{|x-y|> 2|h|\right\}} \frac 1{|x-y|^2} \left|\omega(y)\right|dy
			\\
			\lesssim &
			\left(|\gamma|+\left(1+|x|^2\right)\left\|\omega\right\|_{L^1\cap L^\infty\left(\Omega\right)}\right)|h|\left(1-\log|h|\right),
		\end{align*}
	where $\lesssim$ denotes the possible presence of various multiplicative independent constants, which completes the proof of the lemma.
\end{proof}

For any given initial data $\omega_0\in C_c^1\left(\Omega\right)$ and some possibly small parameter $0<\eps\leq 1$ (how small is decided later on), we introduce now the subspace $C_{\omega_0,\eps}\subset C_c^1\left([0,t_1]\times\Omega\right)$ as follows: a vortex density $\xi\in C_c^1\left([0,t_1]\times\Omega\right)$ belongs to $C_{\omega_0,\eps}$ if and only if
\begin{itemize}
	\item $\left\|\xi\right\|_{L^1\left(\Omega\right)}=\left\|\omega_0\right\|_{L^1\left(\Omega\right)}$ and $\left\|\xi\right\|_{L^\infty\left(\Omega\right)}=\left\|\omega_0\right\|_{L^\infty\left(\Omega\right)}$, for every $t\in [0,t_1]$,
	\item $\xi(0,x)=\omega_0(x)$, for every $x\in\Omega$,
	\item $\operatorname{supp}\xi(t,\cdot)\subset \Omega\cap\left\{\eps<\operatorname{dist}\left(x,\partial\Omega\right)<\eps^{-1}\right\}$, for every $t\in [0,t_1]$.
\end{itemize}
Observe finally that the subspace $C_{\omega_0,\eps}$ inherits its topology from the metric topology of $C_c^1\left([0,t_1]\times\Omega\right)$.

The following proposition is a suitable wellposedness result for a linearized version of \eqref{dynamic5} in the metric subspace $C_{\omega_0,\eps}$. 

\begin{proposition}\label{linear wellposedness}
	Let $\omega_0\in C_c^1\left(\Omega\right)$, $\gamma\in\mathbb{R}$, $0<\eps\leq 1$ and consider any fixed time $t_1>0$. Then, for a well distributed mesh on $\partial \Omega$, provided $\eps$ is sufficiently small, there exists $N_1\geq N_0$ ($N_0$ is determined in Theorem \ref{main theo}) such that, for any $N\geq N_1$ and for any vortex density $\omega\in C_{\omega_0,\eps}$, there is a unique classical solution $\xi\in C_{\omega_0,\eps}$ to the linear equation
	\begin{equation}\label{linear transport}
		\left\{
		\begin{aligned}
			& \partial_{t} \xi + u\cdot \nabla \xi =0,
			\\
			& \xi(t=0) = \omega_0,
		\end{aligned}
		\right.
	\end{equation}
	with a velocity flow
	\begin{equation*}
		u=K_{\mathbb{R}^2}[\omega]+u_{\rm app}^N[\omega,\gamma].
	\end{equation*}
\end{proposition}

\begin{proof}
	Let
	\begin{equation*}
		Z(s,x)\in C^1\left([0,t_2]\times\Omega; \Omega\right),
	\end{equation*}
	be the characteristic curve corresponding to $u$, i.e.\ the curve solving the differential equation
	\begin{equation*}
		\frac{dZ}{ds}=u\left(s,Z\right),
	\end{equation*}
	for some given initial data $Z(0,x)=x\in\Omega$ and some existence time $0<t_2\leq t_1$. As before (see proof of Lemma \ref{classical supp estimate}), since the velocity field $u$ is of class $C^1$ on $\Omega$ (but not on $\overline\Omega$, though), standard results from the theory of ordinary differential equations guarantee the existence, uniqueness and regularity of such curves for any given initial data. It seems {\it a priori} that the time $t_2$ may depend on $x$. However, since the support of $\omega_0$ is bounded and does not intersect $\partial \Omega$, and $u$ is uniformly bounded on compact subsets $K\subset\Omega$, by \eqref{mu 1} and \eqref{mu 2}, independently of $\omega\in C_{\omega_0,\eps}$, it is possible to set the time $t_2$ so small that $Z(t,x)$, for all $x\in\operatorname{supp}\omega_0$, is defined on $[0,t_2]$ and remains uniformly bounded away from $\partial\Omega$.

	In fact, we show now that $Z(t,x)$, for each $x\in\operatorname{supp}\omega_0$, remains uniformly bounded away from $\partial\Omega$ no matter how large $0<t_2\leq t_1$ is. To this end, consider the modified lagrangian flow $Y$ defined by
	\begin{equation*}
		Y\left(s,T(x)\right):=T\left(Z(s,x)\right)\in C^1\left([0,t_2]\times\operatorname{supp}\omega_0; \left\{|x|>1\right\}\right),
	\end{equation*}
	solving the differential equation
	\begin{equation*}
		\frac{dY}{ds}=DT\left(T^{-1}Y\right) u\left(s,T^{-1}Y\right),
	\end{equation*}
	for the initial data $Y(0,T(x))=T(x)$, with $x\in \operatorname{supp}\omega_0$, and the velocity flow
	\begin{equation*}
		v = K_{\Omega}[\omega] + \left(\gamma + \int_{\Omega}\omega(y) dy\right) H(x).
	\end{equation*}
	Then, in view of Theorem \ref{main theo} with $\kappa=2$ and employing \eqref{C1 estimate}-\eqref{quasilipschitz}, a variation of estimate \eqref{dist boundary 2} yields that
		\begin{align*}
			\left|\frac{d\left(\left|Y\right|-1\right)}{ds}\right|
			\hspace{-10mm}&
			\\
			= &
			\left|
			v^t\left(s,T^{-1}Y\right) DT^t\left(T^{-1}Y\right)\frac{Y}{\left|Y\right|}
			+
			\left(u-v\right)^t\left(s,T^{-1}Y\right) DT^t\left(T^{-1}Y\right)\frac{Y}{|Y|}
			\right|
			\\
			= & \bigg| \left(v^t\left(s,T^{-1}Y\right)DT^t\left(T^{-1}Y\right)
			-v^t\left(s,T^{-1}\frac{Y}{|Y|}\right)DT^t\left(T^{-1}\frac{Y}{|Y|}\right)\right)\frac{Y}{|Y|}
			\\
			& +
			\left(u-v\right)^t\left(s,T^{-1}Y\right) DT^t\left(T^{-1}Y\right)\frac{Y}{|Y|}\bigg|
			\\
			\leq & C\left\|v\right\|_{L^\infty(\Omega)}\left|Y-\frac {Y}{|Y|}\right|
			+
			C\left|v\left(s,T^{-1}Y\right)
			-v\left(s,T^{-1}\frac{Y}{|Y|}\right)\right|
			\\
			& + C \left|u-v\right|\left(s,T^{-1}Y\right)
			\\
			\leq &
			C\left(|Y|-1\right)\left(1+\left|\log\left(|Y|-1\right)\right|\right)
			+\frac C{N^2\eps^3}
			\left(
			\frac 1{|Y|-1} + \frac 1{\left(|Y|-1\right)^4}
			\right),
		\end{align*}
	where $C>0$ denotes various constants possibly depending on $T$, $\omega_0$ and $\gamma$, but independent of $\omega$, $\eps$ and $N$. Further denoting $y=\left(|Y|-1\right)^5$ for convenience and rearranging the preceding estimate yields
		\begin{align*}
			\left|\frac{dy}{ds}\right| & \leq C\left(y\left(1+\left|\log y\right|\right) +\frac 1{N^2\eps^3}\right)
			\\
			& \leq C\left(y+\frac 1{N^2\eps^3}\right)\left(1+\left|\log\left(y+\frac 1{N^2\eps^3}\right)\right|\right),
		\end{align*}
	and, therefore,
	\begin{equation*}
		\left|\frac{d\log\left(1+\left|\log \left(y+\frac 1{N^2\eps^3}\right)\right|\right)}{ds}\right| \leq C.
	\end{equation*}
	It follows that, for all $x\in\operatorname{supp}\omega_0$,
	\begin{equation*}
		y+\frac 1{N^2\eps^3}\geq e^{1-Ce^{Cs}},
	\end{equation*}
	where $C\geq 1$ only depends on $T$, $\omega_0$ and $\gamma$, whence
		\begin{align*}
			C'\left|Z(s,x)-x_0\right|^5
			& \geq
			\left|Y(s,T(x))-T(x_0)\right|^5
			\\
			& \geq
			\left(|Y(s,T(x))|-1\right)^5\geq e^{1-Ce^{Cs}}-\frac 1{N^2\eps^3},
		\end{align*}
	for any $x_0\in\partial\Omega$, where $C'>0$ only depends on $T$. Further taking the infimum of the above left-hand side over all $x_0\in\partial\Omega$ yields that
	\begin{equation*}
		C'\operatorname{dist}(Z(s,x),\partial\Omega)^5
		\geq e^{1-Ce^{Cs}}-\frac 1{N^2\eps^3}
		\geq e^{1-Ce^{Ct_1}}-\frac 1{N^2\eps^3},
	\end{equation*}
	for all $x\in\operatorname{supp}\omega_0$ and $s\in [0,t_2]$.

	Now comes the time to set the value of $\eps$ so small that the above right-hand side is larger than $C' \eps^5$. More precisely, the parameter $\eps$ is first chosen so that $e^{1-Ce^{Ct_1}}\geq 2C'\eps^5$. Once $\eps$ is set, it is readily seen that there exists $N_1\geq N_0$ such that $e^{1-Ce^{Ct_1}}-\frac 1{N^2\eps^3}\geq 2C'\eps^5-\frac 1{N^2\eps^3}>C'\eps^5$, for all $N\geq N_1$. In particular, since this implies that $\operatorname{dist}(Z(s,x),\partial\Omega)>\eps$ as long as $Z(s,x)$ exists within $[0,t_1]$, according to classical results on the continuation of solutions to differential equations (see \cite[Chapter~1, Theorem~4.1]{coddington}, for instance), since $u$ is continuous and uniformly bounded pointwise on compact sets independently of $\omega\in C_{\omega_0,\eps}$, we conclude that the solution $Z(s,x)$ exist over $[0,t_1]$. If necessary, it is possible to further reduce the value of $\eps$ so that
	\begin{equation}\label{distance epsilon}
			\eps<\operatorname{dist}(Z(s,x),\partial\Omega)<\eps^{-1},
			\quad\text{for all }x\in\operatorname{supp}\omega_0 \text{ and }s\in [0,t_1].
	\end{equation}

	For any fixed $t\in [0,t_1]$, it holds that the mapping $x\mapsto Z(t,x)$ is a $C^1$-diffeomorphism, preserving the Lebesgue measure, from $\operatorname{supp}\omega_0$ onto its own image. We denote its inverse by $Z^{-1}(t,\cdot)$ and we define the new vortex density $\xi\in C_c^1\left([0,t_1]\times\Omega\right)$ by
	\begin{equation}\label{lagrangian transport}
		\left\{
		\begin{aligned}
			\xi(t,x) & = \omega_0\left(Z^{-1}(t,x)\right), && \text{if } x\in Z\left(t,\operatorname{supp}\omega_0\right),
			\\
			\xi(t,x) & = 0, && \text{otherwise}.
		\end{aligned}
		\right.
	\end{equation}
	Then, by virtue of \eqref{distance epsilon}, one easily verifies that $\xi$ belongs to $C_{\omega_0,\eps}$ and solves the transport equation \eqref{linear transport}.
	
	Finally, the fact that any classical solution of \eqref{linear transport} necessarily satisfies \eqref{lagrangian transport} easily yields the uniqueness of $\xi$, which concludes the proof of the proposition.
\end{proof}

We are now ready to proceed to the actual proof of Theorem \ref{wellposedness}.

\begin{proof}[Proof of Theorem \ref{wellposedness}]
	This demonstration is somewhat lengthy and, so, we split it into several steps:
	\begin{enumerate}
		\item First, we build an approximating sequence using a standard iteration procedure based on the wellposedness of the linear transport equation established in Proposition \ref{linear wellposedness}.
		\item Second, we establish uniform $C^1$-bounds on this approximating sequence.
		\item Next, we show that it is actually a Cauchy sequence in $C^0$, which allows us to pass to the limit in the iteration scheme and obtain a solution of \eqref{dynamic5} in the sense of distributions.
		\item Finally, we explain why this solution is in fact of class $C^1$ and provide some concluding remarks.
	\end{enumerate}
	
	\bigskip
	
	\noindent\textbf{Construction of approximating sequence.}
	Now, in order to establish the existence of the classical solution to \eqref{dynamic5}, we first build an approximating sequence $\left\{\xi_n\right\}_{n\geq 0}$ of vortex densities within the complete metric subspace $C_{\omega_0,\eps}\subset C_c^1\left([0,t_1]\times\Omega\right)$ defined above.

	The first term $\xi_0\in C_{\omega_0,\eps}$ of the approximating sequence is simply given by $\xi_0(t,x)=\omega_0(x)$, for all $(t,x)\in [0,t_1]\times\Omega$. Then, for each $\xi_n\in C_{\omega_0,\eps}$, the following term $\xi_{n+1}\in C_{\omega_0,\eps}$ is defined, by virtue of Proposition \ref{linear wellposedness}, assuming $\eps>0$ is sufficiently small while $N$ is large enough, as the unique solution to
	\begin{equation*}
		\left\{
		\begin{aligned}
			& \partial_{t} \xi_{n+1} + \mu_n\cdot \nabla \xi_{n+1} =0,
			\\
			& \xi_{n+1}(t=0) = \omega_0,
		\end{aligned}
		\right.
	\end{equation*}
	where the velocity flow $\mu_n$ is given by
	\begin{equation*}
		\mu_n=K_{\mathbb{R}^2}[\xi_n]+u_{\rm app}^N[\xi_n,\gamma].
	\end{equation*}

	\bigskip
	
	\noindent\textbf{Uniform boundedness in $C^1$.}
	We show now that $\left\{\xi_n\right\}_{n\geq 0}$ is uniformly bounded in $C^1\left([0,t_1]\times\Omega\right)$.
	
	To this end, observe that the $\xi_n$'s also solve (in the sense of distributions) the equation, for $i=1,2$:
	\begin{equation*}
		\partial_{t} \partial_{x_i}\xi_{n+1} + \mu_n\cdot \nabla \partial_{x_i}\xi_{n+1} =-\partial_{x_i}\mu_n\cdot \nabla \xi_{n+1}.
	\end{equation*}
	It follows that, for any $[a,b]\subset [0,t_1]$,
	\begin{equation*}
		\partial_{x_i}\xi_{n+1}\left(b,Z_n(b,x)\right)
		=
		\partial_{x_i}\xi_{n+1}\left(a,Z_n(a,x)\right)
		-\int_a^b
		\partial_{x_i}\mu_n\cdot \nabla \xi_{n+1}\left(s,Z_n(s,x)\right)
		ds,
	\end{equation*}
	where
	\begin{equation*}
		Z_n(s,x)\in C^1\left([0,t_2]\times\Omega; \Omega\right),
	\end{equation*}
	is the characteristic curve corresponding to $\mu_n$, i.e.\ the curve solving the differential equation
	\begin{equation*}
		\frac{dZ_n}{ds}=\mu_n\left(s,Z_n\right),
	\end{equation*}
	for some given initial data $Z(0,x)=x\in\Omega$ and some existence time $0<t_2\leq t_1$. Recall that, as shown in the proof of Proposition \ref{linear wellposedness}, it is possible to set $t_2=t_1$, for all $x\in\operatorname{supp}\omega_0$.

	Then, by Gr\"onwall's lemma, we obtain
		\begin{align*}
			\left|\nabla\xi_{n+1}\left(b,Z_n(b,x)\right)\right|
			& \leq
			\left|\nabla\xi_{n+1}\left(a,Z_n(a,x)\right)\right|
			e^{\int_a^b \left|\nabla\mu_n\left(s,Z_n(s,x)\right)\right| ds},
			\\
			\left|\partial_t\xi_{n+1}\left(b,Z_n(b,x)\right)\right|
			& \leq
			\left|\nabla\xi_{n+1}\left(a,Z_n(a,x)\right)\right|
			\left|\mu_n\left(b,Z_n(b,x)\right)\right|
			e^{\int_a^b \left|\nabla\mu_n\left(s,Z_n(s,x)\right)\right| ds},
		\end{align*}
	for all $x\in\operatorname{supp}\omega_0$. Further combining the preceding estimates with \eqref{mu 1}-\eqref{mu 3} and \eqref{mu 2}-\eqref{mu 4}, we conclude that
		\begin{align*}
			\left\|\xi_{n+1}\right\|_{C^1\left([a,b]\times\Omega\right)}
			\hspace{-10mm}&
			\\
			\leq &
			\left\|\omega_0\right\|_{L^\infty\left(\Omega\right)}
			+
			\frac C\eps \left\|\nabla\xi_{n+1}(a,\cdot)\right\|_{L^\infty\left(\Omega\right)}
			\left(
			1+
			\left\|\omega_0\right\|_{L^1\cap L^\infty(\Omega)}
			+|\gamma|
			\right)
			\\
			& \times e^{
				(b-a)\frac {C}{\eps^2}
				\left(1+
				\left\|\omega_0\right\|_{L^1\cap L^\infty\left(\Omega\right)}
				+|\gamma|
				\right)
				+(b-a)C
				\left\|\omega_0\right\|_{L^\infty(\mathbb{R}^2)}
				\log\left(1+\left\|\nabla\xi_n\right\|_{L^\infty \left([a,b]\times\mathbb{R}^2\right)}\right)
				}.
		\end{align*}
	Therefore, we deduce
	\begin{equation*}
		\left\|\xi_{n+1}\right\|_{C^1\left([a,b]\times\Omega\right)}
		\leq
		C_0+
		C_0\left\|\nabla\xi_{n+1}(a,\cdot)\right\|_{L^\infty\left(\Omega\right)}
		\left(1+\left\|\xi_n\right\|_{C^1\left([a,b]\times\Omega\right)}\right)^{C_0(b-a)},
	\end{equation*}
	where $C_0>0$ may only depend on $\left\|\omega_0\right\|_{L^1\cap L^\infty\left(\Omega\right)}$, $\gamma$, $\eps$ and $t_1$ but remains independent of $\xi_n$, $\xi_{n+1}$ and $[a,b]$. It follows that, setting $(b-a)$ sufficiently small so that, say,
	\begin{equation}\label{small time 2}
		C_0(b-a)\leq \frac 12,
	\end{equation}
	yields
	\begin{equation*}
		\left\|\xi_{n+1}\right\|_{C^1\left([a,b]\times\Omega\right)}
		\leq
		\frac 12+C_0+\frac {C_0^2}2
		\left\|\nabla\xi_{n+1}(a,\cdot)\right\|_{L^\infty\left(\Omega\right)}^2
		+\frac 12\left\|\xi_n\right\|_{C^1\left([a,b]\times\Omega\right)},
	\end{equation*}
	whence, for each $k=0,\ldots,n$,
        \begin{align*}
            \left\|\xi_{n+1}\right\|_{C^1\left([a,b]\times\Omega\right)}
            \leq &
            \left(\frac 12+C_0+\frac {C_0^2}2
            \left\|\nabla\xi_{n+1}(a,\cdot)\right\|_{L^\infty\left(\Omega\right)}^2
            \right)
            +\frac 12\left\|\xi_n\right\|_{C^1\left([a,b]\times\Omega\right)}
            \\
            \leq &
            \left(\frac 12+C_0+\frac {C_0^2}2
            \sup_{p\geq 0}\left\|\nabla\xi_{p}(a,\cdot)\right\|_{L^\infty\left(\Omega\right)}^2
            \right)\left(\sum_{j=0}^k2^{-j}\right)
            \\
            & +\frac 1{2^{k+1}}\left\|\xi_{n-k}\right\|_{C^1\left([a,b]\times\Omega\right)}
            \\
            \leq &
            1+2C_0+
            C_0^2\sup_{p\geq 0}\left\|\nabla\xi_{p}(a,\cdot)\right\|_{L^\infty\left(\Omega\right)}^2
            +\frac 1{2^{n+1}}\left\|\omega_{0}\right\|_{C^1\left(\Omega\right)}.
        \end{align*}
  
	Since the initial data $\omega_0$ belongs to $C^1\left(\Omega\right)$, the constant $C_0$ only depends on fixed parameters and the bound \eqref{small time 2} on the maximal length of $[a,b]$ only involves $C_0$, we deduce that we may propagate the preceding $C^1$-bound on $[a,b]$ to the whole interval $[0,t_1]$. This yields a uniform bound
	\begin{equation}\label{C1 bound}
		\sup_{n\geq 0} \left\|\xi_{n}\right\|_{C^1\left([0,t_1]\times\Omega\right)}<\infty.
	\end{equation}

	\bigskip
	
	\noindent\textbf{Convergence properties.}
	We have thus produced an approximating sequence $\left\{\xi_n\right\}_{n\geq 0}\subset C_{\omega_0,\eps}$ bounded, by virtue of \eqref{C1 bound}, in the metric topology induced by $C^1\left([0,t_1]\times\Omega\right)$. We analyze now its convergence properties.
	
	To this end, note that
		\begin{align*}
			\partial_{t} \left(\xi_{n+1}-\xi_n\right) + \mu_n\cdot\nabla\left(\xi_{n+1}-\xi_n\right) & = \left(\mu_{n-1}-\mu_n\right)\cdot \nabla \xi_{n},
			\\
			\partial_{t} \left(\xi_{n+1}-\xi_n\right) + \mu_{n-1}\cdot\nabla\left(\xi_{n+1}-\xi_n\right) & = \left(\mu_{n-1}-\mu_n\right)\cdot \nabla \xi_{n+1},
		\end{align*}
	whence, for any $[a,b]\subset [0,t_1]$,
		\begin{align*}
			\left(\xi_{n+1}-\xi_n\right)\left(b,Z_n(b,x)\right)
			= &
			\left(\xi_{n+1}-\xi_n\right)\left(a,Z_n(a,x)\right)
			\\
			& +
			\int_a^b
			\left(\mu_{n-1}-\mu_n\right)\cdot \nabla \xi_{n}\left(s,Z_n(s,x)\right)
			ds,
			\\
			\left(\xi_{n+1}-\xi_n\right)\left(b,Z_{n-1}(b,x)\right)
			= &
			\left(\xi_{n+1}-\xi_n\right)\left(a,Z_{n-1}(a,x)\right)
			\\
			& +
			\int_a^b
			\left(\mu_{n-1}-\mu_n\right)\cdot \nabla \xi_{n+1}\left(s,Z_{n-1}(s,x)\right)
			ds,
		\end{align*}
	which implies, utilizing estimates \eqref{mu 1}, \eqref{mu 2} and \eqref{C1 bound}, for each $k=0,\ldots, n-1$,
	\begin{equation}\label{cauchy estimate}
		\begin{aligned}
			\left\|\xi_{n+1}-\xi_n\right\|_{L^\infty\left([a,b]\times\Omega\right)} \hspace{-15mm}&
			\\
			\leq &
			\left\|\left(\xi_{n+1}-\xi_n\right)(a,\cdot)\right\|_{L^\infty\left(\Omega\right)}
			+
			C_1(b-a)
			\left\|\xi_{n}-\xi_{n-1}\right\|_{L^\infty\left([a,b]\times\Omega\right)}
			\\
			\leq &
			\sum_{j=0}^k
			\left(C_1(b-a)\right)^j
			\left\|\left(\xi_{n+1-j}-\xi_{n-j}\right)(a,\cdot)\right\|_{L^\infty\left(\Omega\right)}
			\\
			& +
			\left(C_1(b-a)\right)^{k+1}
			\left\|\xi_{n-k}-\xi_{n-1-k}\right\|_{L^\infty\left([a,b]\times\Omega\right)}
			\\
			\leq &
			\sum_{j=0}^{n-1}
			\left(C_1(b-a)\right)^j
			\left\|\left(\xi_{n+1-j}-\xi_{n-j}\right)(a,\cdot)\right\|_{L^\infty\left(\Omega\right)}
			\\
			& +
			\left(C_1(b-a)\right)^{n}
			\left\|\xi_{1}-\xi_{0}\right\|_{L^\infty\left([a,b]\times\Omega\right)},
		\end{aligned}
	\end{equation}
	for some independent constant $C_1>0$. As before, we set $(b-a)$ sufficiently small so that, say,
	\begin{equation*}
		C_1(b-a)\leq \frac 12.
	\end{equation*}
	In particular, since the $\xi_n$'s all have the same initial data $\omega_0$, we find that
		\begin{align*}
			\left\|\xi_{n+1}-\xi_n\right\|_{L^\infty\left([0,b-a]\times\Omega\right)}
			& \leq
			\left(C_1(b-a)\right)^{n}
			\left\|\xi_{1}-\xi_{0}\right\|_{L^\infty\left([0,b-a]\times\Omega\right)}
			\\
			& \leq
			\frac 1{2^{n-1}}
			\left\|\omega_0\right\|_{L^\infty\left(\Omega\right)}.
		\end{align*}
	Therefore, utilizing the elementary identity
	\begin{equation*}
		\sum_{j=0}^n \begin{pmatrix}j+k\\ k\end{pmatrix}= \begin{pmatrix}n+k+1\\ k+1\end{pmatrix},
	\end{equation*}
	for each $n,k\in\mathbb{N}$, we obtain
		\begin{align*}
			\left\|\xi_{n+1}-\xi_n\right\|_{L^\infty\left([k(b-a),(k+1)(b-a)]\times\Omega\right)}
			\leq &
			2
			\begin{pmatrix}n+k\\ k\end{pmatrix}
			\left(C_1(b-a)\right)^n
			\left\|\omega_0\right\|_{L^\infty\left(\Omega\right)}
			\\
			\leq &
			\frac 1{2^{n-1}}
			\begin{pmatrix}n+k\\ k\end{pmatrix}
			\left\|\omega_0\right\|_{L^\infty\left(\Omega\right)},
		\end{align*}
	whence
		\begin{align*}
			\left\|\xi_{n+1}-\xi_n\right\|_{L^\infty\left([0,t_1]\times\Omega\right)}
			& \leq
			\frac {C}{2^{\frac n2}},
			&&\text{for all }n\geq 0,
			\\
			\left\|\xi_{m}-\xi_n\right\|_{L^\infty\left([0,t_1]\times\Omega\right)}
			& \leq
			\frac {C'}{2^{\frac n2}},
			&&\text{for all }m>n\geq 0,
		\end{align*}
	for some independent constants $C,C'>0$.

	It follows that $\left\{\xi_n\right\}_{n\geq 0}$ is a Cauchy sequence in $C^0\left([0,t_1]\times\Omega\right)$ and, therefore, there exists $\omega^N\in C\left([0,t_1]\times\Omega\right)$ such that
	\begin{equation}\label{convergence xi mu}
		\begin{aligned}
			\xi_n & \longrightarrow \omega^N
			\quad\text{in }L^\infty\left([0,t_1]\times\Omega\right),
			\\
			\mu_n & \longrightarrow u^N
			\quad\text{in }L^\infty\left([0,t_1]\times K\right),\text{ for any compact set }K\subset\Omega,
		\end{aligned}
	\end{equation}
	where $u^N$ is defined by \eqref{uN} and we have used \eqref{mu 1} and \eqref{mu 2} to derive the convergence of $\mu_n$ from that of $\xi_n$ in both $L^\infty\left([0,t_1]\times\Omega\right)$ and $L^1\left([0,t_1]\times\Omega\right)$ (recall that the vorticities $\xi_n$ have uniformly bounded supports since they belong to $C_{\omega_0,\eps}$). It is then readily seen that $\omega^N$ solves \eqref{dynamic5} in the sense of distributions.

	\bigskip
	
	\noindent\textbf{Regularity of solution and conclusion of proof.}
	In order to complete the proof of (global for large $N$) wellposedness of \eqref{dynamic5} in $C^1$, there only remains to show that $\omega^N$ is actually of class $C^1$. Indeed, the uniqueness of solutions will then easily ensue from an estimate similar to \eqref{cauchy estimate}.
	
	For the moment, the uniform boundedness of $\left\{\xi_n\right\}_{n\geq 0}$ in $C^1\left([0,t_1	]\times\Omega\right)$ only allows us to deduce that $\omega^N$ is Lipschitz continuous (in $t$ and $x$). However, standard estimates (see \cite[p.\ 249]{courant}, for instance, or use the respresentation formula \eqref{representation velocity}) show that this Lipschitz continuity implies that $\nabla u^N$ exists and is continuous in $[0,t_1]\times\Omega$. It follows that the associated characteristic curve $Z^N(t,x)$ solving
	\begin{equation}\label{char ZN}
		\frac{dZ^N}{ds}=u^N\left(s,Z^N\right),
	\end{equation}
	for some given initial data $Z^N(0,x)=x\in\Omega$, belongs to $C^1\left([0,t_3]\times\Omega; \Omega\right)$ for some possibly small existence time $0<t_3\leq t_1$. Moreover, one easily estimates, using \eqref{distance epsilon} and \eqref{convergence xi mu}, that, for all $t\in[0,t_3]$ and $x\in\operatorname{supp}\omega_0$,
		\begin{align*}
			\left|Z_n(t,x)-Z^N(t,x)\right|
			\leq &
			\int_0^t\left|\mu_n\left(s,Z_n(s,x)\right)-u^N\left(s,Z^N(s,x)\right)\right|ds
			\\
			\leq &
			\int_0^t\left|\mu_n\left(s,Z_n(s,x)\right)-u^N\left(s,Z_n(s,x)\right)\right|ds
			\\
			& +
			\int_0^t\left|u^N\left(s,Z_n(s,x)\right)-u^N\left(s,Z^N(s,x)\right)\right|ds
			\\
			\leq &
			o(1)
			+
			C\int_0^t\left|Z_n(s,x)-Z^N(s,x)\right|ds,
		\end{align*}
	which implies, through a straighforward application of Gr\"onwall's lemma, that $Z_n$ converges uniformly in $(t,x)\in [0,t_3]\times\operatorname{supp}\omega_0$ towards $Z^N$. One can therefore assume that $t_3=t_1$, for $Z^N(t_3,x)$ remains uniformly bounded away from $\partial\Omega$ (at a distance at least $\eps$, to be precise), for each $x\in \operatorname{supp}\omega_0$.
	
	Next, as before, since the mapping $x\mapsto Z^N(t,x)$ is a $C^1$-diffeomorphism from $\operatorname{supp}\omega_0$ onto its own image, we consider its inverse $\left(Z^N\right)^{-1}(t,x)$. It is then readily seen that
	\begin{equation*}
		\left\{
		\begin{aligned}
			\omega^N(t,x) & = \omega_0\left(\left(Z^N\right)^{-1}(t,x)\right), && \text{if } x\in Z^N\left(t,\operatorname{supp}\omega_0\right),
			\\
			\omega^N(t,x) & = 0, && \text{otherwise},
		\end{aligned}
		\right.
	\end{equation*}
	which establishes that $\omega^N\in C^1_c\left([0,t_1]\times\Omega\right)$, for $\omega_0\in C^1_c\left(\Omega\right)$, and thereby concludes the proof of Theorem \ref{wellposedness} on the wellposedness of \eqref{dynamic5} in $C^1_c\left([0,t_1]\times\Omega\right)$. Further observe that, repeating the estimates leading up to \eqref{C1 bound} (replacing $\xi_n$ and $\xi_{n+1}$ by $\omega^N$ and $\mu_n$ by $u^N$), it is possible to show that the $C^1$-bound on $\omega^N$ is uniform in $N$.
\end{proof}

\subsection{Proof of Theorem \ref{main convergence}}

Considering the difference of \eqref{dynamic1} and \eqref{dynamic5}, note that
	\begin{align*}
		\partial_{t} \left(\omega-\omega^N\right) + u\cdot\nabla\left(\omega-\omega^N\right) & = - \left(u-u^N\right)\cdot \nabla \omega^N,
		\\
		\partial_{t} \left(\omega-\omega^N\right) + u^N\cdot\nabla\left(\omega-\omega^N\right) & = - \left(u-u^N\right)\cdot \nabla \omega,
	\end{align*}
whence, for every $(t,x)\in [0,t_1]\times\operatorname{supp}\omega_0$,
\begin{equation}\label{dynamic4}
	\begin{aligned}
		\left(\omega-\omega^N\right)\left(t,X(t,x)\right)
		& =
		- \int_0^t \left(u-u^N\right)\cdot \nabla \omega^N\left(s,X(s,x)\right) ds,
		\\
		\left(\omega-\omega^N\right)\left(t,Z^N(t,x)\right)
		& =
		- \int_0^t \left(u-u^N\right)\cdot \nabla \omega\left(s,Z^N(s,x)\right) ds,
	\end{aligned}
\end{equation}
where the characteristics $X(t,x)$ and $Z^N(t,x)$ are respectively defined by \eqref{char X} and \eqref{char ZN} and, as previously explained, exist for all $(t,x)\in [0,t_1]\times\operatorname{supp}\omega_0$ provided $N$ is sufficiently large. We also introduce the velocity flow
\begin{equation*}
	\breve u^N=K_{\mathbb{R}^2}[\omega]+u_{\rm app}^N[\omega,\gamma],
\end{equation*}
where $u_{\rm app}^N[\omega,\gamma]$ is given by \eqref{approx}-\eqref{point vortex}, for the same prescribed $\gamma\in\mathbb{R}$ and where $u_P$ in the right-hand side of \eqref{point vortex} is $K_{\mathbb{R}^2}[\omega]$.

By Theorem \ref{main theo}, it holds that, for any compact set $K\subset\Omega$,
\begin{equation}\label{dynamic2}
	\begin{aligned}
		\| u - \breve u^N \|_{L^\infty(K)}
		\hspace{-10mm} &
		\\
		& \leq \frac{C}{N^\kappa}
		\left(\left(
		\frac 1{\operatorname{dist}\left(\operatorname{supp}\omega,\partial\Omega\right)} + \frac 1{\operatorname{dist}\left(\operatorname{supp}\omega,\partial\Omega\right)^{\kappa+1}}
		\right)\left\|\omega\right\|_{L^1\left(\mathbb{R}^2\right)} + |\gamma|\right),
	\end{aligned}
\end{equation}
where the constant $C>0$ may now depend on $K$.
We also estimate, using \eqref{mu 1} and \eqref{mu 2}, that
\begin{equation}\label{dynamic3}
	\begin{aligned}
		\| \breve u^N - u^N \|_{L^\infty(K)}
		\leq &
		\left\|K_{\mathbb{R}^2}\left[\omega-\omega^N\right]\right\|_{L^\infty(K)}
		+
		\left\|u_{\rm app}^N\left[\omega-\omega^N,0\right]\right\|_{L^\infty(K)}
		\\
		\leq &
		C\left\|\omega-\omega^N\right\|_{L^1\cap L^\infty\left(\Omega\right)}
		\\
		\leq &
		C\left(1+ \left|\operatorname{supp}\omega\right| + \left|\operatorname{supp}\omega^N\right| \right)
		\left\|\omega-\omega^N\right\|_{L^\infty(\Omega)}.
	\end{aligned}
\end{equation}
Therefore, recalling that $\omega$ and $\omega^N$ are all uniformly bounded in $C^1\left([0,t_1]\times\Omega\right)$, combining \eqref{dynamic2} and \eqref{dynamic3} with \eqref{dynamic4} and choosing the compact set $K$ so that $\operatorname{supp}\omega\cup\operatorname{supp}\omega^N\subset K$, for all $N$, we find that
	\begin{align*}
		\left\|(\omega-\omega^N)(t)\right\|_{L^\infty(\Omega)}
		= &
		\left\|(\omega-\omega^N)(t)\right\|_{L^\infty(K)}
		\\
		& \leq
		C
		\int_0^t\left\|(u-u^N)(s)\right\|_{L^\infty(K)}ds
		\\
		& \leq
		\frac C{N^\kappa}
		+C
		\int_0^t\left\|(\omega-\omega^N)(s)\right\|_{L^\infty(\Omega)}ds.
	\end{align*}
Finally, by Gr\"onwall's lemma, we deduce that
\begin{equation*}
	\left\|(\omega-\omega^N)(t)\right\|_{L^\infty(\Omega)}
	\leq \frac C{N^\kappa}e^{Ct},
\end{equation*}
which concludes the proof of the theorem. \qed

\section{An alternative approach: the fluid charge method}\label{charges}

We wish now to propose an alternative method for approximating $u_R$ by discretizing the boundary of the domain in \eqref{eq uR}. It consists in constructing an approximate flow
\begin{equation}\label{approx alt}
	\check u_{\rm app}^N(x):=\frac1{2\pi} \sum_{j=1}^N \frac{\check\gamma_{j}^N}N \frac{x - x_{j}^N}{|x - x_{j}^N|^2} + \gamma H_*(x),
\end{equation}
where the positions $(x_{1}^N, x_{2}^N,\dots , x_{N}^N)$ on the boundary $\partial\Omega$ are still determined by \eqref{xi} and we assume that $H_*(x)\in C^\infty\left(\Omega\right)$ is a given vector field solving \eqref{harmonic5}, such that all its derivatives are continuous up to the boundary $\partial\Omega$, $H_*\cdot n$ has mean zero over $\partial\Omega$ and $H_*\cdot n\in C^\infty\left(\partial\Omega\right)$.

For practical purposes, the field $H_*(x)$ should be either known explicitely or previously computed by other means. For instance, one may consider the harmonic vector field $H_*=H(x)$ defined by \eqref{harmonic} or a single point vortex velocity field $H_*(x)=\frac{(x-x_*)^\perp}{2\pi |x-x_*|^2}=K_{\R^2}[\delta_{x_*}]$, for any given $x_*\in {\overline \Omega^c}$.

Observe that \eqref{approx alt} is essentially a discretization of \eqref{boundary sheet omega 4}. The clear advantage of discretizing \eqref{boundary sheet omega 4} over \eqref{boundary sheet omega} resides in that \eqref{boundary sheet omega 4} only involves the inversion of the regular perturbation of the identity $A+\pi$ whereas \eqref{boundary sheet omega} requires the inversion of the singular integral operator $B$. We argue below, in Section \ref{diagonal}, that this provides a more efficient discretization method because it often yields better conditioned matrices.

By analogy with the electric field produced by single electric charges, we refer to the building blocks $\frac{x - y}{2\pi|x - y|^2}$, with $y\in\mathbb{R}^2$, of the above flow as \emph{fluid charges}. Recall that a fluid charge satisfies
\begin{equation*}
	\left\{
	\begin{array}{l}
		\div \frac{x}{2\pi|x |^2} = \delta(x),
		\\
		\curl \frac{x}{2\pi|x|^2} = 0.
	\end{array}
	\right.
\end{equation*}

\subsection{Static convergence of the fluid charge approximation}

As before, concerning $u_{\rm app}^N$, it is {\it a priori} not obvious that such a flow $\check u_{\rm app}^N$ can be made a good approximation of $u_R$. Nonetheless, note that $\check u_{\rm app}^N$ already naturally satisfies, for any smooth simple closed curve $c_0$ enclosing the obstacle $\mathcal{C}$,
\begin{equation*}
	\left\{
	\begin{array}{lcl}
		\div \check u_{\rm app}^N  =0 & \text{in}& \Omega, \\
		\curl \check u_{\rm app}^N  =0 & \text{in}& \Omega, \\
		\check u_{\rm app}^N  \rightarrow 0 &\text{as}& x\rightarrow\infty, \\
		\oint_{c_0} \check u_{\rm app}^N \cdot \tau ds  = \gamma, &&
	\end{array}
	\right.
\end{equation*}
where $\tau$ denotes the unit tangent vector on $c_0$. Following the developments leading to \eqref{point vortex}, it would now be tempting to enforce that the boundary condition be satisfied as $N\rightarrow\infty$ by setting the density $\check \gamma^N=(\check\gamma_{1}^N,\dots, \check\gamma_{N}^N)\in\mathbb{R}^N$ to be the solution of the following system of $N$ linear equations:
\begin{equation}\label{point vortex alt wrong}
	\frac1{2\pi}\sum_{j=1}^N \frac{\check\gamma_{j}^N}N \frac{\tilde x_{i}^N - x_{j}^N}{|\tilde x_{i}^N - x_{j}^N|^2}\cdot n(\tilde x_{i}^N)
	= -[\left(u_{P}+\gamma H_*\right)\cdot n](\tilde x_{i}^N), \ \text{for all }i=1,\dots, N,
\end{equation}
for some appropriate intermediate mesh points $(\tilde x_{1}^N,\tilde  x_{2}^N,\dots , \tilde x_{N}^N)$ on the boundary $\partial\Omega$. This is, however, not feasible in general because the resulting matrix may not be invertible (consider the case of the unit disk where $\frac{x-y}{|x-y|^2}\cdot n(x)=\frac 12$ for all $x,y\in\partial B(0,1)$). In fact, it turns out that, instead of \eqref{point vortex alt wrong}, the correct condition on the density $\check \gamma^N$ is given by the system
\begin{equation}\label{point vortex alt}
	\begin{aligned}
		\left(\frac1{2\pi}\sum_{j=1}^N \frac{\check\gamma_{j}^N}N \frac{\tilde x_{i}^N - x_{j}^N}{|\tilde x_{i}^N - x_{j}^N|^2}\cdot n(\tilde x_{i}^N)\right)
		+\frac{\check\gamma_i^N}{2\left|\partial\Omega\right|} \hspace{-20mm} &
		\\
		= - & [\left(u_{P}+\gamma H_*\right)\cdot n](\tilde x_{i}^N), \ \text{for all }i=1,\dots, N,
	\end{aligned}
\end{equation}
which is inspired by the integral representation \eqref{boundary sheet omega 4} and the inversion of the operator $A+\pi$. In order to emphasize the dependence of $\check u_{\rm app}^N$ on $\omega$ (through $u_P$) and $\gamma$, we may denote $\check u_{\rm app}^N=\check u_{\rm app}^N[\omega,\gamma]$. Note that $\check u_{\rm app}^N$ is linear in $(\omega,\gamma)$.

In the notation introduced in \eqref{xi}-\eqref{tildexi}, this system can be recast as
\begin{equation*}
	\frac1{N}\sum_{j=1}^N \check\gamma_{j}^N \frac{l\left(\tilde s_{i}^N\right) - l\left(s_{j}^N\right)}
	{\left|l\left(\tilde s_{i}^N\right) - l\left(s_{j}^N\right)\right|^2}\cdot n\left(l\left(\tilde s_{i}^N\right)\right)
	+\frac{\pi \check\gamma_i^N}{\left|\partial\Omega\right|}
	= f\left(\tilde s_i^N\right), \quad \text{for all }i=1,\dots, N,
\end{equation*}
where $f(s)=-2\pi[\left(u_{P}+\gamma H_*\right)\cdot n]\left(l(s)\right)$, for all $s\in\left[0,\left|\partial\Omega\right|\right]$. Equivalently, in the matrix notation of Section \ref{section discrete}, this system becomes
\begin{equation}\label{point toy alt}
	\left(\frac 1N A_N+\frac{\pi}{\left|\partial\Omega\right|}\right)\check\gamma^N
	= \left(f\left(\tilde s_i^N\right)\right)_{1\leq i\leq N}.
\end{equation}
We use here the convention that, whenever $\tilde s_i^N=s_j^N$, the entry of $A_N$ corresponding to $\frac{l\left(\tilde s_{i}^N\right) - l\left(s_{j}^N\right)} {\left|l\left(\tilde s_{i}^N\right) - l\left(s_{j}^N\right)\right|^2}\cdot n\left(l\left(\tilde s_{i}^N\right)\right)$ is naturally determined by the limiting value of the smooth kernel $\frac{l\left(s\right) - l\left(t\right)} {\left|l\left(s\right) - l\left(t\right)\right|^2}\cdot n\left(l\left(s\right)\right)$ as $t$ and $s$ converge to the same value $t_0$, say, that is
	\begin{align*}
		\lim_{s,t\to t_0}\frac{l\left(s\right) - l\left(t\right)} {\left|l\left(s\right) - l\left(t\right)\right|^2}\cdot n\left(l\left(s\right)\right)
		& =
		\lim_{s,t\to t_0} \frac{-\frac 12 l''\left(s\right)(t-s)^2+o(|t-s|^2)} {\left|l\left(s\right) - l\left(t\right)\right|^2}\cdot n\left(l\left(s\right)\right)
		\\
		& =
		-\frac 12 l''\left(t_0\right)\cdot n\left(l\left(t_0\right)\right).
	\end{align*}
In particular, by the uniform boundedness of each component of $A_N$, there exists a constant $C>0$ independent of $N$ such that, for each $1\leq p\leq\infty$,
\begin{equation}\label{boundedness different mesh}
	\frac 1N
	\left\|A_Nz\right\|_{\ell^p}
	\leq C\left\|z\right\|_{\ell^1},
\end{equation}
for all $z\in\mathbb{R}^N$.

In this section, we are going to adopt a new notion of well distributedness, introduced in the coming definition, which differs from the one defined by \eqref{mesh2}-\eqref{mesh}. In order to avoid any possible confusion, we will refer to these newly introduced meshes as being well-$*$ distributed, which will distinguish them from the well distributed meshes determined by \eqref{mesh2}-\eqref{mesh}. It is to be emphasized that well-$*$ distributed meshes are used in the present Section \ref{charges} only.

\begin{definition}
	We say that the points $\left\{x_i^N\right\}_{1\leq i\leq N}$ and $\left\{\tilde x_i^N\right\}_{1\leq i\leq N}$ given by $x_{i}^N:=l\left( s_i^N\right)$ and $\tilde x_{i}^N:=l\left(\tilde s_i^N\right)$, where $(s_{1}^N,\dots , s_{N}^N)\in \mathbb{R}^N$ and $(\tilde s_{1}^N, \dots , \tilde s_{N}^N)\in \mathbb{R}^N$, are \emph{well-$*$ distributed} on $\partial\Omega$ if there exists an integer $\kappa \geq 2$ such that, as $N\to\infty$,
	\begin{equation}\label{mesh3}
		\max_{i=1,\ldots,N}\left|s_i^N-\theta_i^N\right|=\mathcal{O}\left(N^{-\kappa}\right)
		\quad\text{and}\quad
		\max_{i=1,\ldots,N}\left|\tilde s_i^N-\theta_i^N\right|=\mathcal{O}\left(N^{-\kappa}\right),
	\end{equation}
	where
	\begin{equation}\label{mesh4}
	\theta_{i}^N = \frac{(i-1) \left|\partial\Omega\right|}{N}
	\quad \text{for all } i=1,\dots, N.
	\end{equation}
\end{definition}

Our main result concerning the alternative approximation method discussed in the present section is analog to Theorem \ref{main theo} and states that the approximate flow $\check u_{\rm app}^N$, constructed through the procedure \eqref{point vortex alt}, is a good approximation of $u_{R}$ provided the vortices are well-$*$ distributed on $\partial\Omega$:

\begin{theorem}\label{main theo alt}
	Let $\omega\in L^1_c\left(\Omega\right)$ and $\gamma\in\mathbb{R}$ be given. For any $N\geq 2$, we consider a well-$*$ distributed mesh satisfying \eqref{mesh3} and $u_{P}$ defined in \eqref{uP}.
	
	Then, there exists $N_0$ (independent of $\omega$ and $\gamma$) such that, for all $N\geq N_0$, the system \eqref{point vortex alt} admits a unique solution $\check \gamma^N\in \R^N$. Moreover, for any closed set $K\subset \Omega$  there exists a constant $C>0$ independent of $N$, $K$, $\omega$ and $\gamma$ such that
		\begin{align*}
			\| u_{R} - \check u_{\rm app}^N \|_{L^\infty(K)}
			\hspace{-5mm}&
			\\
			\leq & \frac{C}{N^\kappa}
			\left(
			\frac 1{\operatorname{dist}\left(K,\partial\Omega\right)} + \frac 1{\operatorname{dist}\left(K,\partial\Omega\right)^{\kappa+2}}
			\right)
			\\
			& \times \left(\left(
			\frac 1{\operatorname{dist}\left(\operatorname{supp}\omega,\partial\Omega\right)} + \frac 1{\operatorname{dist}\left(\operatorname{supp}\omega,\partial\Omega\right)^{\kappa+1}}
			\right)\left\|\omega\right\|_{L^1\left(\mathbb{R}^2\right)} + |\gamma|\right),
		\end{align*}
	where $\check u_{\rm app}^N$ is given by \eqref{approx alt} in terms of $\check \gamma^N$ and $u_{R}$ is the continuous flow \eqref{uR}.
\end{theorem}

The proof of Theorem \ref{main theo alt} relies on the coming Propositions \ref{inverse D alt} and \ref{prop 32 alt}. It is given {\it per se} after the proof of Proposition \ref{prop 32 alt} is completed, below.

Adapting the proof of Lemma \ref{inverse D} to the present situation and using estimate \eqref{gelfand 2} on the spectral radius of $A-\pi$ rather than \eqref{gelfand} yields the following result.

\begin{proposition}\label{inverse D alt}
	For any integer $N\geq 2$, consider a well-$*$ distributed mesh $(s_{1}^N,\dots , s_{N}^N)\in \mathbb{R}^N$, $(\tilde s_{1}^N, \dots , \tilde s_{N}^N)\in \mathbb{R}^N$. Then, there exist $N_*,k_*\geq 1$ and $\delta>0$ such that
		\begin{align*}
			\left\|
			\left( \frac{\left|\partial\Omega\right|}{N} A_N-\pi\right)^k
			\right\|_{\mathcal{L}\left(\ell^2\right)}^\frac 1k
			\leq 2\pi-\delta,
		\end{align*}
	for all $k\geq k_*$ and $N\geq N_*$.
	
	In particular, provided $N$ is sufficiently large, the Neumann series
	\begin{equation*}
			\left(\frac{\left|\partial\Omega\right|}{N} A_N+\pi\right)^{-1}
			=
			\frac 1{2\pi}\sum_{k=0}^\infty \left(\frac{\pi - \frac{\left|\partial\Omega\right|}{N} A_N}{2\pi} \right)^{k},
	\end{equation*}
	is absolutely convergent in $\mathcal{L}\left(\ell^2\right)$ and the inverse operator it defines is bounded in $\mathcal{L}\left(\ell^2\right)$ uniformly in $N$. It follows that, provided $N$ is sufficiently large, the following problem:
	\begin{equation*}
		z\in \R^N,\quad \left(\frac{\left|\partial\Omega\right|}N A_{N}+\pi\right)z=v,
	\end{equation*}
	has a unique solution for any given $v\in \R^{N}$. Moreover, this solution satisfies:
	\begin{equation*}
	\| z \|_{\ell^1}\leq \| z \|_{\ell^2}
	\leq
	C\| v \|_{\ell^2}
	\leq
	C\| v \|_{\ell^\infty},
	\end{equation*}
	for some independent constant $C>0$.
\end{proposition}

\begin{proof}
	Following the proof of Lemma \ref{inverse D}, for each $k\geq 1$, we denote by $K_k(x,y)$ the kernel of $A^k$, which is smooth and satisfies, for all $x,y\in\partial\Omega$,
		\begin{align*}
			K_k(x,y)= &
			\int_{\partial\Omega\times\dots\times\partial\Omega}
			\frac{x-y_1}{|x-y_1|^2}\cdot n(x)
			\\
			&\times
			\left(\prod_{j=1}^{k-2}\frac{y_j-y_{j+1}}{|y_j-y_{j+1}|^2}\cdot n(y_j)\right)
			\frac{y_{k-1}-y}{|y_{k-1}-y|^2}\cdot n(y_{k-1})
			dy_1\dots dy_{k-1}.
		\end{align*}
	Therefore, by smoothness, approximating in $L_{x,y}^\infty\left(\partial\Omega\times\partial\Omega\right)$ the kernel $\frac{x-y}{|x-y|^2}\cdot n(x)$ by
	\begin{equation*}
		\sum_{i,j=1}^N
		\mathds{1}_{\left[\theta_{i}^N, \theta_{i+1}^N\right)}(s)
		\frac{l\left(\tilde s_{i}^N\right) - l\left(s_{j}^N\right)}
		{\left|l\left(\tilde s_{i}^N\right) - l\left(s_{j}^N\right)\right|^2}\cdot n\left(l\left(\tilde s_{i}^N\right)\right)
		\mathds{1}_{\left[\theta_{j}^N, \theta_{j+1}^N\right)}(s_*),
	\end{equation*}
	where the $\theta_i^N$'s are defined in \eqref{mesh4} and we identify $x=l(s)$ and $y=l(s_*)$, we find, as $N\to\infty$, that $K_{k}(x,y)$ is arbitrarily close in $L_{x,y}^\infty\left(\partial\Omega\times\partial\Omega\right)$ to
		\begin{align*}
			& \sum_{i,j=1}^N
			\mathds{1}_{\left[\theta_{i}^N, \theta_{i+1}^N\right)}(s)
			\mathds{1}_{\left[\theta_{j}^N, \theta_{j+1}^N\right)}(s_*)
			\\
			& \times \left(\frac{\left|\partial\Omega\right|}{N}\right)^{k-1}
			\sum_{j_1,\ldots,j_{k-1}=1}^N
			\frac{l\left(\tilde s_{i}^N\right)-l\left(s_{j_1}^N\right)}{\left|l\left(\tilde s_{i}^N\right)-l\left(s_{j_1}^N\right)\right|^2}\cdot n\left(l\left(\tilde s_{i}^N\right)\right)
			\\
			&\times
			\left(\prod_{n=1}^{k-2}
			\frac{l\left(\tilde s_{j_n}^N\right)-l\left(s_{j_{n+1}}^N\right)}{\left|l\left(\tilde s_{j_n}^N\right)-l\left(s_{j_{n+1}}^N\right)\right|^2}\cdot n\left(l\left(\tilde s_{j_n}^N\right)\right)
			\right)
			\frac{l\left(\tilde s_{j_{k-1}}^N\right)-l\left(s_{j}^N\right)}{\left|l\left(\tilde s_{j_{k-1}}^N\right)-l\left(s_{j}^N\right)\right|^2}\cdot n\left(l\left(\tilde s_{j_{k-1}}^N\right)\right)
			\\
			& =
			\left(\frac{\left|\partial\Omega\right|}{N}\right)^{k-1}
			\sum_{i,j=1}^N
			\mathds{1}_{\left[\theta_{i}^N, \theta_{i+1}^N\right)}(s)
			\left(A_N^k\right)_{ij}
			\mathds{1}_{\left[\theta_{j}^N, \theta_{j+1}^N\right)}(s_*).
		\end{align*}

	It follows that, for any fixed $k\geq 1$ and $\eps>0$, provided $N$ is sufficiently large, with the convention that $K_0(l(s),l(s_*))$ denotes a Dirac mass at $s=s_*$,
		\begin{align*}
			\left\|\left(A-\pi\right)^{k}\right\|_{\mathcal{L}\left(L^2\right)}+\eps \hspace{-20mm} &
			\\
			& =
			\left\|\sum_{n=0}^k \begin{pmatrix}k\\ n\end{pmatrix}(-\pi)^{k-n}A^n\right\|_{\mathcal{L}\left(L^2\right)}+\eps
			\\
			& = \sup_{\varphi\in L^2\left(\partial\Omega\right)}
			\frac{\left\|
			\sum_{n=0}^k \begin{pmatrix}k\\ n\end{pmatrix}(-\pi)^{k-n}\int_0^{\left|\partial\Omega\right|}K_n\left(l(s),l(s_*)\right)\varphi(l(s_*))ds_*\right\|_{L^2_s}}{\left\|\varphi(l(s))\right\|_{L^2_s}}
			+\eps
			\\
			& \geq
			\sup_{z\in\mathbb{R}^N}
			\frac
			{\left\|
			\sum_{n=0}^k \begin{pmatrix}k\\ n\end{pmatrix}(-\pi)^{k-n}
			\left(\frac{\left|\partial\Omega\right|}{N}\right)^{n}
			\sum_{i,j=1}^N
			\mathds{1}_{\left[\theta_{i}^N, \theta_{i+1}^N\right)}(s)
			\left(A_N^n\right)_{ij}
			z_j\right\|_{L^2_s}}
			{\left\|\sum_{i=1}^N z_i \mathds{1}_{\left[\theta_{i}^N, \theta_{i+1}^N\right)}(s)\right\|_{L^2_s}}
			\\
			& =
			\sup_{z\in\mathbb{R}^N}
			\frac
			{\left\|
			\sum_{n=0}^k \begin{pmatrix}k\\ n\end{pmatrix}(-\pi)^{k-n}
			\left(\frac{\left|\partial\Omega\right|}{N}\right)^{n}
			A_N^n z\right\|_{\ell^2}}
			{\left\|z\right\|_{\ell^2}}
			\\
			& =
			\sup_{z\in\mathbb{R}^N}
			\frac{
			\left\|\left(\frac{\left|\partial\Omega\right|}{N}A_N-\pi\right)^kz\right\|_{\ell^2}
			}{\left\|z\right\|_{\ell^2}}
			=\left\|\left(\frac{\left|\partial\Omega\right|}{N}A_N-\pi\right)^k\right\|_{\mathcal{L}\left(\ell^2\right)}.
		\end{align*}
	Further deducing from estimate \eqref{gelfand 2} that there exist $k_0\geq 1$ and $\delta>0$ such that $\left\|\left(A-\pi\right)^{k_0}\right\|_{\mathcal{L}\left(L^2\right)}^\frac 1{k_0} \leq 2\pi-3\delta$, we infer that, setting $\eps>0$ sufficiently small,
	\begin{equation*}
		\left\|\left(\frac{\left|\partial\Omega\right|}{N}A_N-\pi\right)^{k_0}\right\|_{\mathcal{L}\left(\ell^2\right)}
		\leq
		\left\|\left(A-\pi\right)^{k_0}\right\|_{\mathcal{L}\left(L^2\right)}+\eps
		\leq \left(2\pi -3\delta\right)^{k_0} + \eps \leq \left(2\pi -2\delta\right)^{k_0},
	\end{equation*}
	for $N$ sufficiently large.

	Now, for any $k\geq k_0$, we write $k=pk_0+q$ with positive integral numbers and $0\leq q\leq k_0-1$. Then, we obtain
		\begin{align*}
			\left\|\left(\frac{\left|\partial\Omega\right|}{N}A_N-\pi\right)^{k}\right\|_{\mathcal{L}\left(\ell^2\right)}
			& \leq
			\left\|\left(\frac{\left|\partial\Omega\right|}{N}A_N-\pi\right)^{k_0}\right\|_{\mathcal{L}\left(\ell^2\right)}^p
			\left\|\frac{\left|\partial\Omega\right|}{N}A_N-\pi\right\|_{\mathcal{L}\left(\ell^2\right)}^q
			\\
			& \leq \left(2\pi -2\delta\right)^{k-q}
			\left\|\frac{\left|\partial\Omega\right|}{N}A_N-\pi\right\|_{\mathcal{L}\left(\ell^2\right)}^q.
		\end{align*}
	Further using that $N^{-1}A_N$ is a bounded operator over $\ell^2$ uniformly in $N$ (see \eqref{boundedness different mesh}), we arrive at, for some fixed constant $C_*>0$ independent of $N$ and $k$, and for sufficiently large $k$,
	\begin{equation*}
		\left\|\left(\frac{\left|\partial\Omega\right|}{N}A_N-\pi\right)^{k}\right\|_{\mathcal{L}\left(\ell^2\right)}
		\leq C_* \left(2\pi -2\delta\right)^{k}\leq \left(2\pi -\delta\right)^{k},
	\end{equation*}
	which concludes the proof of the proposition.
\end{proof}

The coming result is an adaptation of Proposition \ref{prop 32} and establishes the weak convergence of the discretization of the operator $A+\pi$.

\begin{proposition}\label{prop 32 alt}
	For any integer $N\geq 2$, consider a well-$*$ distributed mesh $(s_{1}^N,\dots , s_{N}^N)\in \mathbb{R}^N$, $(\tilde s_{1}^N, \dots , \tilde s_{N}^N)\in \mathbb{R}^N$ satisfying \eqref{mesh3} and, according to Proposition \ref{inverse D alt}, consider the solution $\check\gamma^N=(\check\gamma_{1}^N,\dots, \check\gamma_{N}^N)\in\mathbb{R}^N$ to the system \eqref{point toy alt} for some periodic function $f\in C^{\kappa}\left(\left[0,\left|\partial\Omega\right|\right]\right)$, where $\kappa \geq 2$. We define the approximations
	\begin{equation}\label{f app alt}
		\begin{aligned}
			\check f_{\rm app}^N(s) & :=
			\frac1{N}\sum_{j=1}^N \check\gamma_{j}^N \frac{l\left(s\right) - l\left(s_{j}^N\right)}
			{\left|l\left(s\right) - l\left(s_{j}^N\right)\right|^2}\cdot \tau\left(l\left(s\right)\right),
			\\
			\check g_{\rm app}^N(s) & :=
			\frac1{N}\sum_{j=1}^N \check\gamma_{j}^N \frac{l\left(s\right) - l\left(s_{j}^N\right)}
			{\left|l\left(s\right) - l\left(s_{j}^N\right)\right|^2}\cdot n\left(l\left(s\right)\right).
		\end{aligned}
	\end{equation}

	Then, for any periodic test function $\varphi\in C^{\infty}\left(\left[0,\left|\partial\Omega\right|\right]\right)$,
		\begin{align*}
			\left|\int_{0}^{\left|\partial\Omega\right|} \left(\check f_{\rm app}^N - B\left(A+\pi\right)^{-1}f \right)\varphi\right|
			& \leq
			\frac{C}{N^{\kappa}}
			\left\|f\right\|_{C^{\kappa}}
			\left\|\varphi\right\|_{C^{\kappa+1}},
			\\
			\left|\int_{0}^{\left|\partial\Omega\right|} \left(\check g_{\rm app}^N - A\left(A+\pi\right)^{-1}f \right)\varphi\right|
			& \leq
			\frac{C}{N^{\kappa}}
			\left\|f\right\|_{C^{\kappa}}
			\left\|\varphi\right\|_{L^2},
		\end{align*}
	where we identify the variable $x$ with the variable $s$ whenever $x=l(s)\in\partial\Omega$ and the singular integrals are defined in the sense of Cauchy's principal value.
\end{proposition}

\begin{proof}
	For any $h\in C^{\kappa}\left(\left[0,\left|\partial\Omega\right|\right]\right)$, an estimate based on Corollary \ref{riemann2} yields that
		\begin{align*}
			\left\|\left(Ah\left(\tilde s_i^N\right)\right)_{1\leq i\leq N}
			-
			\frac{\left|\partial\Omega\right|}{N}A_N\left(h\left(s_i^N\right)\right)_{1\leq i\leq N}
			\right\|_{\ell^\infty}
			\hspace{-70mm} &
			\\
			\leq &
			\sup_{s\in \left[0,\left|\partial\Omega\right|\right]}
			\left|
			\int_0^{\left|\partial\Omega\right|}\frac{l\left(s\right) - l\left(s_*\right)}
			{\left|l\left(s\right) - l\left(s_*\right)\right|^2}\cdot n\left(l\left(s\right)\right)
			h\left(l(s_*)\right)ds_*
			\vphantom{
			-
			\frac{\left|\partial\Omega\right|}{N} \sum_{i=1}^N
			\frac{l\left(s\right) - l\left(s_i^N\right)}
			{\left|l\left(s\right) - l\left(s_i^N\right)\right|^2}\cdot n\left(l\left(s\right)\right)
			h\left(l(s_i^N)\right)
			}
			\right.
			\\
			& -
			\left.
			\vphantom{
			\int_0^{\left|\partial\Omega\right|}\frac{l\left(s\right) - l\left(s_*\right)}
			{\left|l\left(s\right) - l\left(s_*\right)\right|^2}\cdot n\left(l\left(s\right)\right)
			h\left(l(s_*)\right)ds_*
			}
			\frac{\left|\partial\Omega\right|}{N} \sum_{i=1}^N
			\frac{l\left(s\right) - l\left(s_i^N\right)}
			{\left|l\left(s\right) - l\left(s_i^N\right)\right|^2}\cdot n\left(l\left(s\right)\right)
			h\left(l(s_i^N)\right)
			\right|
			\\
			\leq & \frac{C}{N^{\kappa}}
			\sup_{s\in \left[0,\left|\partial\Omega\right|\right]}
			\left\|\frac{l\left(s\right) - l\left(s_*\right)}
			{\left|l\left(s\right) - l\left(s_*\right)\right|^2}\cdot n	\left(l\left(s\right)\right)
			\right\|_{C^{\kappa}_{s_*}
			\left(\left[0,\left|\partial\Omega\right|\right]\right)}\left\|h\right\|_{C^{\kappa}}
			\leq
			\frac{C}{N^{\kappa}}\left\|h\right\|_{C^{\kappa}}.
		\end{align*}
	Exploiting \eqref{mesh3} and that $N^{-1}A_N:\ell^1\to\ell^\infty$ is a bounded operator (see \eqref{boundedness different mesh}), uniformly in $N$, we also find that
		\begin{align*}
			\left\|\left(Ah\left(\tilde s_i^N\right)\right)_{1\leq i\leq N}
			-
			\frac{\left|\partial\Omega\right|}{N}A_N\left(h\left(\tilde s_i^N\right)\right)_{1\leq i\leq N}
			\right\|_{\ell^\infty} \hspace{-30mm} &
			\\
			\leq &
			\frac{C}{N^{\kappa}}\left\|h\right\|_{C^{\kappa}}
			+
			C\left\|
			\left(h\left(s_i^N\right)-h\left(\tilde s_i^N\right)\right)_{1\leq i\leq N}
			\right\|_{\ell^\infty}
			\\
			\leq &
			\frac{C}{N^{\kappa}}\left(\left\|h\right\|_{C^{\kappa}}
			+\left\|h\right\|_{C^{1}}\right)
			\leq
			\frac{C}{N^{\kappa}}\left\|h\right\|_{C^{\kappa}}.
		\end{align*}
		In particular, setting $h=\left(A+\pi\right)^{-1}f$ in the preceding estimate, exploiting the relation \eqref{point toy alt} and using the uniform boundedness of the inverse operator 
		\begin{equation*}
			\left(\frac{\left|\partial\Omega\right|}{N}A_N+\pi\right)^{-1}
		\end{equation*}
		examined in Proposition \ref{inverse D alt}, we deduce that
		\begin{align*}
			\left\|
			\frac{\check\gamma^N}{\left|\partial\Omega\right|}
			-
			\left(\left(A+\pi\right)^{-1}f\left(\tilde s_i^N\right)\right)_{1\leq i\leq N}
			\right\|_{\ell^2} \hspace{-50mm} &
			\\
			= &
			\left\|\left(\frac{\left|\partial\Omega\right|}{N}A_N+\pi\right)^{-1}\left(f\left(\tilde s_i^N\right)\right)_{1\leq i\leq N}
			-
			\left(\left(A+\pi\right)^{-1}f\left(\tilde s_i^N\right)\right)_{1\leq i\leq N}
			\right\|_{\ell^2}
			\\
			\leq &
			C\left\|\left(f\left(\tilde s_i^N\right)\right)_{1\leq i\leq N}
			-
			\left(\frac{\left|\partial\Omega\right|}{N}A_N+\pi\right)
			\left(\left(A+\pi\right)^{-1}f\left(\tilde s_i^N\right)\right)_{1\leq i\leq N}
			\right\|_{\ell^\infty}
			\\
			= &
			C\left\|\left(A\left(A+\pi\right)^{-1}f\left(\tilde s_i^N\right)\right)_{1\leq i\leq N}
			-
			\frac{\left|\partial\Omega\right|}{N}A_N\left(\left(A+\pi\right)^{-1}f\left(\tilde s_i^N\right)\right)_{1\leq i\leq N}
			\right\|_{\ell^\infty}
			\\
			\leq &
			\frac{C}{N^{\kappa}}\left\|\left(A+\pi\right)^{-1}f\right\|_{C^{\kappa}}
			\leq
			\frac{C}{N^{\kappa}}\left(
			\left\|f\right\|_{C^{\kappa}}
			+
			\left\|A\left(A+\pi\right)^{-1}f\right\|_{C^{\kappa}}
			\right)
			\\
			\leq &
			\frac{C}{N^{\kappa}}
			\left\|f\right\|_{C^\kappa},
		\end{align*}
	where we have used in the last step above that $A$ has a smooth kernel and $\left(A+\pi\right)^{-1}$ is bounded over $L^2$. Further using \eqref{mesh3} finally yields the similar estimate
	\begin{equation}\label{check gamma conv}
		\begin{aligned}
			\left\|
			\frac{\check\gamma^N}{\left|\partial\Omega\right|}
			-
			\left(\left(A+\pi\right)^{-1}f\left(s_i^N\right)\right)_{1\leq i\leq N}
			\right\|_{\ell^2} \hspace{-30mm} &
			\\
			\leq &
			\left\|
			\frac{\check\gamma^N}{\left|\partial\Omega\right|}
			-
			\left(\left(A+\pi\right)^{-1}f\left(\tilde s_i^N\right)\right)_{1\leq i\leq N}
			\right\|_{\ell^2}
			\\
			& +
			\left\|
			\left(\left(A+\pi\right)^{-1}f\left(\tilde s_i^N\right)
			-\left(A+\pi\right)^{-1}f\left(s_i^N\right)\right)_{1\leq i\leq N}
			\right\|_{\ell^2}
			\\
			\leq &
			\frac{C}{N^{\kappa}}
			\left(
			\left\|f\right\|_{C^\kappa}
			+\left\|\left(A+\pi\right)^{-1}f\right\|_{C^{1}}\right)
			\leq
			\frac{C}{N^{\kappa}}
			\left\|f\right\|_{C^{\kappa}},
		\end{aligned}
	\end{equation}
	where we have used that $A$ is a regularizing operator.

	Now, note that
	\begin{equation*}
		\int_{0}^{\left|\partial\Omega\right|} \check g_{\rm app}^N \varphi
		=
		\frac1{N}\sum_{j=1}^N \check\gamma_{j}^N
		A^*\varphi\left(l\left(s_j^N\right)\right),
	\end{equation*}
	where we identify $\varphi(x)$ with $\varphi(s)$ whenever $x=l(s)\in\partial\Omega$. Then, employing \eqref{check gamma conv} and Corollary \ref{riemann2}, we deduce that
		\begin{align*}
			\left|\int_{0}^{\left|\partial\Omega\right|} \left(\check g_{\rm app}^N - A\left(A+\pi\right)^{-1}f \right)\varphi\right| \hspace{-40mm} &
			\\
			= & \left|\frac1{N}\sum_{j=1}^N \check\gamma_{j}^N
			A^*\varphi\left(l\left(s_j^N\right)\right)
			- \int_{0}^{\left|\partial\Omega\right|} \left(A+\pi\right)^{-1}f A^*\varphi\right|
			\\
			\leq &
			\left|
			\frac{\left|\partial\Omega\right|}{N}\sum_{j=1}^N
			\left(\frac{\check\gamma^N}{\left|\partial\Omega\right|}-\left(A+\pi\right)^{-1}f\left(s_j^N\right)\right)
			A^*\varphi\left(l\left(s_j^N\right)\right)
			\right|
			\\
			& +
			\left|
			\frac{\left|\partial\Omega\right|}{N}\sum_{j=1}^N \left(A+\pi\right)^{-1}f\left( s_j^N\right)
			A^*\varphi\left(l\left( s_j^N\right)\right)
			- \int_{0}^{\left|\partial\Omega\right|} \left(A+\pi\right)^{-1}f A^*\varphi\right|
			\\
			\leq &
			\frac{C}{N^{\kappa}}
			\left(
			\left\|f\right\|_{C^{\kappa}} \left\|A^*\varphi\right\|_{L^\infty}
			+
			\left\|\left(A+\pi\right)^{-1}f\right\|_{C^{\kappa}} \left\|A^*\varphi\right\|_{C^{\kappa}}
			\right)
			\\
			\leq &
			\frac{C}{N^{\kappa}}
			\left\|f\right\|_{C^{\kappa}} \left\|\varphi\right\|_{L^2},
		\end{align*}
	where we have used, again, in the last step above, that $A$ and $A^*$ are regularizing, which concludes the convergence estimate on $\check g_{\rm app}^N$.

	As for $\check f_{\rm app}^N$, observe that
	\begin{equation*}
		\int_{0}^{\left|\partial\Omega\right|} \check f_{\rm app}^N \varphi
		=
		\frac1{N}\sum_{j=1}^N \check\gamma_{j}^N
		B^*\varphi\left(l\left(s_j^N\right)\right).
	\end{equation*}
	Therefore, a similar estimate based on \eqref{check gamma conv} and Corollary \ref{riemann2} yields that
		\begin{align*}
			\left|\int_{0}^{\left|\partial\Omega\right|} \left(\check f_{\rm app}^N - B\left(A+\pi\right)^{-1}f \right)\varphi\right| \hspace{-40mm} &
			\\
			= & \left|\frac1{N}\sum_{j=1}^N \check\gamma_{j}^N
			B^*\varphi\left(l\left(s_j^N\right)\right)
			- \int_{0}^{\left|\partial\Omega\right|} \left(A+\pi\right)^{-1}f B^*\varphi\right|
			\\
			\leq &
			\left|
			\frac{\left|\partial\Omega\right|}{N}\sum_{j=1}^N
			\left(\frac{\check\gamma^N}{\left|\partial\Omega\right|}-\left(A+\pi\right)^{-1}f\left(s_j^N\right)\right)
			B^*\varphi\left(l\left(s_j^N\right)\right)
			\right|
			\\
			& +
			\left|
			\frac{\left|\partial\Omega\right|}{N}\sum_{j=1}^N \left(A+\pi\right)^{-1}f\left( s_j^N\right)
			B^*\varphi\left(l\left( s_j^N\right)\right)
			- \int_{0}^{\left|\partial\Omega\right|} \left(A+\pi\right)^{-1}f B^*\varphi\right|
			\\
			\leq &
			\frac{C}{N^{\kappa}}
			\left(
			\left\|f\right\|_{C^{\kappa}} \left\|B^*\varphi\right\|_{L^\infty}
			+
			\left\|\left(A+\pi\right)^{-1}f\right\|_{C^{\kappa}} \left\|B^*\varphi\right\|_{C^{\kappa}}
			\right)
			\\
			\leq &
			\frac{C}{N^{\kappa}}
			\left\|f\right\|_{C^{\kappa}} \left\|B^*\varphi\right\|_{C^{\kappa}}.
		\end{align*}
	This will complete the demonstration provided we show that $B^*\varphi$ is sufficiently regular, which is not obvious since $B^*$ has a singular kernel. We have already discussed, in Section \ref{vortex sheet}, the Plemelj-Privalov theorem, which guarantees that $B^*\varphi \in C^{0,\alpha}$ provided $\varphi\in C^{0,\alpha}$, for any given $0<\alpha<1$. This, however, is not sufficient for our purpose and we need now to establish some higher regularity estimate on $B^*\varphi$. To this end, employing \eqref{B1}, we first write
		\begin{align*}
			B^*\varphi(s)
			= & \int_0^{\left|\partial\Omega\right|}\frac{l(s)-l(s_*)}{\left|l(s)-l(s_*)\right|^2}\cdot \tau(l(s_*))\left(\varphi(s)-\varphi(s_*)\right)ds_*
			\\
			= & \int_0^{\left|\partial\Omega\right|}\int_0^1\left|\frac{l(s)-l(s_*)}{s-s_*}\right|^{-2}\frac{l(s)-l(s_*)}{s-s_*}
			\cdot \tau(l(s_*))\varphi'(s_*+t(s-s_*))dtds_*.
		\end{align*}
	Then, upon noticing that $\frac{l(s)-l(s_*)}{s-s_*}$ and $\tau(l(s_*))$ are smooth and that $\frac{l(s)-l(s_*)}{s-s_*}$ remains bounded away from zero, this representation of $B^*\varphi$ easily yields that
	\begin{equation*}
		\left\|B^*\varphi\right\|_{C^{\kappa}}
		\leq C
		\left\|\varphi'\right\|_{C^{\kappa}}
		\leq C
		\left\|\varphi\right\|_{C^{\kappa+1}},
	\end{equation*}
	which ends the proof.
\end{proof}

We are now in a position to give a complete justification of Theorem \ref{main theo alt}.

\begin{proof}[Proof of Theorem \ref{main theo alt}]
	We follow here the method of proof of Theorem \ref{main theo} presented in Section \ref{proof of main}.

	First, for given $\omega\in C_c^{0,\alpha}$ and $\gamma\in\mathbb{R}$, recall that the full plane flow $u_{P}\in C^1\left(\overline \Omega\right)$ is obtained from \eqref{uP} and that the $\left|\partial\Omega\right|$-periodic function $f\in C^\infty\left([0,\left|\partial\Omega\right|]\right)$ is defined by $f(s)=-2\pi[\left(u_{P}+\gamma H_*\right)\cdot n]\left(l(s)\right)$, for all $s\in\left[0,\left|\partial\Omega\right|\right]$. Therefore, with this given $f$, according to Proposition \ref{inverse D alt}, we find a unique solution $\check\gamma^N\in\mathbb{R}^N$ of \eqref{point toy alt}.

	Next, the approximate flow $\check u_{\rm app}^N$ is introduced by \eqref{approx alt} and verifies
		\begin{align*}
			\check u_{\rm app}^N(x)\cdot \tau(x) & = \frac1{2\pi} \check f_{\rm app}^N(s) + \gamma H_*\cdot\tau,
			\\
			\check u_{\rm app}^N(x)\cdot n(x) & = \frac1{2\pi} \check g_{\rm app}^N(s) + \gamma H_*\cdot n,
		\end{align*}
	where $x=l(s)\in\partial\Omega$ and $\check f_{\rm app}^N$, $\check g_{\rm app}^N$ are defined by \eqref{f app alt}. Recall that the values of $H_*$ on $\partial\Omega$ are determined by its limiting values from within $\Omega$. Utilizing identity \eqref{vortex identity} to rewrite the discrete kernel of $\check u_{\rm app}^N$, we find that
		\begin{align*}
			\check u_{\rm app}^N(x)
			= & \frac{-1}{2\pi^2} \sum_{j=1}^N \frac{\check\gamma_{j}^N}N
			\\
			& \times \left(
			\int_{\partial\Omega}\frac{x_{j}^N-z}{\left|x_{j}^N-z\right|^2} \cdot \tau(z) \frac{(x-z)^\perp}{|x-z|^2}dz
			+
			\int_{\partial\Omega}\frac{x_{j}^N-z}{\left|x_{j}^N-z\right|^2}\cdot n(z)\frac{x-z}{|x-z|^2}
			dz
			\right)
			\\
			& + \gamma H_*(x)
			\\
			= &
			\frac1{2\pi^2}
			\int_0^{\left|\partial\Omega\right|}\check f_{\rm app}^N(s_*) \frac{\left(x-l(s_*)\right)^\perp}{\left|x-l(s_*)\right|^2}ds_*
			\\
			& +\frac1{2\pi^2}
			\int_0^{\left|\partial\Omega\right|}\check g_{\rm app}^N(s_*)\frac{x-l(s_*)}{|x-l(s_*)|^2}
			ds_*
			+\gamma H_*(x), \quad\text{on }\Omega.
		\end{align*}

	Furthermore, recall that, according to \eqref{vortex identity 3}, the remainder flow $u_R$ can be expressed as
		\begin{align*}
			u_R(x)
			= & \frac 1{2\pi^2}\int_{\partial\Omega} \frac{(x-y)^\perp}{|x-y|^2}
			B\left(A+\pi\right)^{-1}f(y)dy
			\\
			& +
			\frac 1{2\pi^2}\int_{\partial\Omega} \frac{x-y}{|x-y|^2}
			A\left(A+\pi\right)^{-1}f(y)dy + \gamma H_*(x),\quad\text{on }\Omega,
		\end{align*}
	whereby
		\begin{align*}
			\left(\check u_{\rm app}^N-u_R\right)(x)
			= & \frac 1{2\pi^2}\int_{\partial \Omega} \frac{(x-y)^\perp}{|x-y|^2} \left(\check f^N_{\rm app}-B\left(A+\pi\right)^{-1} f\right)(y)dy
			\\
			& +
			\frac 1{2\pi^2}\int_{\partial \Omega} \frac{x-y}{|x-y|^2} \left(\check g^N_{\rm app} - A\left(A+\pi\right)^{-1}f\right)(y)dy, \quad\text{on }\Omega.
		\end{align*}

	Therefore, in view of Proposition \ref{prop 32 alt}, we deduce that, for any fixed $x\in\Omega$,
		\begin{align*}
			\left|\left(\check u_{\rm app}^N-u_R\right)(x)\right|
			& \leq \frac{C}{N^\kappa}
			\left\|f\right\|_{C^{\kappa}}
			\left\|\frac{x-y}{|x-y|^2}\right\|_{C_y^{\kappa+1}}
			\\
			& \leq \frac{C}{N^\kappa}
			\left\|f\right\|_{C^{\kappa}}
			\sup_{y\in\partial \Omega}\left(\frac{1}{|x-y|}+\frac{1}{|x-y|^{\kappa+2}}\right),
		\end{align*}
	where the constant $C>0$ is independent of $x$, $\omega$ and $\gamma$. Since the support of $\omega$ is bounded away from $\partial\Omega$, it holds that
		\begin{align*}
			\left\|f\right\|_{C^{\kappa}}
			&
			\leq C \sup_{x\in\partial\Omega}
			\int_{\mathbb{R}^2}
			\left(\frac{1}{|x-y|}+\frac{1}{|x-y|^{\kappa+1}}\right)\left|\omega(y)\right|dy
			+C|\gamma|\left\|H_*\cdot n\right\|_{C^{\kappa}}
			\\
			&
			\leq \frac C{\operatorname{dist}\left(\operatorname{supp}\omega,\partial\Omega\right)}\left(1 +\frac 1{\operatorname{dist}\left(\operatorname{supp}\omega,\partial\Omega\right)^\kappa}\right)
			\left\|\omega\right\|_{L^1\left(\mathbb{R}^2\right)}+C|\gamma|\left\|H_*\cdot n\right\|_{C^{\kappa}}.
	\end{align*}
	It follows that, for any closed set $K\subset\Omega$,
\begin{equation*}
		\begin{aligned}
			\| u_{R} - \check u_{\rm app}^N \|_{L^\infty(K)}
			\hspace{-5mm}&
			\\
			\leq & \frac{C}{N^\kappa}
			\left(
			\frac 1{\operatorname{dist}\left(K,\partial\Omega\right)} + \frac 1{\operatorname{dist}\left(K,\partial\Omega\right)^{\kappa+2}}
			\right)
			\\
			& \times \left(\left(
			\frac 1{\operatorname{dist}\left(\operatorname{supp}\omega,\partial\Omega\right)} + \frac 1{\operatorname{dist}\left(\operatorname{supp}\omega,\partial\Omega\right)^{\kappa+1}}
			\right)\left\|\omega\right\|_{L^1\left(\mathbb{R}^2\right)} + |\gamma|\right),
		\end{aligned}
		\end{equation*}
	which, extending the above estimate to all $\omega\in L^1_c\left(\Omega\right)$ by a standard density argument, concludes the proof of the theorem.
\end{proof}

\subsection{A further refinement of the fluid charge method}

With the purpose of improving the accuracy and efficiency of potential numerical methods based on the preceding theorems, we develop now a refinement of Theorem \ref{main theo alt} which is simply based upon noticing that, by \eqref{mean}, the subspace $L^2_0\left(\partial\Omega\right)\subset L^2\left(\partial\Omega\right)$ is invariant under the action of $A+\pi$. In particular, defining the averaging operator over $\partial\Omega$ by
\begin{equation*}
	\ip{h}=
	\int_{\partial\Omega}h(y)dy,
\end{equation*}
for any $h\in L^1\left(\partial\Omega\right)$, since the operators $A+\pi$ and $A-\lambda\ip{\cdot}+\pi$, for any given smooth function $\lambda(x)$ on $\partial\Omega$, coincide on $L^2_0\left(\partial\Omega\right)$, it clearly holds that $A-\lambda\ip{\cdot}+\pi:L^2_0\left(\partial\Omega\right)\to L^2_0\left(\partial\Omega\right)$ has a bounded inverse given by
\begin{equation*}
	\left(A-\lambda\ip{\cdot}+\pi\right)^{-1}
	=\left(A+\pi\right)^{-1}:L^2_0\left(\partial\Omega\right)\to L^2_0\left(\partial\Omega\right).
\end{equation*}
It follows that \eqref{boundary sheet omega 4} may be recast, for any smooth $\lambda(x)$, as
\begin{equation}\label{boundary sheet omega 3}
	u_R(x) = -\int_{\partial\Omega} \frac{x-y}{|x-y|^2}\left(A-\lambda\ip{\cdot}+\pi\right)^{-1}\left[\left(u_P+\gamma H_*\right)\cdot n\right](y)dy
	+ \gamma H_*(x),
\end{equation}
for $\left(u_P+\gamma H_*\right)\cdot n$ in \eqref{boundary sheet omega 4} has mean value zero (recall that $u_P$ is divergence free in $\mathbb{R}^2$).

In fact, if $\ip{\lambda}\neq 2\pi$, then it is straightforward to verify that $A-\lambda\ip{\cdot}+\pi:L^2\left(\partial\Omega\right)\to L^2\left(\partial\Omega\right)$ also has a bounded inverse given by
\begin{equation}\label{bounded inverse}
	\left(A-\lambda\ip{\cdot}+\pi\right)^{-1}h
	=\left(A+\pi\right)^{-1}\left[h+\frac{\ip{h}}{2\pi-\ip{\lambda}}\lambda\right],
\end{equation}
for any $h\in L^2\left(\partial\Omega\right)$.

Aiming at discretizing \eqref{boundary sheet omega 3} rather than \eqref{boundary sheet omega 4}, we propose now to build an approximate flow
\begin{equation}\label{approx alt alt}
	\check{\check u}_{\rm app}^N(x):=\frac1{2\pi} \sum_{j=1}^N \frac{\check{\check\gamma}_{j}^N}N \frac{x - x_{j}^N}{|x - x_{j}^N|^2} + \gamma H_*(x),
\end{equation}
where the density $\check{\check\gamma}^N=\left(\check{\check\gamma}_{1}^N,\dots, \check{\check\gamma}_{N}^N\right)\in\mathbb{R}^N$ solves the following system of $N$ linear equations, for any prescribed smooth function $\lambda(x)$ such that $\ip{\lambda}\neq 2\pi$:
\begin{multline}\label{point vortex alt alt}
		\left(\frac1{2\pi}\sum_{j=1}^N
		\frac{\check{\check\gamma}_{j}^N}N
		\left(\frac{\tilde x_{i}^N - x_{j}^N}{|\tilde x_{i}^N - x_{j}^N|^2}\cdot n(\tilde x_{i}^N)-\lambda\left(\tilde x_i^N\right)\right)\right)
		+\frac{\check{\check\gamma}_i^N}{2\left|\partial\Omega\right|} 
		\\
		= - [\left(u_{P}+\gamma H_*\right)\cdot n] (\tilde x_{i}^N) , \ \text{for all }i=1,\dots, N,
\end{multline}
As usual, we may use the notation $\check{\check u}_{\rm app}^N=\check{\check u}_{\rm app}^N[\omega,\gamma]$ to emphasize the linear dependence of $\check{\check u}_{\rm app}^N$ in $(\omega,\gamma)$.

Equivalently, this system may be recast as
\begin{equation}\label{point toy alt alt}
	\left(\frac 1N A_N-\lambda^N\ip{\cdot}+\frac{\pi}{\left|\partial\Omega\right|}\right)\check{\check\gamma}^N
	= \left(f\left(\tilde s_i^N\right)\right)_{1\leq i\leq N},
\end{equation}
where $f(s)=-2\pi[\left(u_{P}+\gamma H_*\right)\cdot n]\left(l(s)\right)$, for all $s\in\left[0,\left|\partial\Omega\right|\right]$, and the vector $\lambda^N=\left(\lambda^N_1,\ldots,\lambda^N_N\right)\in\mathbb{R}^N$ is defined by $\lambda^N_i=\lambda\left(\tilde x_i^N\right)$, for all $i=1,\ldots,N$.

We arrive now at the following main theorem concerning the convergence of the approximate flow $\check{\check u}_{\rm app}^N$ defined in \eqref{approx alt alt}.

\begin{theorem}\label{main theo alt alt}
	Let $\omega\in L^1_c\left(\Omega\right)$, $\gamma\in\mathbb{R}$ and a smooth function $\lambda$ such that $\ip{\lambda}\neq 2\pi$ be given. For any $N\geq 2$, we consider a well-$*$ distributed mesh satisfying \eqref{mesh3} and $u_{P}$ defined in \eqref{uP}.
	
	Then, there exists $N_0$ (independent of $\omega$ and $\gamma$) such that, for all $N\geq N_0$, the system \eqref{point vortex alt alt} admits a unique solution $\check{\check \gamma}^N\in \R^N$. Moreover, for any closed set $K\subset \Omega$  there exists a constant $C>0$ independent of $N$, $K$, $\omega$ and $\gamma$ such that
		\begin{align*}
			\| u_{R} - \check{\check u}_{\rm app}^N \|_{L^\infty(K)}
			\hspace{-5mm}&
			\\
			\leq & \frac{C}{N^\kappa}
			\left(
			\frac 1{\operatorname{dist}\left(K,\partial\Omega\right)} + \frac 1{\operatorname{dist}\left(K,\partial\Omega\right)^{\kappa+2}}
			\right)
			\\
			& \times \left(\left(
			\frac 1{\operatorname{dist}\left(\operatorname{supp}\omega,\partial\Omega\right)} + \frac 1{\operatorname{dist}\left(\operatorname{supp}\omega,\partial\Omega\right)^{\kappa+1}}
			\right)\left\|\omega\right\|_{L^1\left(\mathbb{R}^2\right)} + |\gamma|\right),
		\end{align*}
	where $\check{\check u}_{\rm app}^N$ is given by \eqref{approx alt alt} in terms of $\check{\check \gamma}^N$ and $u_{R}$ is the continuous flow \eqref{uR}.
\end{theorem}

The justification of the above theorem is based on Propositions \ref{inverse D alt alt} and \ref{prop 32 alt alt}, which are established below. Using these results, the proof of the above theorem follows the same steps {\it mutatis mutandis} as the proof of Theorem \ref{main theo alt} and, so, we leave it to the reader.

The next result is a simple extension of Proposition \ref{inverse D alt} and allows us to solve \eqref{point toy alt alt}.

\begin{proposition}\label{inverse D alt alt}
	For any integer $N\geq 2$, consider a well-$*$ distributed mesh $(s_{1}^N,\dots , s_{N}^N)\in \mathbb{R}^N$, $(\tilde s_{1}^N, \dots , \tilde s_{N}^N)\in \mathbb{R}^N$ and let $\lambda$ be a smooth function such that $\ip{\lambda}\neq 2\pi$.
	
	Then, provided $N$ is sufficiently large, the operator $\frac 1N A_N-\lambda^N\ip{\cdot}+\frac{\pi}{\left|\partial\Omega\right|}\in \mathcal{L}\left(\ell^2\right)$ is invertible. Furthermore, its inverse operator is bounded in $\mathcal{L}\left(\ell^2\right)$ uniformly in $N$. It follows that, provided $N$ is sufficiently large, the following problem:
	\begin{equation*}
		z\in \R^N,\quad
		\left(\frac{1}N A_{N}
		-\lambda^N\ip{\cdot}+\frac{\pi}{\left|\partial\Omega\right|}\right)z=v,
	\end{equation*}
	has a unique solution for any given $v\in \R^{N}$. Moreover, this solution satisfies:
	\begin{equation*}
	\| z \|_{\ell^1}\leq \| z \|_{\ell^2}
	\leq
	C\| v \|_{\ell^2}
	\leq
	C\| v \|_{\ell^\infty},
	\end{equation*}
	for some independent constant $C>0$.
\end{proposition}

\begin{proof}
	In accordance with Proposition \ref{inverse D alt}, the operator $\left(\frac{1}N A_{N}+\frac{\pi}{\left|\partial\Omega\right|}\right)$ is invertible provided $N$ is large. Therefore, any solution $z\in\mathbb{R}^N$ of the above system has to satisfy
	\begin{equation*}
		z-\ip{z}
		\left(\frac{1}N A_{N}
		+\frac{\pi}{\left|\partial\Omega\right|}\right)^{-1}\lambda^N=\left(\frac{1}N A_{N}
		+\frac{\pi}{\left|\partial\Omega\right|}\right)^{-1}v,
	\end{equation*}
	whence
	\begin{equation}\label{mean identity}
		\ip{z}\left(1-
		\ip{\left(\frac{1}N A_{N}
		+\frac{\pi}{\left|\partial\Omega\right|}\right)^{-1}\lambda^N}\right)
		=
		\ip{\left(\frac{1}N A_{N}
		+\frac{\pi}{\left|\partial\Omega\right|}\right)^{-1}v}.
	\end{equation}
	It follows that, provided $\ip{\left(\frac{1}N A_{N}+\frac{\pi}{\left|\partial\Omega\right|}\right)^{-1}\lambda^N}\neq 1$, any solution has to be given by the formula
	\begin{equation*}
		z
		=
		\left(\frac{1}N A_{N}
		+\frac{\pi}{\left|\partial\Omega\right|}\right)^{-1}
		\left(v
		+\frac{\ip{\left(\frac{1}N A_{N}
		+\frac{\pi}{\left|\partial\Omega\right|}\right)^{-1}v}
		}{
		1-
		\ip{\left(\frac{1}N A_{N}
		+\frac{\pi}{\left|\partial\Omega\right|}\right)^{-1}\lambda^N}
		}
		\lambda^N
		\right),
	\end{equation*}
	which is a discrete analog to \eqref{bounded inverse}. It is straightforward to check that, for any given $v\in\mathbb{R}^N$, the above identity also provides a viable solution $z\in\mathbb{R}^N$.
	
	Therefore, there only remains to verify that $\ip{\left(\frac{1}N A_{N}+\frac{\pi}{\left|\partial\Omega\right|}\right)^{-1}\lambda^N}\neq 1$. To this end, observe first that $\left(\frac{\left|\partial\Omega\right|}{N}A_N^*\mathbf{1}\right)_j$, with $\mathbf{1}=\left(1,\ldots,1\right)\in\mathbb{R}^N$, for each $j=1,\ldots,N$, is a discretization of the integral $\left(A^*1\right)\left(l\left(s_{j}^N\right)\right)$. Since $A^*1\equiv\pi$, by \eqref{A1}, we conclude, by the uniform convergence of Riemann sums for smooth functions (see Corollary \ref{riemann2}), that
	\begin{equation*}
		\left\|\frac{\left|\partial\Omega\right|}{N}A_N^*\mathbf{1}-\pi\mathbf{1}\right\|_{\ell^\infty}=\mathcal{O}\left(N^{-\kappa}\right).
	\end{equation*}
	This elementary estimate is similar to \eqref{average adjoint}, the difference being that, here, the mesh is not well distributed and the control does not concern $B_N^*$. By duality, it follows that, for all $z\in\mathbb{R}^N$,
	\begin{equation}\label{average estimate}
		\left|\ip{\left(\frac{1}{N}A_N-\frac{\pi}{\left|\partial\Omega\right|}\right) z}\right|
		\leq \frac C{N^\kappa}\left\|z\right\|_{\ell^1},
	\end{equation}
	for some uniform constant $C>0$. Then, we deduce
		\begin{align*}
			\left|\ip{\lambda^N}-\frac{2\pi}{\left|\partial\Omega\right|}\ip{\left(\frac{1}N A_{N}+\frac{\pi}{\left|\partial\Omega\right|}\right)^{-1}\lambda^N}\right|
			\hspace{-10mm}&
			\\
			= &
			\left|\ip{\left(\frac{1}{N}A_N-\frac{\pi}{\left|\partial\Omega\right|}\right) \left(\frac{1}N A_{N}+\frac{\pi}{\left|\partial\Omega\right|}\right)^{-1}\lambda^N}\right|
			\\
			\leq & \frac C{N^\kappa}\left\|\left(\frac{1}N A_{N}+\frac{\pi}{\left|\partial\Omega\right|}\right)^{-1}\lambda^N	\right\|_{\ell^1},
		\end{align*}
	which, since $\left|\partial\Omega\right|\ip{\lambda^N}=\ip{\lambda}+o\left(1\right)$ by the convergence of Riemann sums for continuous functions, and by the uniform boundedness of $\left(\frac{1}N A_{N}+\frac{\pi}{\left|\partial\Omega\right|}\right)^{-1}\in\mathcal{L}\left(\ell^2\right)$ asserted in Proposition \ref{inverse D alt}, implies that
	\begin{equation}\label{lambda convergence}
		\left|\ip{\left(\frac{1}N A_{N}+\frac{\pi}{\left|\partial\Omega\right|}\right)^{-1}\lambda^N}
		-\frac{\ip{\lambda}}{2\pi}
		\right|
		=o(1).
	\end{equation}
	Recalling that $\ip{\lambda}\neq 2\pi$, we finally deduce that $\ip{\left(\frac{1}N A_{N}+\frac{\pi}{\left|\partial\Omega\right|}\right)^{-1}\lambda^N}\neq 1$ for large values of $N$.
	
	Since $\ip{\left(\frac{1}N A_{N}+\frac{\pi}{\left|\partial\Omega\right|}\right)^{-1}\lambda^N}$ is, in fact, uniformly bounded away from $1$, as $N\to\infty$, the uniform boundedness in $N$ of the inverse operator easily ensues, which concludes the proof of the proposition.
\end{proof}

It is also possible to obtain an extension of Proposition \ref{prop 32 alt} concerning the weak convergence of the discretization of the operator $A-\lambda\ip{\cdot}+\pi$.

\begin{proposition}\label{prop 32 alt alt}
	For any integer $N\geq 2$, consider a well-$*$ distributed mesh $(s_{1}^N,\dots , s_{N}^N)\in \mathbb{R}^N$, $(\tilde s_{1}^N, \dots , \tilde s_{N}^N)\in \mathbb{R}^N$ satisfying \eqref{mesh3} and, according to Proposition \ref{inverse D alt alt}, consider the solution $\check{\check\gamma}^N=(\check{\check\gamma}_{1}^N,\dots, \check{\check\gamma}_{N}^N)\in\mathbb{R}^N$ to the system \eqref{point toy alt alt} for some periodic function $f\in C^{\kappa}\left(\left[0,\left|\partial\Omega\right|\right]\right)$, where $\kappa \geq 2$, with zero mean value $\int_0^{\left|\partial\Omega\right|} f(s)ds=0$ and some smooth function $\lambda$ such that $\ip{\lambda}\neq 2\pi$. We define the approximations
		\begin{align*}
			\check{\check f}_{\rm app}^N(s) & :=
			\frac1{N}\sum_{j=1}^N \check{\check\gamma}_{j}^N \frac{l\left(s\right) - l\left(s_{j}^N\right)}
			{\left|l\left(s\right) - l\left(s_{j}^N\right)\right|^2}\cdot \tau\left(l\left(s\right)\right),
			\\
			\check{\check g}_{\rm app}^N(s) & :=
			\frac1{N}\sum_{j=1}^N \check{\check\gamma}_{j}^N \frac{l\left(s\right) - l\left(s_{j}^N\right)}
			{\left|l\left(s\right) - l\left(s_{j}^N\right)\right|^2}\cdot n\left(l\left(s\right)\right).
		\end{align*}

	Then, for any periodic test function $\varphi\in C^{\infty}\left(\left[0,\left|\partial\Omega\right|\right]\right)$,
		\begin{align*}
			\left|\int_{0}^{\left|\partial\Omega\right|} \left(\check{\check f}_{\rm app}^N - B\left(A+\pi\right)^{-1}f \right)\varphi\right|
			& \leq
			\frac{C}{N^{\kappa}}
			\left\|f\right\|_{C^{\kappa}}
			\left\|\varphi\right\|_{C^{\kappa+1}},
			\\
			\left|\int_{0}^{\left|\partial\Omega\right|} \left(\check{\check g}_{\rm app}^N - A\left(A+\pi\right)^{-1}f \right)\varphi\right|
			& \leq
			\frac{C}{N^{\kappa}}
			\left\|f\right\|_{C^{\kappa}}
			\left\|\varphi\right\|_{L^2},
		\end{align*}
	where we identify the variable $x$ with the variable $s$ whenever $x=l(s)\in\partial\Omega$ and the singular integrals are defined in the sense of Cauchy's principal value.
\end{proposition}

\begin{proof}
	First, we estimate, by \eqref{average estimate} and by the uniform boundedness of the operator $\left(\frac{\left|\partial\Omega\right|}{N} A_N+\pi\right)^{-1}$ in $\mathcal{L}\left(\ell^2\right)$ established in Proposition \ref{inverse D alt},
		\begin{align*}
			\left|
			\ip{\left(\frac{1}N A_{N}
			+\frac{\pi}{\left|\partial\Omega\right|}\right)^{-1}\left(f\left(\tilde s_i^N\right)\right)_{1\leq i\leq N}}
			-\frac{\left|\partial\Omega\right|}{2\pi}\ip{\left(f\left(\tilde s_i^N\right)\right)_{1\leq i\leq N}}\right|
			\hspace{-80mm} &
			\\
			= &
			\left|
			\ip{
			\frac{\left|\partial\Omega\right|}{2\pi}
			\left(\frac{1}N A_{N}
			-\frac{\pi}{\left|\partial\Omega\right|}\right)
			\left(\frac{1}N A_{N}
			+\frac{\pi}{\left|\partial\Omega\right|}\right)^{-1}\left(f\left(\tilde s_i^N\right)\right)_{1\leq i\leq N}}
			\right|
			\\
			\leq & \frac{C}{N^\kappa}
			\left\|
			\left(\frac{1}N A_{N}
			+\frac{\pi}{\left|\partial\Omega\right|}\right)^{-1}\left(f\left(\tilde s_i^N\right)\right)_{1\leq i\leq N}
			\right\|_{\ell^1}
			\\
			\leq & \frac{C}{N^\kappa}
			\left\|
			\left(f\left(\tilde s_i^N\right)\right)_{1\leq i\leq N}
			\right\|_{\ell^2}
			\leq
			\frac{C}{N^\kappa}\left\|f\right\|_{L^\infty},
		\end{align*}
	whence, since $f$ has zero mean value over $\partial\Omega$, by the convergence of Riemann sums for smooth functions (see Corollary \ref{riemann2}),
	\begin{equation*}
		\left|
		\ip{\left(\frac{1}N A_{N}
		+\frac{\pi}{\left|\partial\Omega\right|}\right)^{-1}\left(f\left(\tilde s_i^N\right)\right)_{1\leq i\leq N}}
		\right|
		\leq\frac{C}{N^{\kappa}}\left\|f\right\|_{C^{\kappa}}.
	\end{equation*}
	Further using \eqref{mean identity} and recalling from \eqref{lambda convergence} that $\ip{\left(\frac{1}N A_{N}+\frac{\pi}{\left|\partial\Omega\right|}\right)^{-1}\lambda^N}$ remains bounded away from $1$, as $N\to\infty$, we conclude that
	\begin{equation}\label{mean alt alt}
		\left|
		\ip{\check{\check\gamma}^N}
		\right|
		\leq\frac{C}{N^{\kappa}}\left\|f\right\|_{C^{\kappa}}.
	\end{equation}
	
	Next, according to \eqref{point toy alt alt}, it holds that
	\begin{equation*}
		\left(\frac 1N A_N+\frac{\pi}{\left|\partial\Omega\right|}\right)\check{\check\gamma}^N
		= \left(h^N\left(\tilde s_i^N\right)\right)_{1\leq i\leq N},
	\end{equation*}
	where the periodic function $h^N\in C^{\kappa}\left(\left[0,\left|\partial\Omega\right|\right]\right)$, for each $N$, is defined by $h^N(s)=f(s)+\lambda\left(l(s)\right)\ip{\check{\check\gamma}^N}$, for all $s\in\left[0,\left|\partial\Omega\right|\right]$. Therefore, by Proposition \ref{prop 32 alt}, we infer that
		\begin{align*}
			\left|\int_{0}^{\left|\partial\Omega\right|} \left(\check{\check f}_{\rm app}^N - B\left(A+\pi\right)^{-1}h^N \right)\varphi\right|
			& \leq
			\frac{C}{N^{\kappa}}
			\left\|h^N\right\|_{C^{\kappa}}
			\left\|\varphi\right\|_{C^{\kappa+1}},
			\\
			\left|\int_{0}^{\left|\partial\Omega\right|} \left(\check{\check g}_{\rm app}^N - A\left(A+\pi\right)^{-1}h^N \right)\varphi\right|
			& \leq
			\frac{C}{N^{\kappa}}
			\left\|h^N\right\|_{C^{\kappa}}
			\left\|\varphi\right\|_{L^2},
		\end{align*}
	which, by boundedness of $A$, $B$ and $(A+\pi)^{-1}$ over $L^2\left(\partial\Omega\right)$, implies that
		\begin{align*}
			\left|\int_{0}^{\left|\partial\Omega\right|} \left(\check{\check f}_{\rm app}^N - B\left(A+\pi\right)^{-1}f \right)\varphi\right|
			& \leq C
			\left(\frac{1}{N^{\kappa}}
			\left\|f\right\|_{C^{\kappa}}
			+\left|\ip{\check{\check\gamma}^N}\right|
			\right)
			\left\|\varphi\right\|_{C^{\kappa+1}},
			\\
			\left|\int_{0}^{\left|\partial\Omega\right|} \left(\check{\check g}_{\rm app}^N - A\left(A+\pi\right)^{-1}f \right)\varphi\right|
			& \leq C
			\left(\frac{1}{N^{\kappa}}
			\left\|f\right\|_{C^{\kappa}}
			+\left|\ip{\check{\check\gamma}^N}\right|
			\right)
			\left\|\varphi\right\|_{L^2}.
		\end{align*}
	Finally, incorporating \eqref{mean alt alt} into the preceding estimate completes the proof of the proposition.
\end{proof}

\subsection{Good conditioning of discretized systems}\label{diagonal}

Theorem \ref{main theo alt} provides an alternative method for building approximate flows which may, in many cases, yield efficient numerical methods outperforming the corresponding methods based on Theorem \ref{main theo}. Indeed, Theorem \ref{main theo alt} requires the resolution of systems given by the matrices $\frac 1N A_N+\frac{\pi}{\left|\partial\Omega\right|}$ (for a well-$*$ distributed mesh). The fact that the coefficients of $\frac 1N A_N+\frac{\pi}{\left|\partial\Omega\right|}$ on its diagonal are of order $\mathcal{O}(1)$ and, thus, dominate those off the diagonal, which are of order $\mathcal{O}\left(N^{-1}\right)$, guarantees good conditioning properties, which allow to solve the corresponding systems with good numerical accuracy.

More precisely, supposing for instance that $\partial\Omega$ is strictly convex so that the kernel of the operator $A$, defined in \eqref{AB}, satisfies $\frac{x-y}{|x-y|^2}\cdot n(x)\geq C_0 > 0$, for all $x,y\in\partial\Omega$ and for some $C_0>0$, we see that, for each given $j=1,\ldots, N$, according to \eqref{A1} and Corollary \ref{riemann2}, for sufficiently large $N$,
\begin{equation}\label{strict convex}
	\begin{aligned}
		\sum_{\substack{1\leq i\leq N \\ i\neq j}}\Bigg|\Bigg(\frac{\left|\partial\Omega\right|} N  &A_N+\pi\Bigg)_{ij}\Bigg|
		= 
		\frac{\left|\partial\Omega\right|}N \sum_{\substack{1\leq i\leq N \\ i\neq j}}\left(A_N\right)_{ij}
		\\
		= & \int_{\partial\Omega} \frac{x-x_j^N}{\left|x-x_j^N\right|^2}\cdot n(x) dx +\mathcal{O}\left(N^{-2}\right)
		-\frac{\left|\partial\Omega\right|}N \left(A_N\right)_{jj}
		\\
		\leq & \int_{\partial\Omega} \frac{x-x_j^N}{\left|x-x_j^N\right|^2}\cdot n(x) dx +
		\underbrace{\mathcal{O}\left(N^{-2}\right)
		-\frac{C_0\left|\partial\Omega\right|}N}_{<0 \text{ for large }N}
		\\
		< & \int_{\partial\Omega} \frac{x-x_j^N}{\left|x-x_j^N\right|^2}\cdot n(x) dx
		=\pi < \left|\left(\frac{\left|\partial\Omega\right|}N A_N+\pi\right)_{jj}\right|.
	\end{aligned}
\end{equation}
In other words, whenever $\partial\Omega$ is striclty convex, the matrix $\frac 1N A_N+\frac{\pi}{\left|\partial\Omega\right|}$ is strictly diagonally dominant with respect to columns, which opens the door to efficient and accurate numerical resolution methods for the corresponding systems. In particular, the $LU$ decomposition exists and no pivoting is necessary in Gaussian elimination (see \cite[Section 4.1]{golub}).

Such diagonal dominance properties never hold for the methods based on Theorem \ref{main theo}, which require the resolution of large systems whose coefficients stem from the nonintegrable kernel of the singular operator $B$ defined in \eqref{AB} and are therefore prone to large numerical errors.

Like Theorem \ref{main theo alt}, Theorem \ref{main theo alt alt} provides an alternative method for building approximate flows. In fact, Theorem \ref{main theo alt alt} is more general and reduces to Theorem \ref{main theo alt} by setting $\lambda\equiv 0$ therein. Numerically, the extra degree of freedom provided by the parameter $\lambda$ is significant, for it may lead, in numerous case (depending on the geometry of $\partial\Omega$), to large linear systems whose coefficient matrices can be better conditioned with an appropriate choice of $\lambda$. As previously mentioned, the case $\lambda\equiv 0$ is sufficient to produce well conditioned systems in strictly convex geometries (recall that $\frac 1N A_N+\frac{\pi}{\left|\partial\Omega\right|}$ is strictly diagonally dominant with respect to columns whenever $\partial\Omega$ is strictly convex; see \eqref{strict convex}). Now, we are also able to handle some non-convex geometries by appropriately setting $\lambda\not\equiv 0$.

Indeed, let us suppose, for instance, that the geometry of $\partial\Omega$ is such that the following analytical condition is satisfied:
\begin{equation}\label{non convex condition}
	\sup_{y\in\partial\Omega}\int_{\partial\Omega} \left| \frac{x-y}{\left|x-y\right|^2}\cdot n(x)-\lambda(x) \right| dx
	<\pi,
\end{equation}
for some smooth $\lambda$. Observe that, by \eqref{A1}, it necessarily holds that $0<\ip{\lambda}<2\pi$. Then, we see, for each given $j=1,\ldots, N$, according to \eqref{A1} and Corollary \ref{riemann2}, for sufficiently large $N$, that
	\begin{align*}
		\sum_{\substack{1\leq i\leq N \\ i\neq j}}
		\left|\left(\frac{\left|\partial\Omega\right|}N A_N
		-\left|\partial\Omega\right|\lambda^N\ip{\cdot}+\pi\right)_{ij}\right|
		\hspace{-40mm} &
		\\
		= &
		\frac{\left|\partial\Omega\right|}N \sum_{\substack{1\leq i\leq N \\ i\neq j}}
		\left|\left(A_N\right)_{ij}-\lambda_i^N\right|
		\\
		= & \int_{\partial\Omega} \left| \frac{x-x_j^N}{\left|x-x_j^N\right|^2}\cdot n(x)-\lambda(x) \right| dx +\mathcal{O}\left(N^{-1}\right)
		- \frac{\left|\partial\Omega\right|}N
		\left|\left(A_N\right)_{jj}-\lambda_j^N\right|
		\\
		< & \pi - \frac{\left|\partial\Omega\right|}N
		\left|\left(A_N\right)_{jj}-\lambda_j^N\right|
		\leq \left|\frac{\left|\partial\Omega\right|}N\left(\left( A_N\right)_{jj}-\lambda_j^N\right)+\pi\right|
		\\
		= & \left|\left(\frac{\left|\partial\Omega\right|}N A_N
		-\left|\partial\Omega\right|\lambda^N\ip{\cdot}+\pi\right)_{jj}\right|.
	\end{align*}
In other words, the matrix $\frac 1N A_N-\lambda^N\ip{\cdot}+\frac{\pi}{\left|\partial\Omega\right|}$ is strictly diagonally dominant with respect to columns, for large $N$, as soon as \eqref{non convex condition} is satisfied. Again, we insist on the fact that this property leads to well conditioned systems and thus a significant potential improvement of the corresponding numerical resolution.

\subsection{Geometric interpretation of \eqref{non convex condition}}

For simplicity, we first consider the case where $\lambda(x)$ in \eqref{non convex condition} is identically equal to a constant which we also denote by $0<\lambda<\frac{2\pi}{\left|\partial\Omega\right|}$.

Then, it is readily seen that the $L^1$-condition \eqref{non convex condition} is implied by the stricter $L^2$-condition
\begin{equation*}
	\sup_{y\in\partial\Omega}\int_{\partial\Omega} \left( \frac{x-y}{\left|x-y\right|^2}\cdot n(x)-\lambda \right)^2 dx
	<\frac{\pi^2}{\left|\partial\Omega\right|},
\end{equation*}
whose left-hand side is quadratic in $\lambda$ and therefore minimized, in view of \eqref{A1}, by the value $\lambda=\frac{\pi}{\left|\partial\Omega\right|}$. It follows that \eqref{non convex condition} holds with $\lambda(x)\equiv \frac{\pi}{\left|\partial\Omega\right|}$ provided
\begin{equation*}
	\sup_{y\in\partial\Omega}\int_{\partial\Omega} \left( \frac{x-y}{\left|x-y\right|^2}\cdot n(x) \right)^2 dx
	<\frac{2\pi^2}{\left|\partial\Omega\right|},
\end{equation*}
or, even more stringently,
\begin{equation}\label{non convex condition 2}
	\sup_{x,y\in\partial\Omega} \left| \frac{x-y}{\left|x-y\right|^2}\cdot n(x) \right|
	< \sqrt{2}\frac{\pi}{\left|\partial\Omega\right|}.
\end{equation}

Notice that $\frac{x-y}{\left|x-y\right|^2}\cdot n(x)=\frac{\pi}{\left|\partial\Omega\right|}$, for all $x,y\in\partial\Omega$, if $\partial\Omega$ is a circle. Therefore, we may interpret the preceding conditions with $\lambda(x)\equiv\frac{\pi}{\left|\partial\Omega\right|}$ as a requirement that $\partial\Omega$ does not deviate too much from a circle of equal circumference.

More precisely, for any $R\in\mathbb{R}\setminus\left\{0\right\}$, one easily verifies that the constraint
\begin{equation}\label{constraint 1}
	\frac{x-y}{\left|x-y\right|^2}\cdot n(x)=\frac{1}{2R},
\end{equation}
is equivalent to the relation
\begin{equation}\label{constraint 2}
	\left|y-\left(x-Rn(x)\right)\right|=|R|.
\end{equation}
Recalling that $\frac{x-y}{\left|x-y\right|^2}\cdot n(x)$, for each fixed $y\in\partial\Omega$, has an average value of $\frac\pi{\left|\partial\Omega\right|}$ over $x\in\partial\Omega$, we further introduce $R_{\rm sup}\in \left(0,\frac{\left|\partial\Omega\right|}{2\pi}\right]$ and $R_{\rm inf}\in \left(-\infty,0\right)\cup\left[\frac{\left|\partial\Omega\right|}{2\pi},\infty\right]$ defined by
	\begin{align*}
		\sup_{x,y\in\partial\Omega} \left(\frac{x-y}{\left|x-y\right|^2}\cdot n(x)\right) & =\frac 1{2R_{\rm sup}},
		\\
		\inf_{x,y\in\partial\Omega} \left(\frac{x-y}{\left|x-y\right|^2}\cdot n(x)\right) & =\frac 1{2R_{\rm inf}}.
	\end{align*}
Observe that $R_{\rm inf}\in \left[\frac{\left|\partial\Omega\right|}{2\pi},\infty\right]$ if $\overline\Omega^c$ is convex, whereas $R_{\rm inf}\in \left(-\infty,0\right)$ if $\overline\Omega^c$ is non-convex. In view of the equivalence between between \eqref{constraint 1} and \eqref{constraint 2}, we have the following properties:
\begin{itemize}
	\item $R_{\rm sup}$ is the largest radius $R>0$ such that, for each $x\in\partial\Omega$, the domain $\overline{\Omega}^c$ contains an open ball of radius $R$ tangent to $\partial\Omega$ at $x$,
	\item if $\overline\Omega^c$ is convex, $R_{\rm inf}$ is the smallest radius $R>0$ such that, for each $x\in\partial\Omega$, the domain $\overline{\Omega}^c$ is contained in an open ball of radius $R$ tangent to $\partial\Omega$ at $x$,
	\item if $\overline\Omega^c$ is non-convex, $R_{\rm inf}$ is negative and $|R_{\rm inf}|$ is the largest radius $R>0$ such that, for each $x\in\partial\Omega$, the exterior domain $\Omega$ contains an open ball of radius $R$ tangent to $\partial\Omega$ at $x$.
\end{itemize}

Thus, we arrive at the following geometric interpretation: condition \eqref{non convex condition 2} holds if and only if there exists $\frac{\left|\partial\Omega\right|}{2\pi\sqrt{2}}<R\leq \frac{\left|\partial\Omega\right|}{2\pi}$ such that, for each $x\in\partial\Omega$, there are two open balls of radius $R$, one contained in $\Omega$ and the other in $\overline{\Omega}^c$, both tangent to $\partial\Omega$ at $x$ (note that $R>\frac{\left|\partial\Omega\right|}{2\pi}$ is not possible by the isoperimetric inequality). This criterion includes a large variety of non-convex geometries.

\bigskip

The preceding analysis can also offer a more general geometric interpretation of \eqref{non convex condition} with non-constant parameters $\lambda(x)$. For instance, considering the following reasonable choice of parameter
\begin{equation*}
	\lambda(x)=(1-\sigma)\sup_{y\in\partial\Omega} \left(\frac{x-y}{\left|x-y\right|^2}\cdot n(x)\right)
	+\sigma
	\inf_{y\in\partial\Omega} \left(\frac{x-y}{\left|x-y\right|^2}\cdot n(x)\right),
\end{equation*}
for some given $0\leq\sigma\leq 1$, it is readily seen that \eqref{non convex condition} holds provided that
	\begin{align*}
		(1-\sigma)
		\left(\int_{\partial\Omega}
		\sup_{y\in\partial\Omega} \left(\frac{x-y}{\left|x-y\right|^2}\cdot n(x)\right)
		dx
		-\pi\right)
		\hspace{-20mm}&
		\\
		& +\sigma\left(\pi
		-\int_{\partial\Omega}
		\inf_{y\in\partial\Omega} \left(\frac{x-y}{\left|x-y\right|^2}\cdot n(x)\right)
		dx\right)
		<\pi,
	\end{align*}
or, even more stringently,
\begin{equation}\label{non convex condition 3}
	(1-\sigma)\left(\frac 1{2R_{\rm sup}}-\frac{\pi}{\left|\partial\Omega\right|}\right)
	+\sigma\left(\frac{\pi}{\left|\partial\Omega\right|}-\frac 1{2R_{\rm inf}}\right)
	<\frac{\pi}{\left|\partial\Omega\right|}.
\end{equation}

Note, then, that setting $\sigma=1$ reduces the above condition to the simple requirement that $\partial\Omega$ be strictly convex, i.e.\ $R_{\rm inf}>0$, whereas the value $\sigma=0$ yields the new criterion
\begin{equation}\label{non convex condition 4}
	R_{\rm sup}> \frac{\left|\partial\Omega\right|}{4\pi},
\end{equation}
which is much less restrictive than \eqref{non convex condition 2}. The geometric interpretation of \eqref{non convex condition 4} is as follows: condition \eqref{non convex condition 4} holds if and only if there exists $\frac{\left|\partial\Omega\right|}{4\pi}<R\leq \frac{\left|\partial\Omega\right|}{2\pi}$ such that, for each $x\in\partial\Omega$, there is an open ball of radius $R$ contained in $\overline{\Omega}^c$ and tangent to $\partial\Omega$ at $x$.

Finally, other values $0<\sigma<1$ in \eqref{non convex condition 3} can be loosely interpreted as an interpolation of the geometric conditions for $\sigma=0$ and $\sigma=1$.

\subsection{Dynamic convergence of the fluid charge approximation}

We end this section on the fluid charge method by providing dynamic theorems which are analog to Theorems \ref{wellposedness} and \ref{main convergence} on the vortex method.

To this end, for any prescribed smooth $\lambda$ such that $\ip{\lambda}\neq 2\pi$ and for sufficiently large integers $N$ (at least as large as $N_0$ determined by Theorem \ref{main theo alt alt} so that \eqref{point vortex alt alt} is invertible; see also Proposition \ref{inverse D alt alt}), we consider the approximate system
\begin{equation}\label{dynamic alt}
	\left\{
	\begin{aligned}
		& \partial_{t} \check{\check \omega}^N + \check{\check u}^N\cdot \nabla \check{\check \omega}^N =0,
		\\
		& \check{\check \omega}^N(t=0) = \omega_0,
	\end{aligned}
	\right.
\end{equation}
for some initial data $\omega_0\in C_c^1\left(\Omega\right)$ extended by zero outside $\Omega$ and with a velocity flow
\begin{equation*}
	\check{\check u}^N=K_{\mathbb{R}^2}\left[\check{\check\omega}^N\right]+\check{\check u}_{\rm app}^N\left[\check{\check\omega}^N,\gamma\right],
\end{equation*}
where $\check{\check u}_{\rm app}^N\left[\check{\check\omega}^N,\gamma\right]$ is given by \eqref{approx alt alt}-\eqref{point vortex alt alt}, for some prescribed $\gamma\in\mathbb{R}$ and where $u_P$ in the right-hand side of \eqref{point vortex alt alt} is now $K_{\mathbb{R}^2}\left[\check{\check\omega}^N\right]$.

By repeating the arguments leading up to Theorem \ref{wellposedness}, we deduce the following corresponding result. Its proof only requires slight adaptations from the proof of Theorem \ref{wellposedness} and so we omit it.

\begin{theorem}\label{wellposedness alt}
	Let $\omega_0\in C_c^1\left(\Omega\right)$, $\gamma\in\mathbb{R}$, a smooth function $\lambda$ be such that $\ip{\lambda}\neq 2\pi$ and consider any fixed time $t_1>0$. Then, for a well-$*$ distributed mesh on $\partial \Omega$, there exists $N_1\geq N_0$ ($N_0$ is determined in Theorem \ref{main theo alt alt}) such that, for any $N\geq N_1$, there is a unique classical solution $\check{\check\omega}^N\in C^1_c\left([0,t_1]\times \Omega\right)$ to \eqref{dynamic alt}. Moreover, the sequence of solutions $\left\{\check{\check\omega}^N\right\}_{N\geq N_1}$ is uniformly bounded in $C^1_c\left([0,t_1]\times \Omega\right)$.
\end{theorem}

The following theorem establishes the convergence of system \eqref{dynamic alt} towards system \eqref{dynamic1} as $N\to\infty$. Its proof is similar to the justification of Theorem \ref{main convergence} and so we leave it to the reader.

\begin{theorem}\label{main convergence alt}
	Let $\omega_0\in C_c^1\left(\Omega\right)$, $\gamma\in\mathbb{R}$, a smooth function $\lambda$ be such that $\ip{\lambda}\neq 2\pi$ and consider any fixed time $t_1>0$. Then, for a well-$*$ distributed mesh on $\partial\Omega$, as $N\to\infty$, the unique classical solution $\check{\check\omega}^N\in C^1_c\left([0,t_1]\times \Omega\right)$ to \eqref{dynamic alt} converges uniformly towards the unique classical solution $\omega\in C^1_c\left([0,t_1]\times \Omega\right)$ to \eqref{dynamic1}. More precisely, it holds that
	\begin{equation*}
		\left\|\omega-\check{\check \omega}^N\right\|_{L^\infty([0,t_1]\times\Omega)}=\mathcal{O}\left(N^{-\kappa}\right).
	\end{equation*}
\end{theorem}

\appendix

\section{Convergence rates of Riemann sums}\label{riemann appendix}

The following elementary lemma is a reminder about standard estimates on the rate of convergence of Riemann sums.

\begin{lemma}\label{riemann}
	Consider the uniformly distributed mesh $(\theta_{1}^N,\dots , \theta_{N}^N)\in \left[0,\left|\partial\Omega\right|\right)^N$, $(\tilde \theta_{1}^N, \dots , \tilde \theta_{N}^N)\in \left[0,\left|\partial\Omega\right|\right)^N$ defined by \eqref{mesh} and let $g$ be a smooth periodic function on $\left[0,\left|\partial\Omega\right|\right]$.
	
		Then, for any $\kappa \geq 2$ and $N\geq 2$
	\begin{equation*}
		\left|\int_0^{\left|\partial\Omega\right|}g(\theta){d}\theta - \frac{\left|\partial\Omega\right|}{N} \sum_{i=1}^{N} g(\tilde \theta_{i}^N) \right|\leq \frac{\pi^2 | \partial\Omega|^{\kappa +1}}{3(2\pi N)^{\kappa }}\|g^{(\kappa )} \|_{L^\infty}.
	\end{equation*}
\end{lemma}

\begin{proof}
As $g$ is of class $C^1$, its Fourier series converges uniformly to $g$:
\[
g (x)= \sum_{n=-\infty}^{+\infty} c_{n}(g) e^{i 2\pi nx/| \partial\Omega|}, \quad \text{where } c_{n}(g) = \frac{1}{| \partial\Omega|} \int_{0}^{| \partial\Omega|} g(t) e^{-i 2\pi n t/| \partial\Omega| }\, dt.
\]
Hence we compute for any $N$
\begin{align*}
  \frac{| \partial\Omega|}{N} \sum_{j=1}^{N}  g(\tilde \theta_{j}^N)
  =&  \frac{| \partial\Omega|}{N} \sum_{j=1}^{N}   \sum_{n=-\infty}^{+\infty} c_{n}(g) e^{i 2\pi n \tilde \theta_{j}^N/| \partial\Omega| }\\
  =& | \partial\Omega|c_{0}(g) +\frac{| \partial\Omega|}{N} \sum_{n\neq 0}c_{n}(g) e^{i\pi n  /N}  \sum_{j=0}^{N-1}  (e^{i2\pi n  /N})^{j}\\
    =&| \partial\Omega| c_{0}(g) +| \partial\Omega|\sum_{n\neq 0}c_{n N}(g)(-1)^n .
\end{align*}
By $\kappa $ integrations by parts in the definition of $c_{nN}(g)$, we get for any $\kappa \geq 2$
\begin{align*}
		\left|\int_0^{| \partial\Omega|}g(\theta){d}\theta - \frac{| \partial\Omega|}{N} \sum_{i=1}^{N} g(\tilde \theta_{i}^N) \right|
		& \leq 
		  \frac{| \partial\Omega|^{\kappa +1}}{ (2\pi N)^\kappa } \| g^{(\kappa )} \|_{L^\infty} \sum_{n\neq 0} \frac1{|n|^\kappa }
		  \\
		& \leq 
		  \frac{ | \partial\Omega|^{\kappa +1}}{(2\pi N)^\kappa } \| g^{(\kappa )} \|_{L^\infty} \frac{\pi^2}{3}
\end{align*}
	which ends the proof of the lemma.
\end{proof}

The preceding lemma can also be easily adapted to more general meshes, which is the content of the following result.

\begin{corollary}\label{riemann2}
	For any $N\geq 2$, consider a mesh $(\tilde s_{1}^N,\dots , \tilde s_{N}^N)\in \left[0,\left|\partial\Omega\right|\right)^N$ satisfying
	\begin{equation*}
		\max_{i=1,\ldots,N}\left|\tilde s_i^N-\tilde \theta_i^N\right|=\mathcal{O}\left(N^{-\kappa }\right),
	\end{equation*}
	for some $\kappa \geq 2$ and let $g$ be a smooth periodic function on $\left[0,\left|\partial\Omega\right|\right]$.
	
	Then, there exists a constant $C>0$ depending only on $\kappa $ and
	\begin{equation*}
		\sup_{N\geq 2} \max_{i=1,\ldots,N}N^\kappa  \left|\tilde s_i^N-\tilde \theta_i^N\right|
	\end{equation*}
	such that
	\begin{equation*}
		\left|\int_0^{\left|\partial\Omega\right|}g(s){d}s - \frac{\left|\partial\Omega\right|}{N} \sum_{i=1}^{N} g(\tilde s_{i}^N) \right|\leq \frac{C}{N^{\kappa }}\|g \|_{C^{\kappa }},
	\end{equation*}
	for any $N\geq 2$.
\end{corollary}

\begin{proof}
	By Lemma \ref{riemann}, it is readily seen that
		\begin{align*}
			\left|\int_0^{\left|\partial\Omega\right|}g(s){d}s - \frac{\left|\partial\Omega\right|}{N} \sum_{i=1}^{N} g(\tilde s_{i}^N) \right|
			\hspace{-20mm} &
			\\
			& \leq
			\left|\int_0^{\left|\partial\Omega\right|}g(\theta){d}\theta - \frac{\left|\partial\Omega\right|}{N} \sum_{i=1}^{N} g(\tilde \theta_{i}^N) \right|
			+
			\frac{\left|\partial\Omega\right|}{N} \sum_{i=1}^{N}\left| g(\tilde \theta_{i}^N) - g(\tilde s_{i}^N)\right|
			\\
			& \leq
			\frac{C}{N^{\kappa }}\|g \|_{C^{\kappa }}
			+
			\left|\partial\Omega\right|
			\|g \|_{C^{1}}
			\left(\max_{i=1,\ldots,N}\left|\tilde s_i^N-\tilde \theta_i^N\right|\right)
			\\
			& \leq
			\frac{C}{N^{\kappa }}\|g \|_{C^{\kappa }},
		\end{align*}
which concludes the proof.
\end{proof}

\bigskip

\noindent {\bf Acknowledgements.} The authors were partially supported by the project \emph{Instabilities in Hydrodynamics} funded by the Paris city hall (program \emph{Emergences}) and the \emph{Fondation Sciences Math\'ematiques de Paris}. E.D.\ and C.L.\ are partially supported by the Agence Nationale de la Recherche, Project SINGFLOWS, grant ANR-18-CE40-0027-01. C.L.\ is partially supported by the Agence Nationale de la Recherche, Project IFSMACS, grant ANR-15-CE40-0010.

\bibliographystyle{abbrv}
\bibliography{VM}

\end{document}